\documentclass[twocolumn]{autart}    

\usepackage[a4paper, total={7.25in, 9.6in}]{geometry}

\usepackage{amsmath,amssymb,amsfonts}
\usepackage[scr = dutchcal]{mathalpha}
\usepackage{graphicx}
\usepackage{textcomp}
\usepackage{stmaryrd}
\usepackage{tabularx}
\usepackage{multicol}
\usepackage{multirow}
\usepackage{float}
\usepackage{mathtools}
\usepackage{comment}
\usepackage{bm}		
\usepackage{enumitem}

\usepackage{wrapfig} 

\usepackage{color}

\AtBeginDocument{
\setlength{\abovedisplayskip}{5pt} 
\setlength{\belowdisplayskip}{5pt} 
\setlength{\abovedisplayshortskip}{0pt} 
\setlength{\belowdisplayshortskip}{2.5pt} 
}

\newcommand{\mbf}[1]{\mathbf{ #1}}
\newcommand{\tnf}[1]{\textnormal{#1}}
\newcommand{\tbf}[1]{\textbf{#1}}
\newcommand{\mbs}[1]{\boldsymbol{#1}}

\newcommand{\mcl}[1]{\mathcal{#1}}
\newcommand{\mscr}[1]{\mathscr{#1}}

\newcommand{\R}{\mathbb{R}}
\newcommand{\N}{\mathbb{N}}

\newcommand{\norm}[1]{\left\lVert{#1}\right\rVert}
\newcommand{\ip}[2]{\left\langle #1, #2 \right\rangle}

\newcommand{\bmat}[1]{\begin{bmatrix}#1\end{bmatrix}}

\newcommand{\smallbmat}[1]{\left[\scriptsize\begin{smallmatrix}
		#1\end{smallmatrix} \right]}
\newcommand{\mat}[1]{\begin{matrix}#1\end{matrix}}

\newcommand{\lmat}[1]{\begin{array}{l}#1\end{array}}

\newtheorem{example}[thm]{Example}
\newenvironment{proof}{%
	\ \\[-3.5\baselineskip]\begin{pf}}{%
		\hfill $\blacksquare$\end{pf} \ \\[-3\baselineskip]}

\floatstyle{ruled}
\newfloat{block}{thbp}{lop}
\floatname{block}{Block}

\let\bl\bigl
\let\bbl\Bigl
\let\bbbl\biggl
\let\bbbbl\Biggl
\let\br\bigr
\let\bbr\Bigr
\let\bbbr\biggr
\let\bbbbr\Biggr

\newcommand{\adj}[0]{*}		

\usepackage{upgreek}
\newcommand{\PI}[0]{\Uppi}
\newcommand{\PIset}[0]{\bm{\Pi}}
\newcommand{\PIparam}[0]{\bm{\Gamma}}

\newcommand\Resize[2]{\resizebox{#1}{!}{\mbox{\ensuremath{\displaystyle #2}}}}
\newcommand\scalemath[2]{\scalebox{#1}{\mbox{\ensuremath{\displaystyle #2}}}}

\begin{document}

\begin{frontmatter}
	
	\title{A State-Space Representation of Linear Multivariate PDEs and Stability Analysis using SDP\thanksref{footnoteinfo}
	} 
	
	\thanks[footnoteinfo]{\tbf{Acknowledgement:} This work was supported by the National Science Foundation grants 2337751 and 2429973.}
	
	\author[Tempe]{Declan S. Jagt}\ead{djagt@asu.edu},    
	\author[Tempe]{Matthew M. Peet}\ead{mpeet@asu.edu},               
	
	\address[Tempe]{Arizona State University, SEMTE, P.O. Box 876106, Tempe, AZ 85287-6106,	USA}


	\begin{abstract}
		\ \\[-0.25\baselineskip]
		\noindent Recently, it has been shown that stability analysis and control of coupled Partial Differential Equations (PDEs) in a single spatial variable can be more conveniently performed using the Partial Integral Equation (PIE) representation. This PIE offers an equivalent, state-space representation of the PDE on the Hilbert space $L_{2}$, and is parameterized by an algebra of Partial Integral (PI) operators, allowing e.g. stability to be analyzed by solving a linear operator inequality on PI operator variables. In this paper, we show how this PIE framework for univariate PDEs can be extended inductively to multivariate PDEs on a hyper-rectangle. 
		Specifically, assuming the boundary conditions defining the domain of the PDE to be decoupled along distinct spatial directions, we propose a readily verifiable condition for existence of a bijection between the PDE domain and $L_{2}$.
		We derive an explicit expression for the map defining this bijection, and use this map to construct an equivalent PIE representation for a broad class of linear multivariate PDEs. Next, we embed the parameters defining this PIE representation in a class of multivariate PI operators, and prove that this class forms a $*$-algebra---allowing PI operator inequalities for stability analysis, estimation, and control of univariate PDEs to be similarly formulated for multivariate PDEs.
		Finally, we show how such an operator inequality for stability analysis of multivariate PDEs can be solved with semidefinite programming, using a positive matrix parameterization of positive semidefinite multivariate PI operators. This framework for representation and stability analysis of multivariate PDEs is incorporated in the PIETOOLS software, and applied to analyze stability of 2D heat, wave, and plate equations, obtaining accurate bounds on the rate of decay.\\[-0.5\baselineskip]
	\end{abstract}
	
\end{frontmatter}

\section{Introduction}\label{sec:introduction}

Partial Differential Equations (PDEs) are frequently used to model physical phenomena such as fluid flow (e.g. Navier-Stokes), vibrations in flexible structures (e.g. Kirchhoff--Love), and population growth (e.g. Fisher's equation). Such PDE models provide a simplified representation of the phenomena, which can then be used to analyze stability properties or design a stabilizing controller. 
However, while there exist well-defined state-space and transfer function based algorithms for simulation, stability analysis and control of Ordinary Differential Equations (ODEs), applicable for any suitably well-posed problem, these methods do not readily generalize to PDEs. In particular, development of similar universal tools for PDEs is complicated by the fact that the state of the PDE is infinite-dimensional, and is further required to satisfy a set of boundary conditions.
Moreover, PDEs can involve an arbitrary number of spatial variables, and tools developed for univariate PDEs may not readily extend to a multivariate setting. As such, most work on analysis and control of (multivariate) PDEs is ad hoc, tailoring results to a limited class of PDEs and boundary conditions. An overview of such methods is as follows. \\[-0.75\baselineskip]

\subsection{Methods of Analysis and Control of PDEs}

Semigroup theory is a mathematical framework for analysis of PDEs, with associated notions of solution and well-posedness being defined in terms of existence of a strongly continuous semigroup~\cite{curtain1995introduction,bensoussan2007representation}. In this framework, operator inequality conditions may be proposed for analysis and control purposes~\cite{lasiecka2000PDEControl,lasiecka1991controlKirchoff,mahawattege2021fluid}. 
However, because the semigroup framework does not provide explicit methods for finding a solution or solving operator inequalities, PDE problems expressed in this framework are typically discretized through projection of the state onto a finite set of basis functions. Discretization of the dynamics is called ``early-lumping''\footnote{Examples of early-lumping can be found in~\cite{gayme2011amplification,liu2021io_analysis_flow,chernyshenko2014SOS_NS,fuentes2022SOS_NS}.} and discretization of operator inequalities is termed ``late-lumping''\footnote{Examples of late-lumping can be found in~\cite{moghadam2013boundary,riesmeier2022LateLumping}.}. However, in each case, care must be taken in selecting a finite-dimensional basis to ensure compatibility with boundary conditions. Furthermore, inferring properties of the PDE from properties of the discretization requires extensive, ad hoc, analysis.

Another PDE framework is the backstepping methodology~\cite{vazquez2024Backstepping}, often used to design stabilizing boundary feedback controllers. The discretization step in this case lies in numerically solving certain Volterra type integral equations for reconstruction of the backstepping state transformation. The resulting kernel may then be used to define an invertible state transformation and stabilizing controller gain~\cite{krstic2007backstepping}. However, defining the backstepping transformation for PDEs with multiple spatial variables requires an extension of the Volterra framework, and integration of function-valued boundary conditions. Consequently, such multivariate methods typically require use of spatial invariance or symmetry of solutions to reduce the dimensionality of the problem~\cite{vazquez2008BacksteppingNS,xu2008Backstepping2D,vazquez2019boundary}.

Finally, spectral discretization by wave number can be used in combination with transfer-function approaches to perform energy gain analysis. This method has been applied to linearized fluid flow in multiple spatial variables~\cite{jovanovic2005Energy_Amplification_NS,lieu2013L2_gain_Couette_flow,jovanovic2021io_flow_control}, yielding decomposition into streamwise and spanwise components. This results in an operator-valued transfer function which is then discretized.

Methods which avoid discretization are often based on construction of a suitable metric---i.e. a Lyapunov functional. The challenge in this case is parameterization of quadratic Lyapunov functionals and verifying negativity of the derivative along solutions to the PDE (often through the use of Linear Matrix Inequalities (LMIs)). In most cases, the Lyapunov functional is chosen to be the $L_2$ norm of the state, and negativity of the derivative of this function is established using inequalities such as Poincar\'e, Wirtinger, Jensen, etc., together with integration by parts. However, restrictions on the structure of the Lyapunov functional and the need for ad hoc analysis limits the scope of such methods and implies conservatism in the resulting stability conditions. This approach has been applied to design observers for the 2D Navier-Stokes equations~\cite{kang2019ObserverNS,zayats2021ObserverNS,zhuk2023ObserverNS} and controllers for the Kuramoto-Sivashinsky equation~\cite{kang2021ControlKSE}. 
Works which extend these approaches to Lyapunov functionals defined by weighted $L_2$ norms using Sum-of-Squares~\cite{parrilo2000SOS} for parabolic PDEs 
include~\cite{papachristodoulou2006AnalysisPDEs,meyer2015StabilityPDEs} (linear PDEs) and~\cite{valmorbida2015IntegralInequalities2D,ahmadi2016AnalysisPDEs} (nonlinear PDEs). Applications to fluid flow include~\cite{yu2008StabilityFlow2D,ahmadi2019AnalysisFlow}. 
Moment-based dual approaches can be found in~\cite{magron2020MomentsControl,korda2022Moments,henrion2023MomentsPDEs}. 
In all cases, however, the need for ad hoc analysis and the use of restricted classes of Lyapunov functionals limits accuracy of the results~\cite{fantuzzi2022sMomentSharpness}.

While this brief survey on methods for analysis and control of multivariate PDEs is clearly incomplete, in each case we find that the need for ad hoc analysis
limits the accuracy, scope and reliability of the method. To address this problem, we will examine the question of whether there exists some more fundamental state-space representation of PDEs which allows for the design of universal tools for simulation, analysis and control, which would be applicable to any PDE expressed in terms of this representation.

\subsection{The PIE Representation for Stability Analysis and Control of 1D PDEs}

The need for universal simulation, analysis and control methods for PDEs has motivated the development of a state-space framework for representation of these systems. Specifically, for PDEs in a single spatial variable, we have the Partial Integral Equation (PIE) representation~\cite{shivakumar2022GPDE_Arxiv}. This PIE representation defines the fundamental (or PIE) state as the highest-order spatial derivative ($\partial_s^{d}$) of the PDE state. This fundamental state lies in the Hilbert space $L_2^{n}$ and does not admit boundary conditions or Sobolev regularity constraints. 

Because the PIE state is unconstrained, establishing equivalence of PDE and PIE requires invertibility of $\partial_s^{d}:\mcl{D}\to L_{2}^{n}$ on the domain of the PDE, $\mcl{D}$, as specified by the boundary conditions and regularity constraints. For univariate coupled PDEs, necessary and sufficient conditions for invertibility of $\partial_s^{d}:\mcl{D}\to L_{2}^{n}$ have been proposed~\cite{gohberg1990LinearOperators},
and result in an analytic construction of this inverse as defined by a Volterra type integral operator. Specifically, this integral operator is of the form 
\begin{equation*}
	\bl(\mcl{T}\mbf{v}\br)(s)=\int_{a}^{s}\mbs{T}_{1}(s,\theta)\mbf{v}(\theta)\, d\theta +\int_{s}^{b}\mbs{T}_{2}(s,\theta)\mbf{v}(\theta)\, d\theta,
\end{equation*}
where $\mbs{T}_1,\mbs{T}_2$ are polynomials and can be computed analytically using the matrices which parameterize the boundary conditions associated with $\mcl{D}$. Linear integral operators of this form (denoted Partial Integral (PI) operators) are bounded and form a $*$-algebra, where $*$ denotes the $L_2$-adjoint. This algebraic structure is significant in that it implies that the evolution of the fundamental state may be represented as a PIE of the form $\partial_{t}\mcl{T}\mbf{v}=\mcl{A}\mbf{v}$ where $\mcl{A}$ is likewise a PI operator albeit with inclusion of a multiplier operator (inclusion of polynomial multipliers preserves the $*$-algebraic structure). 

Once a PIE representation has been established, there exists an array of algorithmic tools for simulation, analysis, control and estimation, 
most of which have been incorporated in the PIETOOLS software package~\cite{shivakumar2025PIETOOLS} (see also literature on PIE methods in 1D).

When PDEs involve multiple spatial variables, establishing an equivalent PIE representation raises new challenges. In this case, the fundamental state is similarly defined as the highest-order spatial derivative of the PDE state, e.g. $\mbf{v}=\partial_{x}^{d}\partial_{y}^{d}\mbf{u}$, and the goal (for a given PDE domain $\mcl{D}$) is to construct an inverse of $\partial_{x}^{d}\partial_{y}^{d}:\mcl{D}\to L_{2}^{n}$. The challenge lies in the boundary conditions which restrict the PDE state $\mbf{u} \in \mcl{D}$. Specifically, if the PDE is defined on an $N$-dimensional hyper-rectangle, the boundary conditions constrain the values of the PDE state on the $2N$ surfaces of the hyper-rectangle, each of which is itself an $(N-1)$-dimensional hyper-rectangle. Furthermore, unlike in the 1D case, the boundary values of the multivariate PDE state are coupled in ways which create new questions of well-posedness. For example, in the simplest case, the surfaces of the hyper-rectangle intersect at the corners and hence boundary conditions must be consistent at these corners (e.g. the boundary conditions must not imply $\mbf{u}(x,0)\neq \mbf{u}(0,y)$ at $(x,y)=(0,0)$). This question of consistency for 2D PDEs was explored in~\cite{jagt2022PIE_2D}, wherein a consistent basis for boundary conditions was proposed and it was shown that projection of the boundary conditions onto this basis was able to resolve the question of consistency. The disadvantage of this approach, however, is the need for projection of the boundary conditions and for inversion of multivariate integral operators to obtain the PIE representation.

In contrast to~\cite{jagt2022PIE_2D}, this paper considers a restricted set of boundary conditions defined as the intersection of lifted 1D domains---e.g. $\mcl{D}:=\mcl{D}_x \cap \mcl{D}_y$ where $\mcl{D}_y:=\{\mbf{u} \mid A_y\mbf{u}(x,0)+B_y\mbf{u}(x,1)=0\}$ and  $\mcl{D}_x:=\{\mbf{u} \mid A_x\mbf{u}(0,y)+B_x\mbf{u}(1,y)=0\}$. 
This approach includes many commonly used boundary conditions, and has the additional advantages that the parameters defining the boundary conditions may be expressed in a more intuitive manner, as well as avoiding the need for inversion of multivariate operators.
Moreover, using this approach, consistency of the boundary conditions may be readily expressed as a testable constraint on the parameters, e.g. $A_{x},A_{y},B_{x}$ and $B_{y}$, as will be shown in Section~\ref{sec:Tmap}.
Assuming this constraint to be satisfied, the inverse to the spatial differential operator may then be constructed inductively as, e.g., $(\partial_{x}\partial_{y})^{-1}=(\partial_{x})^{-1}(\partial_{y})^{-1}$---reducing the need for complicated multivariate notation and greatly simplifying the construction of the PIE representation.

The construction of the PIE representation for a class of linear, $N$D PDEs is presented in Section~\ref{sec:Tmap}. In Section~\ref{sec:PIs}, induction is then used to define a class of multivariate PI operators, which parameterize the PIE representation of multivariate PDEs. 
Using these PI operators, linear PI operator inequalities for e.g. stability analysis and optimal control of univariate PDEs may then be generalized to the multivariate setting, and we present such an operator inequality for stability analysis of multivariate PDEs in Section~\ref{sec:stability}. Finally, we show how this operator inequality can be solved using semidefinite programming, by parameterizing a class of positive semidefinite multivariate PI operators by positive semidefinite matrices. In Section~\ref{sec:Examples}, we use this approach to verify stability of 2D heat, wave, and plate equations. A more detailed overview of the organization of the paper is provided in Section~\ref{sec:outline}.

\section{Notation}

Let $\N$ denote the natural numbers and $\N_{0}:=\N\cup\{0\}$. For $i,j\in\N_{0}$, we denote a set of indices as $\{i:j\}:=\{i,i+1,\hdots,j\}$. 
For $N\in\N$ and multi-indices $\alpha,\beta\in\N_{0}^{N}$, we write $\alpha\leq \beta$ if $\alpha_{i}\leq \beta_{i}$ for all $i\in\{1:N\}$. 
We denote $\vec{0}:=(0,\hdots,0)\in\N^{N}$, and $\vec{2}:=(2,\hdots,2)\in\N^{N}$.

Define $\Omega:=\prod_{i=1}^{N}[a_{i},b_{i}]:=[a_{1},b_{1}]\times\cdots\times[a_{N},b_{N}]$ for $[a_{i},b_{i}]\subset\R$.
Denote by $\R^{n}[s]$ the space of $\R^{n}$-valued polynomials in variables $s$. Let $L_{2}^{n}[\Omega]$ denote the Hilbert space of real-valued, square-integrable functions on $\Omega$, with standard inner product 
$\ip{\mbf{u}}{\mbf{v}}_{L_2}\!:=\int_{\Omega}\mbf{u}(s)^T\mbf{v}(s)ds$. 
For suitably differentiable $\mbf{u}\in L_2^{n}[\Omega]$, define
\begin{equation*}
	\partial^{k}_{z}\mbf{u}:=\frac{\partial^{k}}{\partial z^{k}}\mbf{u},	\qquad
	D^{\alpha}\mbf{u}:=\partial_{s_{1}}^{\alpha_{1}}\cdots\partial_{s_{N}}^{\alpha_{N}}\mbf{u}, 
\end{equation*}
where $k \in \N_{0}$ and $\alpha \in \N^{N}_{0}$.
For sufficiently regular $\mbf{u}\in L_{2}^{n}[\Omega]$, we denote limit values in one argument as $\mbf{u}(\theta_{i}):=\mbf{u}(s)|_{s_{i}=\theta_{i}}:=\mbf{u}(s_{1},\hdots,s_{i-1},\theta_{i},s_{i+1},\hdots,s_{N})$.

Let $\mcl{L}(L_{2}^{n},L_{2}^{m})$ denote the space of bounded linear operators from $L_{2}^{n}$ to $L_{2}^{m}$, with induced operator norm $\norm{.}_{\tnf{op}}$, and let $\mcl{L}(L_{2}^{n}):=\mcl{L}(L_{2}^{n},L_{2}^{n})$.
For $\mbs{K}\in L_2^{m\times n}[\Omega]$, define the multiplier operator $\tnf{M}[\mbs{K}]\in\mcl{L}(L_{2}^{n}[\Omega],L_{2}^{m}[\Omega])$ by $(\text{M}[\mbs{K}]\mbf{u})(s):=\mbs{K}(s)\mbf{u}(s)$.
We write $I_{n}\in\R^{n\times n}$ for the $n\times n$ identity matrix, and $0_{m\times n}\in\R^{m\times n}$ for a matrix of all zeros. For $A\in\R^{m\times n}$, we let $[A]_{i,j}$ denote the element in row $i$ and column $j$ of $A$.

\subsection{Sobolev Space with Bounded Mixed Derivatives}\label{subsec:notation_sobolev}

Traditionally, solutions to PDEs in multiple spatial dimensions are allowed to lie in the Sobolev Hilbert spaces $W_{2}^{d}$---implying $D^{\alpha}\mbf{u}(t)\in L_2$ for all $\alpha\in\N_{0}^{N}$ such that $\norm{\alpha}_1:=\sum_{i=1}^{N}\alpha_{i}\leq d$. The disadvantage of this approach is that solutions must be defined in the weak sense and that the boundary conditions need to be defined using trace operators. 
Recently, however, there has been a growing interest in multivariate PDE solutions which are required to exist in more restrictive spaces, requiring  $D^{\alpha}\mbf{u}(t)\in L_2$ for all $\alpha$ such that $\norm{\alpha}_{\infty}:=\max_{i=1}^{N}\alpha_{i}\leq d$. 
More generally, for $\delta\in\N_{0}^{N}$, we define the Sobolev space with dominating mixed smoothness of order $\delta$ on $\Omega$ as
\begin{equation*}
	\Resize{\linewidth}{S_{2}^{\delta,n}[\Omega]\!:=\!\bl\{\mbf{u}\in L_2^{n}[\Omega]\mid\! D^{\alpha}\mbf{u}\in L_2^{n}[\Omega],~\forall \alpha\in\N_{0}^{N}\!:\alpha\leq\delta\br\}.}
\end{equation*}
Spaces of this form have the advantage that if $\delta_{i}>0$ for all $i$, then $S_{2}^{\delta}$ can be embedded in a suitable space of H\"older continuous functions~\cite{abdulla2023SobolevEmbedding}---implying that elements admit H\"older continuous representatives and that boundary values are well-defined without the need for trace operators. 
In particular, for every $\alpha\leq\delta$ and $i\in\{1:N\}$, if $\alpha_{i}<\delta_{i}$, then $D^{\alpha}\mbf{u}|_{s_{i}=a_{i}}$ and $D^{\alpha}\mbf{u}|_{s_{i}=b_{i}}$ are well-defined for any $\mbf{u}\in S_{2}^{\delta}[\Omega]$.
In addition, using $S_{2}^{\delta}$ has the advantage that we may decompose
$S_{2}^{\delta}[\Omega]=\bigcap_{i=1}^{N}S_{2}^{\delta_{i}\cdot\tnf{e}_{i}}[\Omega]$,
 where $\tnf{e}_{i}\in\R^{N}$ is the $i$th standard Euclidean basis vector. 
Based on this decomposition, we will also consider PDE domains of the form $\mcl{D}:=\bigcap_{i=1}^{N}\mcl{D}_{i}\subseteq S_{2}^{\delta}$, where each $\mcl{D}_{i}\subseteq S_{2}^{\delta_{i}\cdot\tnf{e}_{i}}[\Omega]$ is a lifted univariate domain. Specifically, for $\Omega:=\prod_{i=1}^{N}[a_{i},b_{i}]$ and a univariate domain, $\mcl{D}\! \subseteq\! S_{2}^{d,n}[a_{i},\!b_{i}]$, we define the $i^{th}$ lifting of that domain as
\begin{align*}
	\hat{\mcl{D}}[i]&:=\bl\{\mbf u \in S_2^{d\cdot\tnf{e}_{i},n}[\Omega]\mid \tnf{a.e.}~ s_j \in [a_{j},b_{j}],\; j \neq i \;,	\\[-0.2em]
	&\hspace*{1.5cm}\mbf u(s_1,...,s_{i-1},\,\bullet\,,s_{i+1},...,s_{N}) \in \mcl D\br\},
\end{align*}
where for $\mbf{u} \in S_2^{d\cdot\tnf{e}_{i},n}[\Omega]$ and $s_j \in [a_{j},b_{j}]$ for $j \neq i$, $\mbf{v}=\mbf{u}(s_1,\hdots,s_{i-1},\,\bullet\,,s_{i+1},\hdots,s_{N})$ means that $\mbf{v}(s_i)=\mbf{u}(s)$ for all $s_i \in [a_{i},b_{i}]$. Then, for $\mcl D_i\subseteq S_{2}^{\delta_{i},n}[a_{i},b_{i}]$, through some abuse of notation, we let $\bigcap_{i=1}^{N}\mcl{D}_{i}:=\bigcap_{i=1}^{N}\hat{\mcl{D}}_{i}[i]$.
\begin{example}\label{ex:BCs_DN_intro}
	Consider the following 2D PDE domain,
	\begin{equation*}
		\mcl{D}:=\bbbl\{\mbf{u}\in S_{2}^{(2,2)}[[0,1]^2]\,\bbr|\ \lmat{\mbf{u}(0,y)=\mbf{u}(1,y)=0\\[-0.4em]\mbf{u}(x,0)=\mbf{u}_{y}(x,1)=0}\bbbr\}.
	\end{equation*}
	Then $\mcl{D}=\hat{\mcl{D}}_{1}[1]\cap\hat{\mcl{D}}_{2}[2]$, where
	\begin{align*}
		&\mcl{D}_{1}:=\{\mbf{u}\in S_{2}^{2}[0,1]\mid \mbf{u}(0)=\mbf{u}(1)=0\},	\\
& \mapsto \hat{\mcl{D}}_{1}[1]:=\{\mbf{u}\in S_{2}^{(2,0)}[[0,1]^2]\mid \mbf{u}(0,y)=\mbf{u}(1,y)=0\},\\
		&\mcl{D}_{2}:=\{\mbf{u}\in S_{2}^{2}[0,1]\mid \mbf{u}(0)=\mbf{u}_{y}(1)=0\},\\
		& \mapsto \hat{\mcl{D}}_{2}[2]:=\{\mbf{u}\in S_{2}^{(0,2)}[[0,1]^2]\mid \mbf{u}(x,0)=\mbf{u}_{y}(x,1)=0\}.
	\end{align*}%
\end{example}

\section{Problem Definition and Organization}\label{sec:outline}

Consider the problem of representation and stability analysis of a class of autonomous linear PDEs of the form
\begin{equation}\label{eq:PDE_standard_intro}
	\partial_{t}\mbf{u}(t,s)=\!\sum_{0\leq\alpha\leq\delta}\!\mbs{A}_{\alpha}(s)D^{\alpha}\mbf{u}(t,s),	\quad ~t\geq 0,~s\in\Omega,	\\[-0.4\baselineskip]	
\end{equation}
parameterized by $\mbs{A}_{\alpha}\in \R^{n\times n}[s]$, with $\Omega:=\prod_{i=1}^{N}[a_{i},b_{i}]$.
We suppose that the domain of the PDE may be decomposed as $\mbf{u}(t)\in \mcl{D}:=\bigcap_{i}\mcl{D}_{i}$, where for each $i$,
\begin{align}\label{eq:BCs_standard_intro}
	\mcl{D}_{i}:=\bbbl\{&\mbf{u}\in S_{2}^{\delta_{i}\cdot\tnf{e}_{i},n}[\Omega]~\bbbl|~ \forall j\in\{0:\delta_{i}-1\},\\[-0.6em]
	& \sum_{k=0}^{\delta_{i}-1}\bbl[B^{i}_{j,k}(\partial_{s_{i}}^{k}\mbf{u})|_{s_{i}=a_{i}} 
	+C^{i}_{j,k}(\partial_{s_{i}}^{k}\mbf{u})|_{s_{i}=b_{i}}\bbr]=0\bbbr\},\notag
\end{align}
with parameters $B^{i}_{j,k},C^{i}_{j,k}\in\R^{n\times n}$. 
Note that this representation assumes a hyperrectangular spatial geometry, $\Omega=\prod_{i=1}^{N}[a_{i},b_{i}]$, in order for the PDE domain $\mcl{D}$ to admit a decomposition into univariate domains $\mcl{D}_{i}$ along each interval $[a_{i},b_{i}]$. This decomposition will simplify the construction of multivariate PIE representations by allowing for $N$ inductive incorporations of univariate domains, as shown in Section~\ref{sec:Tmap}. Furthermore, by requiring the parameters $\mbs{A}_{\alpha}$ to be polynomial, the resulting PIE representation will be defined by polynomial parameters. This use of polynomial parameters will allow a stability test to be formulated as a semidefinite program, as shown in Section~\ref{sec:stability}. Although a PIE representation can be constructed for PDEs with non-polynomial parameters---thereby also supporting more complex geometries through the use of e.g. spherical coordinates---this complicates the formulation of an SDP-based stability test, and hence will not be addressed.

Despite the aforementioned limitations, the proposed class of PDEs and boundary conditions is quite general, including e.g., 2nd-order, 2D PDEs of the form\footnote{We may convert to the form in Eqn.~\eqref{eq:PDE_standard_intro} using $\mbs{A}_{(0,0)}=A_{0}$, $[\mbs{A}_{(1,0)},\mbs{A}_{(0,1)}]=A_{1}$ and $[\mbs{A}_{(2,0)},\mbs{A}_{(0,2)}]=A_{2}$.}
\begin{equation*}
	\mbf{u}_t=A_{0}\mbf{u} +A_{1}\scalemath{0.9}{\bmat{\mbf{u}_{x}\\\mbf{u}_{y}}} +A_{2}\scalemath{0.9}{\bmat{\mbf{u}_{xx}\\\mbf{u}_{yy}}},\quad \mbf{u}\in\mcl{D}=\mcl{D}_{1}\cap\mcl{D}_{2}, 
\end{equation*}
with boundary conditions parameterized as
\begin{equation*}
	\Resize{\linewidth}{
	\mcl{D}_{1}\!:=\!\bbbl\{\!\mbf u\;\bbbl|
	B_{x}\!\smallbmat{\mbf{u}(0,y)\\\mbf{u}_{x}(0,y) \\ \mbf{u}(1,y)\\\mbf{u}_{x}(1,y)}\!\!=\!0\bbbr\},~
	\mcl{D}_{2}\!:=\!\bbbl\{\!\mbf u\;\bbbl|
	B_{y}\!\smallbmat{\mbf{u}(x,0)\\\mbf{u}_{y}(x,0) \\ \mbf{u}(x,1)\\\mbf{u}_{y}(x,1)}\!\!=\!0\bbbr\},}
\end{equation*}
This formulation includes heat and wave\footnote{Let $A_{0}=\smallbmat{0&0\\1&0}$ and $A_{2}=\smallbmat{0&1&0&1\\0&0&0&0}$ and use $\mbf{u}(t)=\smallbmat{u(t)\\ u_{t}(t)}$.} equations with Dirichlet, Robin, and mixed boundary conditions, such as the following reaction-diffusion system.
\begin{example}\label{ex:BCs_DN_0}	
	As an example of the proposed class of PDEs, consider the following 2D reaction-diffusion equation,
	\begin{align*}
		\mbf{u}_{t}(t,x,y)&\!=\!\mbf{u}_{xx}(t,x,y)\!+\!\mbf{u}_{yy}(t,x,y)\!+\!r\mbf{u}(t,x,y),	\\
		\mbf{u}(t,0,y)&=\mbf{u}(t,1,y)=0,	\quad
		\mbf{u}(t,x,0)=\mbf{u}_{y}(t,x,1)=0,
	\end{align*}
	where $r\in\R$. 
	We represent this system in the standardized form in~\eqref{eq:PDE_standard_intro} using $\mcl{D}$ as in Example~\ref{ex:BCs_DN_intro} and $\mbs{A}_{(2,0)}=\mbs{A}_{(0,2)}=1$, $\mbs{A}_{(0,0)}=r$, and $\mbs{A}_{\alpha}=0$ for all other $\alpha\leq \delta$, with $\delta=(2,2)$.
\end{example}

Considering the class of PDEs in~\eqref{eq:PDE_standard_intro}, our first goal is to construct an equivalent PIE representation, exploiting the decomposition of the PDE domain as $\mcl{D}=\bigcap_{i=1}^{N}\mcl{D}_{i}$.

\subsection{PIE Representation and Fundamental State}

Having proposed a general class of linear, $N$D PDEs,
we now seek to represent solutions to such PDEs using evolution equations of the form
\begin{equation}\label{eq:PIE1}
	\partial_{t}\mcl{T}\mbf{v}(t)=\mcl{A}\mbf{v}(t),\qquad \mbf{v}(t)\in L_{2}^{n}[\Omega],~t\geq 0, 
\end{equation}
where $\mcl{T}=\prod_{i=1}^{N} \mcl T_i$ and $\mcl{A}= \sum_{j=1}^{M}\prod_{i=1}^{N} \mcl{A}_{i,j}$, for $\mcl{T}_{i}$ and $\mcl{A}_{i,j}$ univariate integral operators of the form
\begin{align}\label{eq:PI1}
	&(\mcl{R} \mbf{v})(s):=\mbs{R}_{0}(s_i)\mbf{v}(s)\\[-0.1em]
	&\quad+\int_{a_{i}}^{s_i} \mbs{R}_{1}(s_i,\theta_i)\mbf{v}(\theta_i)\, d\theta_i	
	+\int_{s_i}^{b_{i}}\mbs{R}_{2}(s_i,\theta_i) \mbf{v}(\theta_i)\, d\theta_i,	\notag
\end{align}
for polynomials $\mbs{R}_{k}$.
Operators of this form belong to the algebra of Partial Integral (PI) operators defined formally in Defn.~\ref{defn:PI_ND}, and evolution equations of the form in~\eqref{eq:PIE1} are referred to as Partial Integral Equations (PIEs). 
Solutions to the PDE and PIE are then related through the maps $\mbf{v}(t)=D^{\delta}\mbf{u}(t)$ and $\mbf{u}(t)=\mcl{T}\mbf{v}(t)$.

\subsection{Equivalence of PDE and PIE -- Overview of Section~\ref{sec:Tmap}}\label{subsec:overview}

While the map from PDE to PIE state, $D^{\delta}: \mcl{D}\to L_2^{n}$, is specified by the order $\delta$ of the PDE in~\eqref{eq:PDE_standard_intro}, construction of the map from PIE to PDE state, $\mcl{T}:=(D^{\delta})^{-1}:L_2^{n} \to \mcl D$, is nontrivial. Indeed, inversion of $D^{\delta}$ on the domain $\mcl{D}$ is the essential step in constructing the PIE representation---establishing a bijection between $L_{2}^{n}$ and $\mcl{D}$. 
We perform this inversion in Section~\ref{sec:Tmap}, exploiting the decomposition of the PDE domain as $\mcl{D}:=\mcl{D}_{1}\cap\cdots\cap\mcl{D}_{N}$ to inductively define $(D^{\delta})^{-1}=(\partial_{s_{N}}^{\delta_{N}})^{-1}\cdots(\partial_{s_{1}}^{\delta_{1}})^{-1}$. In particular, in Subsection~\ref{subsec:Tmap:1D}, we first recall that if each set $\mcl{D}_{i}$ is defined by well-posed boundary conditions of the form in~\eqref{eq:BCs_standard_intro}, then $\partial_{s_{i}}^{\delta_{i}}:\mcl{D}_{i}\to L_{2}^{n}$ is invertible, and the inverse $\mcl{T}_{i}:L_{2}^{n}\to \mcl{D}_{i}$ is of the form in~\eqref{eq:PI1}, with $\mbs{R}_{0}=0$ and where the formulae for $\mbs{R}_1,\mbs{R}_2$ are given in~\cite{gohberg1990LinearOperators}. It is then relatively easy to show through induction that $\mcl{T}:=\mcl{T}_{N}\cdots\mcl{T}_{1}$ is both a left- and right-inverse of $D^{\delta}:=\partial_{s_{1}}^{\delta_{1}}\cdots\partial_{s_{N}}^{\delta_{N}}$---i.e. $D^{\delta}\mcl{T}=I_{n}:L_{2}^{n}\to L_{2}^{n}$ and $\mcl{T}D^{\delta}=I_{n}:\mcl{D}\to\mcl{D}$. The following example illustrates this construction for a 2D domain with mixed boundary conditions.
\begin{example}\label{ex:BCs_DN_1}
	To illustrate the construction of an inverse to $D^{\delta}:\mcl{D}\to L_{2}$, consider $\mcl{D}=\mcl{D}_{1}\cap\mcl{D}_{2}$ as in Example~\ref{ex:BCs_DN_intro}.
	Then, we can define $\mcl{T}_{1}:=(\partial_{x}^{2})^{-1}:L_{2}[0,1]\to\mcl{D}_{1}$ 
	and $\mcl{T}_{2}:=(\partial_{y}^{2})^{-1}:L_{2}[0,1]\to\mcl{D}_{2}$ as\footnote{the definitions of these operators follow from Cor.~\ref{cor:Tmap_1D}.}
	\begin{align*}
		(\mcl{T}_{1}\mbf{v})(x)
		&:=\int_{0}^{x}\theta(x-1)\mbf{v}(\theta)\, d\theta +\int_{x}^{1}x(\theta-1)\mbf{v}(\theta)\, d\theta,	\\
		(\mcl{T}_{2}\mbf{v})(y)
		&:=-\int_{0}^{y}\eta\mbf{v}(\eta)\, d\eta -\int_{y}^{1}y\mbf{v}(\eta)\, d\eta.
	\end{align*}
	It follows that $\mcl{T}D^{(2,2)}=D^{(2,2)}\mcl{T}=I$, where
	\begin{align*}
		&(\mcl{T}\mbf{v})(x,y)
		=\int_{0}^{x}\int_{0}^{y}\theta(1-x)\eta \mbf{v}(\theta,\eta)\, d\eta d\theta	\\ &+\!\int_{0}^{x}\!\int_{y}^{1}\!\!\theta(1\!-\!x)y \mbf{v}(\theta,\eta)\, d\eta d\theta
		+\!\int_{x}^{1}\!\int_{0}^{y}\!\!x(1\!-\!\theta)\eta \mbf{v}(\theta,\eta)\, d\eta d\theta	\\
		&\quad+\int_{x}^{1}\!\int_{y}^{1}x(1-\theta)y \mbf{v}(\theta,\eta)\, d\eta d\theta.
	\end{align*}
\end{example}

Having defined $\mcl{T}$ such that $\mcl{T}D^{\delta}:\mcl{D}\to\mcl{D}$, for any $\mbf{u}(t)$ satisfying the PDE in~\eqref{eq:PDE_standard_intro}, $\mbf{v}(t)=D^{\delta} \mbf{u}(t)$ will satisfy the PIE in~\eqref{eq:PIE1} (letting $\mcl{A}=\sum_{\alpha \leq \delta} \tnf{M}[\mbs{A}_{\alpha}]\, D^{\alpha}\, \mcl{T}$), so that stability of the PIE implies stability of the PDE.
However, to show that if $\mbf{v}(t)$ satisfies the PIE, then $\mbf{u}(t)=\mcl{T}\mbf{v}(t)$ satisfies the PDE, we must also show that $\mcl{T}:L_{2}^{n}\to \mcl{D}$---i.e. that $\mcl{T}\mbf{v}$ satisfies the boundary conditions defining $\mcl{D}$ for all $\mbf{v}\in L_{2}^{n}$.

While for scalar-valued states (i.e. $n=1$), the definition of $\mcl{T}$ ensures $\mcl{T}\mbf{v}\in \mcl{D}$ for all $\mbf{v}\in L_{2}$, the extension to vector-valued states requires an additional restriction on the domains $\mcl{D}_{i}$ defining $\mcl{D}$. 
Specifically, in Subsection~\ref{subsec:Tmap:consistency}, we show that even in the 2D case, $\mcl{D}=\mcl{D}_{1}\cap\mcl{D}_{2}$, there exist domains $\mcl D_1,\mcl D_2$ (each well-posed in the 1D sense) and $\mbf{v}\in L_{2}^{n}[[0,1]^2]$ such that $\mcl{T}\mbf{v}\notin\mcl{D}$. This problem manifests as inconsistency of $\mcl D_{i}$ at the corners of the spatial domain. For illustration, in Example~\ref{ex:BCs_DN_intro}, replacing $\mbf{u}(0)=0$ by $\mbf{u}(0)=1$ in the definition of $\mcl{D}_{1}$, this would be incompatible with $\mcl{D}_{2}$, as $\mbf{u}\in\hat{\mcl{D}}_{1}[1]\cap\hat{\mcl{D}}_{2}[2]$ would then require both $\mbf{u}(0,0)=0$ and $\mbf{u}(0,0)=1$.
To avoid such incompatibility, we propose a necessary and sufficient condition on the parameters $B^{i}_{j,k}$ and $C^{i}_{j,k}$ which define the $\mcl{D}_i$ for these domains to be consistent (Defn.~\ref{defn:consistent_BCs}).
This consistency condition then guarantees that $\mcl{T}:L_{2}^{n}\to \mcl{D}$---implying that any solution to the PIE defines a solution to the PDE. 
This establishes that stability of the PIE is equivalent to stability of the PDE.

Having established the main result, in Section~\ref{sec:stability}, we then propose a simple Lyapunov condition for stability of the PIE (and hence PDE) in terms of a set of linear PI operator inequalities. These operator inequalities can be tested using semidefinite programming via software such as PIETOOLS and a brief description of this software is provided in Section~\ref{sec:Examples}, with application to several examples.

\section{An Equivalent PIE Representation of Linear Multivariate PDEs}\label{sec:Tmap}

In~\cite{shivakumar2022GPDE_Arxiv}, it was shown how for a univariate domain $\mcl{D}\subseteq  S_{2}^{d}[a,b]$ defined by a set of suitably admissible linear boundary conditions, we can define a left- and right-inverse $\mcl{T}:L_{2}[a,b]\to \mcl{D}$ to $\partial_{s}^{d}:\mcl{D}\to L_{2}[a,b]$.
Using this relation, an equivalent PIE representation can be derived for a broad class of linear, univariate PDEs.

As discussed in Subsection~\ref{subsec:overview}, our approach to representation of multivariate PDEs is based on inductive application of the univariate construction when the domain of the $N$D PDE may be decomposed as $\mcl{D}=\hat{\mcl{D}}_{1}[1]\cap\cdots\cap \hat{\mcl{D}}_{N}[N]$, where each $\mcl{D}_{i}\subseteq S_{2}^{\delta_{i}}[a_{i},b_{i}]$ is defined by boundary conditions only with respect to the $i$th spatial variable (as in~\eqref{eq:BCs_standard_intro}). When these domains are consistent, the map $\mcl T:L_{2}^{n}[\Omega]\to \mcl{D}$ may then be constructed as  $\mcl{T}:=\mcl{T}_{1}\cdots\mcl{T}_{N}$, where $\mcl{T}_{i}:=(\partial_{s_{i}}^{\delta_{i}})^{-1}:L_{2}^{n}[\Omega]\to \mcl{D}_{i}$. 
To illustrate, consider the special case of Dirichlet and Neumann boundary conditions, defining the following 1D domains:
\begin{align}\label{eq:PDE_dom_DN}
	\mcl{D}_{\tnf{D}}&:=\{\mbf{u} \in S_{2}^{2,n}[a_{i},b_{i}]\;\mid\; \mbf u(a_i)=\mbf u(b_i)=0\},		\\
	\mcl{D}_{\tnf{N}1}&:=\{\mbf{u} \in S_{2}^{2,n}[a_{i},b_{i}]\;\mid\; \mbf{u}(a_i)=\mbf{u}_{s}(b_i)=0\},	\notag\\
	\mcl{D}_{\tnf{N}2}&:=\{\mbf{u} \in S_{2}^{2,n}[a_{i},b_{i}]\;\mid\; \mbf{u}_{s}(a_i)=\mbf{u}(b_i)=0\}.		\notag 
\end{align}
Defining the associated $N$D domain $\mcl D=\bigcap_{i=1}^{N}\mcl{D}_i:=\bigcap_{i=1}^{N}\hat{\mcl{D}}_{i}[i]$ as in Subsection~\ref{subsec:notation_sobolev}, we then have the following result (a special case of the main result in Thm.~\ref{thm:PDE2PIE}).

\begin{cor}\label{cor:Tmap_DN}
	Let $\mcl{D} =\bigcap_{i=1}^{N} \mcl{D}_{i}$ where $\mcl{D}_{i}\in\{\mcl{D}_{\tnf{D}},\mcl{D}_{\tnf{N}1},\mcl{D}_{\tnf{N}2}\}$ for each $i\in\{1:N\}$, for $\mcl{D}_{\tnf{D}}$, $\mcl{D}_{\tnf{N}1}$, $\mcl{D}_{\tnf{N}2}$ as in~\eqref{eq:PDE_dom_DN}. Then $\mbf{u}=\mcl T D^{\vec{2}}\mbf{u}$ for all $\mbf{u} \in \mcl{D}$, and $\mbf{v}=D^{\vec{2}}\mcl{T} \mbf{v}$ and $\mcl{T} \mbf{v} \in \mcl{D}$ for all $\mbf{v} \in L_2^{n}$, where
	$(\mcl T\mbf{v})(s):=\int_{\Omega}\mbs{G}(s,\theta) \mbf{v}(\theta)d\theta$ with $\mbs{G}(s,\theta):=\prod_{i=1}^{N}\mbs{G}_{i}(s_{i},\theta_{i})$ for $\mbs{G}_{i}(x,y):=\bbl\{\!\scalemath{0.9}{\lmat{\mbs{h}_{i}(x,y)+(x-y),~ y\leq x,\\[-0.3em]\mbs{h}_{i}(x,y),\hspace*{1.63cm} y>x,}}$
	where
	\begin{equation*}
		\mbs{h}_{i}(x,y):=\begin{cases}
			\frac{-(x-a_{i})(b_{i}-y)}{b_{i}-a_{i}},	&	\mcl{D}_i={\mcl{D}}_{\tnf{D}},	\\
			a_{i}-x,	&	\mcl{D}_i={\mcl{D}}_{\tnf{N}1},	\\
			y-b_{i},	&	\mcl{D}_i={\mcl{D}}_{\tnf{N}2},
		\end{cases}
		\quad x,y\in[a_{i},b_{i}].
	\end{equation*}
	Moreover, $\mbf{u}(t)\in\mcl{D}$ satisfies the $N$D heat equation, $\partial_{t}\mbf{u}(t)=\nabla^2\mbf{u}(t)=\sum_{i=1}^{N}\partial_{s_{i}}^2\mbf{u}(t)$, if and only if $\mbf{v}(t)=D^{\vec{2}}\mbf{u}(t)\in L_{2}^{n}$ satisfies the PIE $\partial_{t} \mcl{T}\mbf{v}(t)=\mcl{A}\mbf{v}(t)$, where $\mcl{A}\mbf{v}(s):=\sum_{j=1}^{N}\prod_{i\neq j}(\int_{[a_{i},b_{i}]}\mbs{G}_{i}(s_{i},\theta_{i}))\mbf{v}(\theta)d\theta$.
\end{cor}

\textbf{Note:} For $\mcl{T}$ as in Cor.~\ref{cor:Tmap_DN}, we have $\mcl{T}=\mcl{T}_{1}\cdots\mcl{T}_{N}$, where $(\mcl{T}_i\mbf{v})(s):=\int_{a_{i}}^{b_{i}}\mbs{G}_{i}(s_i,\theta)\mbf u(s_1,...,s_{i-1},\theta,s_{i+1},...,s_N)d \theta$.

To establish results of this form, in Subsection~\ref{subsec:Tmap:1D}, we recall the construction of $\mcl{T}_i$ and the 1D PIE representation. Consistency of the domains, $\mcl{D}_{i}$, will be addressed in Subsection~\ref{subsec:Tmap:consistency}.
Finally, the construction of $\mcl{T}=\mcl{T}_{1}\cdots\mcl{T}_{N}$ and the $N$D PIE representation is presented in Subsection~\ref{subsec:Tmap:PDE2PIE}.

\subsection{The PIE Representation of 1D PDEs}\label{subsec:Tmap:1D}

Consider a linear 1D PDE of the form
\begin{equation*}
	\partial_{t}\mbf{u}(t,s)=(\mscr{H}\mbf{u})(t,s),\quad \mbf{u}(t)\in\mcl{D},~s\in[a,b],
\end{equation*}
where $(\mscr{H}\mbf{u})(t,s):=\sum_{k=0}^{d} \mbs{H}_{k}(s-a)\partial_{s}^{k}\mbf{u}(t,s)$, and 
where $\mcl{D}$ is the domain of the PDE, defined by a suitable set of linear boundary conditions.
The crucial step in constructing the PIE representation for such a PDE is that of inverting the operator $\partial_s^d : \mcl{D} \to L_2^{n}[a,b]$. A class of domains for which this can be achieved has already been proposed in~\cite{shivakumar2022GPDE_Arxiv}, Defn.~9, and we define a similar (slightly less general) class as follows.
\begin{defn}\label{defn:admissible_BCs_1D}
	We say that $\mcl{D}$ of the form
	\begin{align}\label{eq:PDEdom_1D}
		\mcl{D}:=&\bbbl\{\mbf{u}\in S_{2}^{d,n}[a,b]~\bbbl|~ \forall j\in\{0:d-1\},\\[-0.5em] 
		&\hspace*{1.5cm}\sum_{k=0}^{d-1}\bbl[B_{j,k}(\partial_{s}^{k}\mbf{u})(a) +C_{j,k}(\partial_{s}^{k}\mbf{u})(b)\bbr]=0\bbbr\},	\notag
	\end{align}
	is \tbf{admissible} if the matrix $(H_{a} +H_{b}\mbs{Q}(b-a))$ is invertible, where $H_{a},H_{b}\in\R^{nd\times nd}$ are defined by $[H_a]_{j,k}=B_{j,k}$, $[H_b]_{j,k}=C_{j,k}$, and where
	\begin{equation}\label{eq:Q}
		\mbs{Q}(z):=\scalemath{1}{\bmat{I_{n}&z I_{n}&\cdots&\frac{z^{d-1}}{(d-1)!}I_{n}\\[-0.2em]0_{n}&I_{n}&\ddots&\frac{z^{d-2}}{(d-2)!}I_{n}\\[-0.2em]\vdots&\vdots&\ddots&\vdots\\[-0.4em]0_{n}&0_{n}&\cdots&I_{n}}}.
	\end{equation}
\end{defn}

The matrix $\mbs{Q}(b-a)$ in Defn.~\ref{defn:admissible_BCs_1D} may be used to express the upper-boundary values, $\partial_{s}^{i}\mbf{u}(b)$, in terms of the lower-boundary values, $\partial_{s}^{j}\mbf{u}(a)$, relating the two through $\partial_{s}^{d}\mbf{u}$. Admissibility of the domain then implies that sufficient boundary conditions are specified to uniquely define each of the boundary values in terms of $\partial_{s}^{d}\mbf{u}$---thereby ensuring invertibility of $\partial_{s}^{d}:\mcl{D}\to L_{2}^{n}[a,b]$. The following result, a corollary of~\cite[Sec XIV.3, Thm~3.1]{gohberg1990LinearOperators}, provides an explicit expression for the resulting inverse $\mcl{T}: L_{2}^{n}\to \mcl{D}$, as well as for the composition of $\mcl{T}$ with the differential operator $\mscr{H}$\footnote{the original theorem from~\cite{gohberg1990LinearOperators} presents an inverse to the more general operator $\tau:=\sum_{i=0}^{d-1}\mbs{a}_{i}(s)\partial_{s}^{i}+\partial_{s}^{d}$ for $\mbs{a}_{i}\in L_{2}$, allowing for the construction of different PIE representations. However, we fix $\mbs{a}_{i}=0$ in order for $\mcl{T}=\tau^{-1}$ to be defined by polynomials.}.

\begin{cor}\label{cor:gohberg_PIE}
For given $a,b\in\R$, $d\in\N_{0}$ and $\mbs{H}_{k}\in \R^{n\times n}[s]$, let $\mcl{D}\subseteq S_{2}^{d,n}[a,b]$ be admissible as per Defn.~\ref{defn:admissible_BCs_1D} with $H_a,H_b$ as defined therein. Let $\mscr{H}:\mcl{D}\to L_{2}^{n}$ be defined by $(\mscr{H}\mbf{u})(s):=\sum_{k=0}^{d} \mbs{H}_{k}(s-a)\partial_{s}^{k}\mbf{u}(s)$. 
Then for any $\mbf{u}\in \mcl{D}$, $\mcl{T}\partial_{s}^{d}\mbf{u}=\mbf{u}$ and $\mscr{H}\mbf{u}=\mcl{A}\partial_{s}^{d}\mbf{u}$, and for any $\mbf{v} \in L_2^{n}[a,b]$, $\mcl{T}\mbf{v} \in\mcl{D}$, $\partial_{s}^{d}\mcl{T}\mbf{v}=\mbf{v}$, and $\mscr{H}\mcl{T}\mbf{v}=\mcl{A}\mbf{v}$, where for $d\neq 0$,
\begin{align*}
	&(\mcl{T}\mbf{v})(s):=\int_{a}^{b}  \mbs{G}_T(s,\theta)\mbf{v}(\theta)\,  d\theta,	\\
	&(\mcl{A}\mbf{v})(s):=\mbs{H}_{d}(s-a)\mbf{v}(s)+\int_{a}^{b}  \mbs{G}_A(s,\theta)\mbf{v}(\theta)\,  d\theta,
\end{align*}
where
\begin{align*}
	\mbs{G}_T(s,\theta)&:=
	\scalemath{0.95}{\begin{cases}
		\mbf e_1(s-a)^T(I_{n}-K)\mbf e_d(b-\theta)&\theta \leq s,\\[-0.1em]
		-\mbf e_1(s-a)^TK\mbf e_d(b-\theta)& s < \theta,
	\end{cases}}	\\
	\mbs{G}_A(s,\theta)&:=
	\scalemath{0.95}{\begin{cases}
		\mbf c_A(s-a)^T(I_{n}-K)\mbf e_d(b-\theta)&\theta \leq s,\\[-0.1em]
		-\mbf c_A(s-a)^TK\mbf e_d(b-\theta)& s < \theta,
	\end{cases}}
\end{align*}
where $K:=(H_a+H_b \mbs{Q} (b-a))^{-1}H_b$ with $\mbs{Q}$ as in~\eqref{eq:Q}, and $\mbf{e}_{1}:=(\hat{\mbf{e}}_{1}\otimes I_{n})^T$, $\mbf{e}_{d}:=(\hat{\mbf{e}}_{d}\otimes I_{n})^T$, where
\begin{align*}
	&\hat{\mbf{e}}_1(z):=\!\bmat{1,~z,~...,~\frac{z^{d-1}}{(d-1)!}},	\quad
	\hat{\mbf{e}}_d(z):=\!\bmat{\frac{z^{d-1}}{(d-1)!},~...,~ z,~1},	\\
	&\mbf{c}_A(z)^T\!:=\!	
	\scalemath{0.95}{\bbl[\mat{\mbs{H}_0(z),~ \mbs{H}_0(z) z \!+\! \mbs{H}_1(z),~ ...,\, \sum_{k=0}^{d-1} \mbs{H}_{k}(z)\frac{z^{d-k-1}}{(d-k-1)!}}\bbr]}. 
\end{align*}
Additionally, if $d=0$, the identities hold with $\mcl{T}=I_{n}$ and $\mcl{A}=\tnf{M}[\mbs{A}_{0}]$.
\end{cor}

Cor.~\ref{cor:gohberg_PIE} shows that, for $\mcl{D}$ admissible, we can define the inverse to the differential operator $\partial_{s}^{d}:\mcl{D}\to L_{2}$ as an integral operator $\mcl{T}:L_{2}\to \mcl{D}$. Moreover, for a PDE of the form $\partial_{t}\mbf{u}(t)=\mscr{H}\mbf{u}(t)$, we can define an equivalent PIE representation as $\partial_{t}\mcl{T}\mbf{v}(t)=\mcl{A}\mbf{v}(t)$, using fundamental state $\mbf{v}(t)=\partial_{s}^{d}\mbf{u}(t)$. 
The following corollary adapts this result to the multivariate space, $\partial_{s_{i}}^{\delta_{i}}:\hat{\mcl{D}}[i]\to L_{2}[\Omega]$.
 
\begin{cor}\label{cor:Tmap_1D}
	For any $i\in\{1:N\}$, let $\mcl{D}_{i}:=\hat{\mcl{D}}[i]\subseteq S_{2}^{\delta_{i}\cdot \tnf{e}_i,n}[\Omega]$ for $\Omega=\prod_{i=1}^{N}[a_{i},b_{i}]$ and $\mcl{D}$ of the form in~\eqref{eq:PDEdom_1D} with $d=\delta_{i}$ and $[a,b]=[a_{i},b_{i}]$. Define associated $\mcl{T},\mcl{A}$, and $\mscr{H}$ as in Cor.~\ref{cor:gohberg_PIE}. If we define the extensions of these operators, $\mcl{T}_{i},\mcl{A}_{i}:L_2^{n}[\Omega]\to L_{2}^{n}[\Omega]$ and $\mscr{H}_{i}:\mcl{D}_{i}\to L_{2}^{n}[\Omega]$, as 
	\begin{align*}
		&(\mcl{T}_i\mbf{v})(s):=\bl(\mcl{T}\mbf{v}(s_1,...,s_{i-1},\bullet,s_{i+1},...,s_{N})\br)(s_{i}),	\\
		&(\mcl{A}_{i}\mbf{v})(s):=\bl(\mcl{A}\mbf {v}(s_1,...,s_{i-1},\bullet,s_{i+1},...,s_{N})\br)(s_i),	\\
		&(\mscr{H}_{i}\mbf{v})(s):=\bl(\mscr{H}\mbf {v}(s_1,...,s_{i-1},\bullet,s_{i+1},...,s_{N})\br)(s_{i}),
	\end{align*}
	then for all $\mbf{u}\in\mcl{D}_{i}$, $\mcl{T}_{i}\partial_{s_{i}}^{\delta_{i}}\mbf{u}=\mbf{u}$ and $\mcl{A}_{i}\partial_{s_{i}}^{\delta_{i}}\mbf{u}=\mscr{H}_{i}\mbf{u}$, and for all $\mbf{v}\!\in\! L_{2}^{n}[\Omega]$, $\mcl{T}_{i}\mbf{v}\!\in\!\mcl{D}_{i}$, $\partial_{s_{i}}^{\delta_{i}}\mcl{T}_{i}\mbf{v}\!=\!\mbf{v}$ and $\mscr{H}_{i}\mcl{T}_{i}\mbf{v}\!=\!\mcl{A}_{i}\mbf{v}$.
\end{cor}
\begin{proof}
For any $\mbf{u}\in\mcl{D}_{i}$ and $\mbf{v}\in L_{2}^{n}[\Omega]$, we have $\mbf{u}(s_1,...,s_{i-1},\bullet,s_{i+1},...,s_N) \in \mcl{D}$ and\\ $\mbf{v}(s_1,...,s_{i-1},\bullet,s_{i+1},...,s_{N})\in L_{2}^{n}[a_{i},b_{i}]$ for almost every $s_{j}\in[a_{j},b_{j}]$ for $j\neq i$. By Cor.~\ref{cor:gohberg_PIE}, it follows that
\begin{align*}
	&(\mcl{T}\partial_{s_{i}}^{\delta_{i}}\mbf{u}(s_1,...,s_{i-1},\bullet,s_{i+1},...,s_{N}))(s_{i})
	=\mbf{u}(s),	\\
	&(\mcl{A}\partial_{s_{i}}^{\delta_{i}}\mbf{u}(s_1,...,s_{i-1},\bullet,s_{i+1},...,s_{N}))(s_{i})
	=\mcl{H}\mbf{u}(s),		\\
	&\partial_{s_{i}}^{\delta_{i}}(\mcl{T}_{i}\mbf{v}(s_1,...,s_{i-1},\bullet,s_{i+1},...,s_{N}))(s_{i})
	=\mbf{v}(s),	\\
	&(\mcl{H}\mcl{T}\mbf{v}(s_1,...,s_{i-1},\bullet,s_{i+1},...,s_{N}))(s_{i})
	=\mcl{A}\mbf{v}(s), 
\end{align*}
and $(\mcl{T}\mbf{v})(s_1,...,s_{i-1},\bullet,s_{i+1},...,s_{N})\in\mcl{D}$. By definition of $\mcl{T}_{i}$, $\mcl{A}_{i}$, and $\mcl{H}_{i}$, then, $\mcl{T}_{i}\partial_{s_{i}}^{\delta_{i}}\mbf{u}=\mbf{u}$ and $\mcl{A}_{i}\partial_{s_{i}}^{\delta_{i}}\mbf{u}=\mscr{H}_{i}\mbf{u}$, as well as $\partial_{s_{i}}^{\delta_{i}}\mcl{T}_{i}\mbf{v}=\mbf{v}$, $\mscr{H}_{i}\mcl{T}_{i}\mbf{v}=\mcl{A}_{i}\mbf{v}$, and $\mcl{T}_{i}\mbf{v}\in\mcl{D}_{i}$. 
\end{proof}

By Cor.~\ref{cor:Tmap_1D}, for any admissible $\mcl{D}_{i}$, we can define an inverse to $\partial_{s_{i}}^{\delta_{i}}:\mcl{D}_{i}\to L_{2}^{n}$ as a lifted, univariate integral operator, $\mcl{T}_{i}:L_{2}^{n}\to\mcl{D}_{i}$.
Using this result, the following corollary shows that $\mcl{T}:=\mcl{T}_{N}\cdots\mcl{T}_{1}$ defines an inverse of $D^\delta$ when the domain of $D^\delta$ is extended to the range of $\mcl T$. 

\begin{cor}\label{cor:LRinverse}
	For $\mcl{D}_{i}\subseteq S_{2}^{\delta_{i}\cdot\tnf{e}_{i},n}$ and associated $\mcl{T}_{i}:L_{2}^{n}\to\mcl{D}_{i}$ as in Cor.~\ref{cor:Tmap_1D}, let $\mcl{T}:=\mcl{T}_{N}\cdots\mcl{T}_{1}$ and $\mcl{D}:=\bigcap_{i}\mcl{D}_{i}$. Then $\mcl T D^\delta \mbf{u}=\mbf{u}$ and $D^{\delta}\mcl{T}\mbf{v}=\mbf{v}$ for any $\mbf{u}\in \mcl D$ and $\mbf{v}\in L_2^{n}$.\\[-2.0\baselineskip]
\end{cor}
\begin{pf}
	The result follows from the fact that 
	$\mcl{T}_{i}$ defines both a left- and right-inverse to $\partial_{s_{i}}^{\delta_{i}}:\mcl{D}_{i}\to L_{2}^{n}$. 
	 A formal proof is given in Cor.~\ref{cor:LRinverse_appx} in Appx.~\ref{appx:proofs}.\hfill $\blacksquare$
\end{pf}

By Cor.~\ref{cor:LRinverse}, $\mcl{T}$ is the inverse of $D^{\delta}:\mcl{D}\to L_{2}^{n}$, if and only if the image of $\mcl{T}$ is $\mcl{D}$. 
In the following subsection, we prove a necessary and sufficient condition for $\tnf{Im}(\mcl{T})=\mcl{D}$ to hold as \textit{consistency} of the domains $\mcl{D}_{i}$.

\subsection{Consistency of PDE Domains}\label{subsec:Tmap:consistency}

In this subsection, we see that if $\mcl{D}_{i}$ are admissible in the sense of Defn.~\ref{defn:admissible_BCs_1D}, and $\mcl{T}_i$ are as defined in Cor.~\ref{cor:Tmap_1D}, then we may define a necessary and sufficient condition for invertibility of $D^\delta := \prod_i \partial_{s_i}^{\delta_i}$ on $\mcl D = \bigcap_i \mcl{D}_i$ and show that such an inverse may be constructed as $\mcl T:=\prod_i \mcl{T}_i$.
Since we have already shown in Cor.~\ref{cor:LRinverse} that this $\mcl T$ is a left- and right-inverse of $D^\delta$ on the image of $\mcl{T}$, the only challenge is now to show that $\tnf{Im}(\mcl{T})=\mcl{D}$ or, equivalently, that $D^{\delta}:\mcl{D} \rightarrow L_2$ is surjective. To illustrate this challenge, the following example presents a 2D case where $D^{\delta}:\mcl{D}=\mcl{D}_{1}\cap\mcl{D}_{2} \rightarrow L_2$ is not surjective, as a result of ``inconsistency'' of the domains $\mcl{D}_i$.

\begin{example}\label{ex:inconsistent_BCs}
To illustrate the problem of inconsistent PDE domains, consider the domain $\mcl{D}=\mcl{D}_{1}\cap \mcl{D}_{2}$, where
\begin{align*}
	\mcl{D}_{1}&:=\bbbl\{\bbbl[\scalemath{0.9}{\mat{\mbf{u}_{1}\\[-0.2em]\mbf{u}_{2}}}\bbbr]\in S_{2}^{(1,0),2}[[0,1]^2]\:\bbl|\: \scalemath{0.9}{\begin{array}{l}\mbf{u}_{1}(0,y)=0\\[-0.2em]\mbf{u}_{2}(0,y)=\mbf{u}_{1}(1,y)\end{array}}\bbbr\},		\\[-0.2em]
	\mcl{D}_{2}&:=\bbbl\{\bbbl[\scalemath{0.9}{\mat{\mbf{u}_{1}\\[-0.2em]\mbf{u}_{2}}}\bbbr]\in S_{2}^{(0,1),2}[[0,1]^2]\:\bbl|\: \scalemath{0.9}{\begin{array}{l}\mbf{u}_{1}(x,0)=\mbf{u}_{2}(x,1)\\[-0.2em]\mbf{u}_{2}(x,0)=0\end{array}}\bbbr\}.
\end{align*}
Both of these domains are admissible per Defn.~\ref{defn:admissible_BCs_1D}, and by Cor.~\ref{cor:Tmap_1D}, we have $\mbf{u}=\mcl{T}_{1}\partial_{x}\mbf{u}=\mcl{T}_{2}\partial_{y}\mbf{u}$ for all $\mbf{u}\in \mcl{D}$, where
\begin{align*}
	\bl(\mcl{T}_{1}\mbf{v}\br)(x,y)&:=\int_{0}^{x}\bbbl[\scalemath{0.9}{\mat{1&0\\[-0.2em] 1&1}}\bbbr]\mbf{v}(\theta,y)d\theta +\int_{x}^{1}\bbbl[\scalemath{0.9}{\mat{0&0\\[-0.2em] 1&0}}\bbbr]\mbf{v}(\theta,y)d\theta,\\
	\bl(\mcl{T}_{2}\mbf{v}\br)(x,y)&:=\int_{0}^{y}\bbbl[\scalemath{0.9}{\mat{1&1\\[-0.2em] 0&1}}\bbbr]\mbf{v}(x,\eta)d\eta +\int_{y}^{1}\bbbl[\scalemath{0.9}{\mat{0&1\\[-0.2em] 0&0}}\bbbr]\mbf{v}(x,\eta)d\eta.
\end{align*}
Defining $\mcl{T}:=\mcl{T}_{2}\mcl{T}_{1}$, by Cor.~\ref{cor:LRinverse}, we then have $\partial_{x}\partial_{y}\mcl{T}\mbf v=\mbf{v}$ and $\mcl{T}\partial_{x}\partial_{y}\mbf{u}=\mbf{u}$ for all $\mbf{v}\in L_{2}^{2}[[0,1]^2]$ and $\mbf{u}\in\mcl{D}$. 
However, if we define $\mbf{v}_1=\mbf{v}_{2}=1$, then $\mbf{v} \in L_2^{2}[[0,1]^2]$, and $\mbf{u}=\mcl{T}\mbf{v}$ implies $\mbf{u}_1=1+x+xy$ and $\mbf u_2=y+xy$. Clearly, $\partial_{x}\partial_{y} \mcl{T} \mbf{v}=\mbf{v}$ and $\mcl{T}\partial_{x}\partial_{y}\mbf{u}=\mbf{u}$, but $\mbf{u}=\mcl{T}\mbf{v}\notin\mcl{D}$ since $\mbf{u}_{1}(0,y)\neq 0$. The problem, here, is that the domains $\mcl{D}_{1}$ and $\mcl{D}_{2}$ are not consistent. In particular, any $\mbf{u}\in\mcl{D}_{1}$ must satisfy $\mbf{u}_{1}(0,0)=0$, whereas any $\mbf{u}\in\mcl{D}_{2}$ must satisfy $\mbf{u}_{1}(0,0)=\mbf{u}_{2}(0,1)$. Although it is possible for both of these constraints to be satisfied simultaneously, this does impose another constraint, $\mbf{u}_{2}(0,1)=0$, which is not implied by either $\mbf{u}\in\mcl{D}_{1}$ or $\mbf{u}\in\mcl{D}_{2}$ individually. Similarly, any $\mbf{u}\in\mcl{D}$ must also satisfy $\mbf{u}_{2}(1,1)=0$, as well as $\mbf{u}_{1}(1,0)=\mbf{u}_{1}(1,1)=0$---all constraints that need not be satisfied for $\mbf{u}\in \mcl{D}_{1}$ or $\mbf{u}\in \mcl{D}_{2}$. 
As a result of these auxiliary constraints, any $\mbf{u}\in \mcl{D}$ must satisfy $\int_{0}^{1}\int_{0}^{1}\partial_{x}\partial_{y}\mbf{u}(x,y)dy dx =\mbf{u}(1,1)-\mbf{u}(0,1)-\mbf{u}(1,0)+\mbf{u}(0,0)= 0$, meaning that $\partial_{x}\partial_{y}:\mcl{D}\to L_{2}^{2}[[0,1]^2]$ is not surjective onto $L_{2}$, and thus not right-invertible.
\end{example}

This example illustrates a 2D case where the domains $\mcl{D}_i$ impose different constraints at the corners of the spatial domain. To avoid such inconsistencies, 
we propose the following constraint on the domains $\mcl D_i$, in terms of the corresponding parameters $B^{i}_{j,k},C^{i}_{j,k}$.

\begin{defn}[Consistent Domains]\label{defn:consistent_BCs}
	For each $i\in\{1:N\}$, let $\{[a_{i},b_{i}],\delta_{i},B^{i}_{j,k},C^{i}_{j,k}\}$ define $\mcl{D}_{i}:=\hat{\mcl{D}}[i]$, where $\mcl{D}\subseteq S_{2}^{\delta_{i},n}[a_{i},b_{i}]$ is of the form in~\eqref{eq:PDEdom_1D} and admissible as per Defn.~\ref{defn:admissible_BCs_1D}. Define $[H^{i}_a]_{j,k}:=B^{i}_{j,k}$, $[H^{i}_b]_{j,k}:=C^{i}_{j,k}$ and associated $K^{i}:=(H_{a}^{i}+H_{b}^{i}\mbs{Q}(b_{i}-a_{i}))^{-1}H_{b}^{i}\in\R^{n \delta_{i}\times n \delta_{i}}$ for $\mbs{Q}$ as in Defn.~\ref{defn:admissible_BCs_1D}, and decompose $K^{i}$ as
	\begin{align}\label{eq:M_decomp_2D} 
		~~K^{i}=\scalemath{0.95}{\bmat{K^{i}_{1,1}&\cdots&K^{i}_{1,\delta_{i}}\\[-0.2em]\vdots&\ddots&\vdots\\[-0.2em]K^{i}_{\delta_{i},1}&\cdots&K^{i}_{\delta_{i},\delta_{i}}}},\quad
		\mat{\text{where}~
		K^{i}_{k,\ell}\in\R^{n\times n},\\\qquad\;\forall k,\ell\in\{1:\delta_{i}\}.}
		\\[-1.3\baselineskip]\notag
	\end{align}
	We say the $\mcl{D}_{i}$ are \tbf{consistent} if $K^{i}_{k,p}K^{j}_{\ell,q}=K^{j}_{\ell,q}K^{i}_{k,p}$ for all $i,j\in\{1:N\}$, $k,p\in\{1:\delta_{i}\}$ and $\ell,q\in\{1:\delta_{j}\}$.\\[0.2em]
	\textbf{Note:} \textit{Commutability of the full matrices $K^{i}$ and $K^{j}$ is not necessary or sufficient for consistency, unless $\delta_{i}=\delta_{j}=1$.}
\end{defn}	

We now apply this consistency definition to Example~\ref{ex:inconsistent_BCs}.

\begin{example}\label{ex:inconsistent_2}
To illustrate the consistency condition from Defn.~\ref{defn:consistent_BCs}, consider again the 2D domain $\mcl{D}=\mcl{D}_{1}\cap\mcl{D}_{2}$ from Example~\ref{ex:inconsistent_BCs}. In this case, $\delta_{1}=\delta_{2}=1$, $n=2$, and $\mcl{D}_{1}$ is defined by $B^{1}=I_{2}$ and $C^{1}=\smallbmat{0&0\\1&0}$, and $\mcl{D}_{2}$ is defined by $B^{2}=I_{2}$ and $C^{2}=\smallbmat{0&1\\0&0}$. The corresponding matrices $K^{i}$ in Defn.~\ref{defn:consistent_BCs} are given by $K^{1}=\smallbmat{0&0\\1&0}$ and $K^{2}=\smallbmat{0&1\\0&0}$, which do not commute. Thus, $\mcl{D}_{1}$ and $\mcl{D}_{2}$ are not consistent.
\end{example}

Given this definition of consistency, the following lemma shows that $D^{\delta}:\bigcap_{i=1}^{N}\mcl{D}_{i}\to L_{2}^{n}$ is right-invertible only if the domains $\mcl{D}_i$ are consistent in the sense of Defn.~\ref{defn:consistent_BCs}. For simplicity, we consider only 2 spatial variables, extending this result to $N$D in Lem.~\ref{lem:RI}. 

\begin{lem}\label{lem:compatibility_Du(a,c)_2D}
	Let $\mcl{D}_{1}$, $\mcl{D}_{2}$ (of the form in Eqn.~\eqref{eq:PDEdom_1D} with $d=\delta_1$, $d=\delta_2$, respectively) be admissible as per Defn.~\ref{defn:admissible_BCs_1D}. 
	If there exists $\mcl{T}:L_{2}^{n}\to \mcl{D}_{1}\cap\mcl{D}_{2}$ such that $\partial_{x}^{\delta_{1}}\partial_{y}^{\delta_{2}}\mcl{T}\mbf{v}=\mbf{v}$ for all $\mbf{v}\in L_{2}^{n}$, then $\mcl{D}_{1}$ and $\mcl{D}_{2}$ are consistent.
\end{lem}
\begin{proof}
	We provide an outline of the proof here, referring to Lem.~\ref{lem:compatibility_Du(a,c)_2D_appx} in Appx.~\ref{appx:proofs} for a full proof. 
	We will assume $\delta_{1}=\delta_{2}=d$ and $\Omega=[0,1]^2$.\\	
	\indent Suppose there exists an operator $\mcl{T}:L_{2}^{n}[\Omega]\to \mcl{D}_{1}\cap\mcl{D}_{2}$ such that $\partial_{x}^{d}\partial_{y}^{d}\mcl{T}\mbf{v}=\mbf{v}$ for all $\mbf{v}\in L_{2}^{n}[\Omega]$. For any $\mbf{v}\in L_{2}^{n}[\Omega]$, let $\mbf{u}=\mcl{T}\mbf{v}$. Then, $\mbf{v}=\partial_{x}^{d}\partial_{y}^{d}\mbf{u}$ and $\mbf{u}\in \mcl{D}_{1}\cap \mcl{D}_{2}$.
	Using this fact, we can show that for each $k,\ell\in\{1:d\}$, we can express the corner value $(\partial_{x}^{k-1}\partial_{y}^{\ell-1}\mbf{u})(0,0)$ in terms of $\mbf{v}$ in two distinct manners, as 
	\begin{align*}
		&\int_{0}^{1}\!\int_{0}^{1}\!\bl(K^{2}_{(\ell,:)}\mbf{e}_{d}(1-y)\br)\bl(K^{1}_{(k,:)} \mbf{e}_{d}(1-x)\br)\mbf{v}(x,y)\,dy\,dx	\\[-0.1em]
		&=\bl(\partial_{x}^{k-1}\partial_{y}^{\ell-1}\mbf{u}\br)(0,0)	\\[-0.1em]
		&=\!\int_{0}^{1}\!\int_{0}^{1}\!\bl(K^{1}_{(k,:)} \mbf{e}_{d}(1-x)\br)\bl(K^{2}_{(\ell,:)}\mbf{e}_{d}(1-y)\br)\mbf{v}(x,y)\,dy\,dx,
	\end{align*}
	where $\mbf{e}_{d}$ is as in Cor.~\ref{cor:gohberg_PIE} and $K^{i}_{(k,:)}:=\bl[K^{i}_{k,i},\cdots,K^{i}_{k,\delta_{i}}\br]$. 
	Using the fact that $\mbf{e}_{d}(z)$ is polynomial, it follows that we must have $K^{1}_{k,p}K^{2}_{\ell,q}=K^{2}_{\ell,q}K^{1}_{k,p}$ for all $k,p,\ell,q$.
\end{proof}

While Lem.~\ref{lem:compatibility_Du(a,c)_2D} only shows that consistency of the $\mcl{D}_{i}$ is necessary for $\partial_{s_{i}}^{\delta_{i}}\partial_{s_{j}}^{\delta_{j}}:\mcl{D}_{i}\cap\mcl{D}_{j}\to L_{2}^{n}$ to be right-invertible, this condition is also sufficient.  
To establish this result, we first show that consistency of the $\mcl{D}_i$ implies commutability of the $\mcl{T}_i$---i.e. $\mcl{T}_i\mcl{T}_j=\mcl{T}_j\mcl{T}_i$.

\begin{lem}\label{lem:M_commute_T}
	For each $i\in\{1:N\}$, let $\mcl{D}_{i}$ and associated $\mcl{T}_{i}:L_{2}^{n}\to\mcl{D}_{i}$ be as in Cor.~\ref{cor:Tmap_1D}. 
	Then, for any $i,j\in\{1:N\}$,   $\mcl{T}_{i}\mcl{T}_{j}\mbf{v}=\mcl{T}_{j}\mcl{T}_{i}\mbf{v}$ for all $\mbf{v}\in L_2$ if and only if $\mcl{D}_{i}$ and $\mcl{D}_{j}$ are consistent in the sense of Defn.~\ref{defn:consistent_BCs}.
\end{lem}
\begin{pf}
A formal proof is given in Lem.~\ref{lem:M_commute_T_appx} in Appx.~\ref{appx:proofs}. This proof uses the definition
	\begin{align*}
		&(\mcl{T}_{i}\mbf{v})(s):=\int_{a_{i}}^{s_{i}}\scalemath{0.9}{\frac{(s_{i}-\theta_{i})^{\delta_{i}-1}}{(\delta_{i}-1)!}}\mbf{v}(s)|_{s_{i}=\theta_{i}}\,d\theta_{i}	\\[-0.4em]
		&\qquad-\!\int_{a_{i}}^{b_{i}}\bl[\mbf{e}_{1}(s_{i}-a_{i})^T K^{i}\mbf{e}_{\delta_{i}}(b_{i}-\theta_{i})\,\mbf{v}(s)|_{s_{i}=\theta_{i}}\br]d\theta_{i},
	\end{align*}
	for $\mbf{v}\in L_{2}[\Omega]$, and 
	where $\mbf{e}_{1},\mbf{e}_{\delta_{i}}$ are as in Cor.~\ref{cor:gohberg_PIE} (letting $d=\delta_{i}$). It follows that $\mcl{T}_{i}$ and $\mcl{T}_{j}$ commute if and only if $\mbf{e}_{1}(s_{i}-a_{i})^T K^{i}\mbf{e}_{\delta_{i}}(b_{i}-\theta_{i})$ and $\mbf{e}_{1}(s_{j}-a_{j})^T K^{j}\mbf{e}_{\delta_{j}}(b_{j}-\theta_{j})$ commute for all $s_{i},\theta_{i},s_{j},\theta_{j}$. 
	Since $\mbf{e}_{1}^T K^{i}\mbf{e}_{\delta_{i}}$ is a matrix-valued polynomial with coefficients $K^{i}_{p,\ell}$, 
	it follows that $\mcl{T}_{i}$ and $\mcl{T}_{j}$ commute if and only if $K^{i}_{k,p}$ and $K^{j}_{\ell,q}$ commute for all $k,p,\ell,q$, and thus $\mcl{D}_{i}$ and $\mcl{D}_{j}$ are consistent. \hfill$\blacksquare$
\end{pf}
\ \\[-3.0\baselineskip]

Lem.~\ref{lem:M_commute_T} shows that consistency of the $\mcl{D}_{i}$ is necessary and sufficient for the operators $\mcl{T}_{i}:L_{2}^{n}\to \mcl{D}_{i}$ and $\mcl{T}_{j}:L_{2}^{n}\to\mcl{D}_{j}$ for $i,j\in\{1:N\}$ to commute. Using this result, we finally prove that $D^{\delta}:\mcl{D}\to L_{2}^{n}$ is right-invertible if and only if the domains $\mcl{D}_{i}$ are consistent.

\begin{lem}\label{lem:RI}
	For each $i\in\{1:N\}$, let $\mcl{D}_{i}\subseteq S_{2}^{\delta_{i}\cdot\tnf{e}_{i},n}$ as in Cor.~\ref{cor:Tmap_1D} be admissible as per Defn.~\ref{defn:admissible_BCs_1D}, and let $\mcl{D}:=\bigcap_{i=1}^{N}\mcl{D}_{i}$.
	Then $D^\delta: \mcl{D}\rightarrow L_2^{n}$ is right-invertible if and only if the $\mcl{D}_i$ are consistent as per Defn.~\ref{defn:consistent_BCs}. Furthermore, the right-inverse is then given by $\mcl{T}:=\prod_{i=1}^{N}\mcl{T}_{i}$ for $\mcl{T}_{i}$ as in Cor.~\ref{cor:Tmap_1D}.
\end{lem}
\begin{proof}
	For sufficiency, suppose the domains $\mcl{D}_i$ are consistent. Then, by Lem.~\ref{lem:M_commute_T}, the operators $\mcl{T}_{i}:L_{2}^{n}\to\mcl{D}_{i}$ defined in Cor.~\ref{cor:Tmap_1D} commute. Defining $\mcl{T}:=\prod_{i=1}^{N}\mcl{T}_{i}$,
	then, $\mcl{T}\mbf{v}=\mcl{T}_{i}(\prod_{j\neq i}\mcl{T}_{j})\mbf{v}\in\mcl{D}_{i}$ for all $i$, and therefore $\mcl{T}:L_{2}^{n}\to \mcl{D}$. Since, by Cor.~\ref{cor:LRinverse}, $D^{\delta}\mcl{T}\mbf{v}=\mbf{v}$ for all $\mbf{v}\in L_{2}^{n}$, it follows that $\mcl{T}$ defines a right-inverse to $D^{\delta}:\mcl{D}\to L_{2}^{n}$.

	For necessity, suppose that 
	there exists $\mcl{T}:L_{2}^{n}\to \mcl{D}$ such that $D^{\delta}\mcl{T}=I_{n}$. To prove that the domains $\mcl{D}_{i}$ are consistent, we first show that for every $i\neq j$, the operator $\partial_{s_{i}}^{\delta_{i}}\partial_{s_{j}}^{\delta_{j}}:\mcl{D}_{i}\cap\mcl{D}_{j}\to L_{2}^{n}$ is right-invertible. To this end, note that the operators $\partial_{s_{i}}^{\delta_{i}}$ commute on $\mcl{D}\subseteq S_{2}^{\delta,n}$, so that $D^{\delta}\mbf{u}=\partial_{s_{i}}^{\delta_{i}}\partial_{s_{j}}^{\delta_{j}}(\prod_{k\notin\{i,j\}}\partial_{s_{k}}^{\delta_{k}})\mbf{u}$ for $\mbf{u}\in\mcl{D}$. Here, by Lem.~\ref{lem:PDEdom_diff_appx} in Appx.~\ref{appx:proofs}, we have for every $i\neq j$ that $(\prod_{k\notin\{i,j\}}\partial_{s_{k}}^{\delta_{k}})\mbf{u}\in \mcl{D}_{i}\cap\mcl{D}_{j}$ for all $\mbf{u}\in \mcl{D}$. Defining then $\tilde{\mcl{T}}_{i,j}:=(\prod_{k\notin\{i,j\}}\partial_{s_{k}}^{\delta_{k}})\mcl{T}$, the fact that $\mcl{T}:L_{2}^{n}\to\mcl{D}$ implies $\tilde{\mcl{T}}_{i,j}:L_{2}^{n}\to \mcl{D}_{i}\cap\mcl{D}_{j}$. Furthermore, $\tilde{\mcl{T}}_{i,j}$ also satisfies $\partial_{s_{i}}^{\delta_{i}}\partial_{s_{j}}^{\delta_{j}}\tilde{\mcl{T}}_{i,j}\mbf{v}=\partial_{s_{i}}^{\delta_{i}}\partial_{s_{j}}^{\delta_{j}}(\prod_{k\notin\{i,j\}}\partial_{s_{k}}^{\delta_{k}})\mcl{T}\mbf{v}=D^{\delta}\mcl{T}\mbf{v}=\mbf{v}$ for all $\mbf{v}\in L_{2}^{n}$, and thus defines a right-inverse to $\partial_{s_{i}}^{\delta_{i}}\partial_{s_{j}}^{\delta_{j}}:\mcl{D}_{i}\cap\mcl{D}_{j}\to L_{2}^{n}$. Since this holds for all $i\neq j$, it follows by Lem.~\ref{lem:compatibility_Du(a,c)_2D} that the domains $\mcl{D}_{i}$ are consistent.
\end{proof}

Lem.~\ref{lem:RI} shows that if the $\mcl{D}_{i}$ are admissible and we define $\mcl{D}:=\bigcap_{i}\mcl{D}_{i}$, then consistency of $\mcl{D}_i$ is necessary and sufficient for $D^{\delta}: \mcl{D}\to L_{2}^{n}$ to be right-invertible, with right-inverse $\mcl{T}:=\prod_{i=1}^{N}\mcl{T}_{i}$. Furthermore, recall that Cor.~\ref{cor:LRinverse} establishes that $\mcl{T}$ is a left inverse of $D^\delta$---implying that $D^{\delta}:\mcl{D} \rightarrow L_2$ is invertible if and only if the domains $\mcl{D}_{i}$ are consistent. This implies a bijection between the domain $\mcl D$ and the Hilbert space $L_2$. In the following subsection, we will apply this bijection to the state of a class of linear multivariate PDEs to construct an equivalent evolution equation on the Hilbert space $L_2$, where this evolution equation is not constrained by boundary conditions.

\subsection{A PIE Representation of Multivariate PDEs}\label{subsec:Tmap:PDE2PIE}

Having established necessary and sufficient conditions for $D^{\delta}:\bigcap_{i}\mcl{D}_{i}\to L_{2}^{n}$ to be invertible, and having obtained an explicit representation of this inverse,
we now apply this result to construct the PIE representation of a class of linear, coupled, multivariate PDEs. Specifically, we parameterize a class of PDEs by $\delta, \mbs{A}_{\alpha}, \mcl D_i$ as
\begin{align}\label{eq:PDE_standard}
	\	\notag \\[-1.3\baselineskip]
	\partial_{t}\mbf{u}(t,s)=\!\!\sum_{\vec{0}\leq\alpha\leq\delta}\!\mbs{A}_{\alpha}(s)D^{\alpha}\mbf{u}(t,s),\quad \mbf{u}(t)\in \bigcap_{i=1}^{N}\mcl{D}_{i},
	\\[-1.3\baselineskip] \notag
\end{align}
where $s\in\Omega:=\prod_{i=1}^{N}[a_{i},b_{i}]$, $\mbs{A}_{\alpha} \in \R[s]$, and where the $\mcl D_i$ are lifted, univariate, and admissible domains as per Defn.~\ref{defn:admissible_BCs_1D}, and are consistent as per Defn.~\ref{defn:consistent_BCs}. 
Given a PDE of the form in~\eqref{eq:PDE_standard}, we define a classical solution as follows.

\begin{defn}[Classical Solution to PDE]
	For $\mbs{A}_{\alpha} \in \R^{n\times n}[s]$ and admissible and consistent $\mcl D_i$, we say that $\mbf{u}:[0,\infty)\to S_{2}^{\delta,n}[\Omega]$ solves the PDE defined by $\{\mbs{A}_{\alpha},\mcl{D}_{i}\}$ with initial condition $\mbf{u}_{0} \in \mcl D:=\bigcap_{i=1}^{N}{\mcl{D}}_{i}$ if $\mbf{u}$ is Fr\'echet differentiable, $\mbf{u}(0)=\mbf{u}_{0}$, and $\mbf{u}(t)$ satisfies~\eqref{eq:PDE_standard} for all $t\ge 0$.
\end{defn}

Having specified a class of multivariate PDEs, we now map solutions of such PDEs to solutions of an associated evolution equation on the \textit{fundamental state} $\mbf v(t):=D^\delta \mbf u(t)$. To define this evolution, however, we need a mapping from $\mbf v(t)$ to $\mbf u(t)$. This is achieved in the following theorem, by applying Cor.~\ref{cor:gohberg_PIE} to each of the univariate differential operators in $D^{\alpha}=\prod \partial_{s_{i}}^{\alpha_i}$, and combining the results using the embedding of univariate operators in multivariate space. 

\begin{thm}[Multivariate Extension of Cor.~\ref{cor:gohberg_PIE}]\label{thm:Tmap}
	\ \\
	For $n,N\in \N$, $\Omega=\prod_{i=1}^{N} [a_i,b_i]$, and $\delta \in \N_{0}^N$, let $\mcl{D}:=\bigcap_{i=1}^{N}\mcl{D}_{i} \subseteq S_{2}^{\delta,n}[\Omega]$ with $\mcl{D}_{i}\subseteq S_2^{\delta_i\cdot\tnf{e}_{i},n}[\Omega]$ admissible and consistent and define $\mscr{H}:=\sum_{\vec{0}\leq\alpha\leq\delta}\tnf{M}[\mbs{A}_{\alpha}]D^{\alpha}$. 
	
Let $\mcl{T}_i, \mcl{A}_{i,j}$ be defined as in Cor.~\ref{cor:gohberg_PIE} and extended to the multivariate space as in Cor.~\ref{cor:Tmap_1D} (letting $\mcl{T}\mapsto\mcl{T}_{i}$ and $\mcl{A}\mapsto\mcl{A}_{i,j}$), for $d=\delta_i$, $\mcl D=\mcl D_i$, and $\{\mbs H_k\}$ where $\mbs H_k:=\bbl\{\scalemath{0.9}{\mat{I_{n},\hspace*{0.25cm} k=j,\\[-0.4em] 0,\quad k\neq j,}}$ for each $k\in\{0:d\}$.
	Define $\mcl{T}:=\prod_{i=1}^{N}\mcl{T}_{i}$ and $\mcl{A}:=\sum_{\vec{0}\leq\alpha\leq\delta}\bl(\tnf{M}[\mbs{A}_{\alpha}]\prod_{i=1}^{N}\mcl{A}_{i,\alpha_{i}}\br)$. Then:
	\begin{enumerate}
		\item For all $\mbf{u}\in \mcl{D}$, we have $D^{\delta}\mbf{u}\in L_{2}^{n}[\Omega]$, $\mbf{u}=\mcl{T}D^{\delta}\mbf{u}$, and $\mscr{H}\mbf{u} =\mcl{A}D^{\delta}\mbf{u}$.
		\item 
		For all $\mbf{v}\in L_2^{n}[\Omega]$, we have $\mcl{T}\mbf{v}\in \mcl{D}$, $\mbf{v}=D^{\delta}\mcl{T}\mbf{v}$, and $\mcl{A}\mbf{v}=\mscr{H}\mcl{T}\mbf{v}$.
	\end{enumerate}
\end{thm}
\begin{proof}
	For the first statement, fix arbitrary $\mbf{u}\in\mcl{D}\subseteq S_{2}^{\delta,n}$. Then $D^{\delta}\mbf{u}\in L_{2}^{n}$ since $\mbf{u}\in S_{2}^{\delta,n}$, and $\mbf{u}=\mcl{T}D^{\delta}\mbf{u}$ by Cor.~\ref{cor:LRinverse}. Finally, by Cor.~\ref{cor:Tmap_1D} (and by definition of the $\mbs{H}_{k}$), we have $\mcl{A}_{i,j}\partial_{s_{i}}^{\delta_i}\mbf{u}=\partial_{s_{i}}^{j}\mbf{u}$ for all $j$ and $i$, and therefore
	\begin{align*}
		& \\[-1.5\baselineskip]
		&\mcl{A}D^{\delta}\mbf{u}
		=\!\sum_{\alpha\leq\delta}\tnf{M}[\mbs{A}_{\alpha}]\prod_{i=1}^{N}\bl(\mcl{A}_{i,\alpha_{i}}\br)\prod_{\ell=1}^{N}\bl(\partial_{s_{\ell}}^{\delta_{\ell}}\br)\mbf{u}	\\[-0.5em]
		&=\!\sum_{\alpha\leq\delta }\tnf{M}[\mbs{A}_{\alpha}]\prod_{i=1}^{N}\bl(\mcl{A}_{i,\alpha_{i}}\partial_{s_{i}}^{\delta_{i}}\br)\mbf{u}		
		\!=\!\sum_{\alpha\leq\delta }\tnf{M}[\mbs{A}_{\alpha}]\prod_{i=1}^{N}(\partial_{s_{i}}^{\alpha_{i}})\mbf{u}
		=\mscr{H}\mbf{u}, \\[-1.5\baselineskip]
	\end{align*}
	where we remark that $\mcl{A}_{i,\alpha_{i}}$ and $\partial_{s_{\ell}}^{\delta_{\ell}}$ commute for $i\neq \ell$ by Lem.~\ref{lem:Top_orth_Dop} in Appx.~\ref{appx:proofs}.
	
	For the second statement, fix arbitrary $\mbf{v}\in L_{2}^{n}$. Then, by Cor.~\ref{cor:LRinverse} and Lem.~\ref{lem:RI}, we have $D^{\delta}\mcl{T}\mbf{v}=\mbf{v}$ and $\mcl{T}\mbf{v}\in\mcl{D}$. Finally, by Cor.~\ref{cor:gohberg_PIE}, we have $\mcl{A}_{i,j}\mbf{v}=\partial_{s_{i}}^{j}\mcl{T}_{i}\mbf{v}$, and therefore
	\begin{align*}
		& \\[-1.5\baselineskip]
		\mcl{A}\mbf{v}
		&=\sum_{\alpha\leq\delta}\tnf{M}[\mbs{A}_{\alpha}]\prod_{i=1}^{N}\bl(\mcl{A}_{i,\alpha_{i}}\br)\mbf{v}	
		=\sum_{\alpha\leq\delta}\tnf{M}[\mbs{A}_{\alpha}]\prod_{i=1}^{N}\bl(\partial_{s_{i}}^{\alpha_{i}}\mcl{T}_{i}\br)\mbf{v}	\\[-0.5em]
		&\quad~=\sum_{\alpha\leq\delta}\tnf{M}[\mbs{A}_{\alpha}]\prod_{\ell=1}^{N}\bl(\partial_{s_{\ell}}^{\alpha_{\ell}}\br)\prod_{i=1}^{N}\bl(\mcl{T}_{i}\br)\mbf{v}
		=\mscr{H}\mcl{T}\mbf{v},
	\end{align*}
	where we remark that also $\mcl{T}_{i}$ and $\partial_{s_{\ell}}^{\alpha_{\ell}}$ commute for $i\neq \ell$ by Lem.~\ref{lem:Top_orth_Dop} in Appx.~\ref{appx:proofs}.
\end{proof}

We now apply the results of Thm.~\ref{thm:Tmap} to the PDE in~\eqref{eq:PDE_standard}, by defining an associated fundamental state as follows.
\begin{defn}[Fundamental State]
For a linear PDE as in~\eqref{eq:PDE_standard}, defined by $\{\delta,\mbs{A}_{\alpha},\mcl{D}_{i}\}$, we define the fundamental state associated with PDE state $\mbf{u}(t)$ as $\mbf{v}(t):=D^{\delta}\mbf{u}(t)$.	
\end{defn}

Taking the highest-order mixed derivative of the PDE state, the fundamental state captures a minimal amount of information necessary to represent solutions to the PDE at any time, discarding boundary information of the solution that is already encoded in the domains $\mcl{D}_{i}$. If the domains are admissible and consistent, Thm.~\ref{thm:Tmap} then proves that the PDE state can be recovered from the fundamental state as $\mbf{u}(t)=\mcl{T} \mbf{v}(t)$, showing that no information is lost. Using the mapping $\sum_{\vec{0}\leq\alpha\leq\delta}\tnf{M}[\mbs{A}_{\alpha}]D^{\alpha}\mbf{u}(t) =\mcl{A}\mbf{v}(t)$ from Thm.~\ref{thm:Tmap}, 
it follows that $\mbf{v}(t)$ satisfies the following evolution equation,
\begin{equation}\label{eq:PIE_standard}
	\partial_{t}\mcl{T}\mbf{v}(t)=\mcl{A}\mbf{v}(t),\qquad \mbf{v}(t)\in L_{2}^{n}[\Omega].
\end{equation}
We refer to a system of this form as a Partial Integral Equation (PIE).
We define a solution to the PIE as follows.

\begin{defn}[Classical Solution to PIE]
	For a given initial state $\mbf{v}_{0}\in L_{2}^{n}[\Omega]$, we say that $\mbf{v}:[0,\infty)\to L_{2}^{n}[\Omega]$ solves the PIE defined by $\{\mcl{T},\mcl{A}\}$ if $\mbf{v}$ is Fr\'echet differentiable, $\mbf{v}(0)=\mbf{v}_{0}$, and $\mbf{v}(t)$ satisfies~\eqref{eq:PIE_standard} for all $t\geq 0$.
\end{defn}

For a given PDE as in~\eqref{eq:PDE_standard}, defining $\{\mcl{T},\mcl{A}\}$ as in Thm.~\ref{thm:Tmap}, the following theorem shows that any solution to the PDE can be mapped to a solution to the PIE~\eqref{eq:PIE_standard}, and vice versa, 
so that the PIE in fact defines an equivalent representation of the PDE.

\begin{thm}\label{thm:PDE2PIE}
	For $N\in\N$ and $\delta\in\N_{0}^{N}$, let $\mcl{D}_{i}\subseteq S_{2}^{\delta_{i}\cdot\tnf{e}_{i},n}$ be admissible and consistent, and let $\mcl{D}:=\bigcap_{i=1}^{N}\mcl{D}_{i}$. Let $\mbs{A}_{\alpha}\in\R^{n\times n}[s]$, for each $\vec{0}\leq\alpha\leq\delta$. Define associated operators $\{\mcl{T},\mcl{A}\}$ as in Thm.~\ref{thm:Tmap}. Then the following hold:
	\begin{enumerate}
		\item 
		If $\mbf{u}$ solves the PDE defined by $\{\mbs{A}_{\alpha},\mcl{D}_{i}\}$ with initial value $\mbf{u}_{0}\in \mcl{D}$, then $\mbf{v}:=D^{\delta}\mbf{u}$ solves the PIE defined by $\{\mcl{T},\mcl{A}\}$ with initial value $\mbf{v}_{0}=D^{\delta}\mbf{u}_{0}$, and $\mbf{u}=\mcl{T}\mbf{v}$. 
		
		\item 
		If $\mbf{v}$ solves the PIE defined by $\{\mcl{T},\mcl{A}\}$ with initial value $\mbf{v}_{0}\in L_{2}^{n}$, then $\mbf{u}=\mcl{T}\mbf{v}$ solves the PDE defined by $\{\mbs{A}_{\alpha},\mcl{D}_{i}\}$ with initial value $\mbf{u}_{0}=\mcl{T}\mbf{v}_{0}$, and $\mbf{v}=D^{\delta}\mbf{u}$. 
	\end{enumerate}
\end{thm}
\begin{proof}
For the first statement, fix arbitrary $\mbf{u}_{0}\in \mcl{D}$, and let $\mbf{u}$ be an associated solution to the PDE in~\eqref{eq:PDE_standard}. Then $\mbf{u}(t)\in \mcl{D}\subseteq S_{2}^{\delta,n}$ for $t\geq 0$. Let $\mbf{v}(t):=D^{\delta}\mbf{u}(t)\in L_{2}^{n}$ for all $t\geq 0$ and $\mbf{v}_{0}:=D^{\delta}\mbf{u}_{0}\in L_{2}^{n}$. Clearly, then, $\mbf{u}(0)=\mbf{u}_{0}$ implies $\mbf{v}(0)=\mbf{v}_{0}$. In addition, by Thm.~\ref{thm:Tmap}, and by definition of the operators $\{\mcl{T},\mcl{A}\}$, we have $\mbf{u}(t)=\mcl{T}D^{\delta}\mbf{u}(t)
=\mcl{T}\mbf{v}(t)$ and $\sum_{\vec{0}\leq\alpha\leq\delta}\mbs{A}_{\alpha}(s)D^{\alpha}\mbf{u}(t,s)=\mcl{A}\mbf{v}(t,s)$.
It follows that $\partial_{t}\mcl{T}\mbf{v}(t,s)-\mcl{A}\mbf{v}(t,s)=\partial_{t}\mbf{u}(t,s)-\sum_{\vec{0}\leq \alpha\leq\delta}\mbs{A}_{\alpha}(s)D^{\alpha}\mbf{u}(t,s)$.
Since $\mbf{u}(t)$ satisfies~\eqref{eq:PDE_standard}, it follows that $\mbf{v}(t)$ satisfies~\eqref{eq:PIE_standard}, and thus the first statement holds.

For the second statement, fix arbitrary $\mbf{v}_{0}\in L_{2}^{n}[\Omega]$, and let $\mbf{v}$ be a corresponding solution to the PIE defined by $\{\mcl{T},\mcl{A}\}$. Let $\mbf{u}(t):=\mcl{T}\mbf{v}(t)$. Then, by Thm.~\ref{thm:Tmap}, $\mbf{u}(t)\in \mcl{D}$ for all $t\geq 0$, and $\sum_{\vec{0}\leq\alpha\leq\delta}\mbs{A}_{\alpha}(s)D^{\alpha}\mbf{u}(t,s)=\mcl{A}\mbf{v}(t,s)$. By the proof of the first statement, we find that $\partial_{t}\mcl{T}\mbf{v}(t,s)-\mcl{A}\mbf{v}(t,s)=\partial_{t}\mbf{u}(t,s)-\sum_{\vec{0}\leq \alpha\leq\delta}\mbs{A}_{\alpha}(s)D^{\alpha}\mbf{u}(t,s)$.
Since $\mbf{v}(t)$ satisfies~\eqref{eq:PIE_standard}, it follows that $\mbf{u}(t)$ satisfies~\eqref{eq:PDE_standard}, and thus the second statement holds as well. 
\end{proof}

By Thm.~\ref{thm:PDE2PIE}, for any PDE of the form in~\eqref{eq:PDE_standard} with admissible and consistent $\mcl{D}_{i}$, we can define an equivalent representation as a PIE of the form in~\eqref{eq:PIE_standard}, with any classical solution $\mbf{u}$ to the PDE admitting a classical solution $D^{\delta}\mbf{u}$ to the PIE, and any classical solution $\mbf{v}$ to the PIE admitting a classical solution $\mcl{T}\mbf{v}$ to the PDE.
While Theorem~\ref{thm:PDE2PIE} does not guarantee existence or uniqueness of such classical solutions, it does imply that well-posedness of either representation can be inferred from well-posedness of the other. Here, the admissibility and consistency conditions already ensure well-posedness of the boundary value problem defined by domain $\mcl{D}$, in the sense that for any $\mbf{v}\in L_{2}$ there exists a unique $\mbf{u}=\mcl{T}\mbf{v}\in\mcl{D}$ satisfying $D^{\delta}\mbf{u}=\mbf{v}$.
This precludes a potential cause of ill-posedness of the PDE, simplifying well-posedness analysis in the PIE representation. In particular, for 1D PDEs, well-posedness can be tested in the PIE representation by solving an optimization program similar to the stability test in Section~\ref{sec:stability}, as has recently been shown in~\cite{jagt2026WellPosedPIE_arXiv}.
	
Although Thm.~\ref{thm:PDE2PIE} requires the domains $\mcl{D}_{i}$ to be admissible and consistent, this imposes only mild constraints on the boundary conditions.
In practice, most common boundary conditions, including Dirichlet, Robin, and mixed boundary conditions, satisfy the admissibility constraints from Defn.~\ref{defn:admissible_BCs_1D}. 
In addition, the consistency conditions from Defn.~\ref{defn:consistent_BCs} are trivially satisfied in the scalar-valued case, $n=1$, and will be satisfied for $n>1$ if the matrices $B^{i}_{j,k},C^{i}_{j,k}$ defining the $\mcl{D}_{i}$ in Defn.~\ref{defn:admissible_BCs_1D} are all diagonal---i.e. if there is no coupling between state variables in the boundary conditions. 
For example, imposing mixed Dirichlet and Neumann boundary conditions on all state variables, 
letting $\mcl{D}_{i}$ be of the form of $\mcl{D}_{\tnf{D}}$, $\mcl{D}_{\tnf{N}1}$ or $\mcl{D}_{\tnf{N}2}$ in Eqn.~\eqref{eq:PDE_dom_DN}, 
the resulting $\mcl{D}_{i}$ will be consistent. 
The following corollary---a reformulation of Cor.~\ref{cor:Tmap_DN} from the start of this section---explicitly constructs the operators defining the PIE representation of the $N$D heat equation with such Dirichlet and Neumann boundary conditions.

\begin{cor}\label{cor:Tmap_DN_2}
	Let $\mcl{D} =\bigcap_{i=1}^{N} \mcl{D}_{i}$, where $\mcl{D}_{i}\in\{\mcl{D}_{\tnf{D}},\mcl{D}_{\tnf{N}1},\mcl{D}_{\tnf{N}2}\}$ for each $i\in\{1:N\}$, for $\mcl{D}_{\tnf{D}}$, $\mcl{D}_{\tnf{N}1}$, $\mcl{D}_{\tnf{N}2}$ as in~\eqref{eq:PDE_dom_DN}. Then $\mbf{u}=\mcl T D^{\vec{2}}\mbf{u}$, $\mbf{v}=D^{\vec{2}}\mcl{T} \mbf{v}$ and $\mcl{T} \mbf{v} \in \mcl{D}$ for all $\mbf{u} \in \mcl{D}$ and $\mbf{v} \in L_2^{n}$, where
	$(\mcl T\mbf{v})(s):=\int_{\Omega}\mbs{G}(s,\theta) \mbf{v}(\theta)d\theta$ with $\mbs{G}(s,\theta):=\prod_{i=1}^{N}\mbs{G}_{i}(s_{i},\theta_{i})$ for $\mbs{G}_{i}(x,y):=\bbl\{\!\scalemath{0.9}{\lmat{\mbs{h}_{i}(x,y)+(x-y),~ y\leq x,\\[-0.3em]\mbs{h}_{i}(x,y),\hspace*{1.63cm} y>x,}}$
	where\\[-1.25\baselineskip]
	\begin{equation*}
		\mbs{h}_{i}(x,y):=\begin{cases}
			\frac{-(x-a_{i})(b_{i}-y)}{b_{i}-a_{i}},	&	\mcl{D}_i={\mcl{D}}_{\tnf{D}},	\\
			a_{i}-x,	&	\mcl{D}_i={\mcl{D}}_{\tnf{N}1},	\\
			y-b_{i},	&	\mcl{D}_i={\mcl{D}}_{\tnf{N}2}.
		\end{cases}
	\end{equation*}
	Moreover, $\mbf{u}(t)\in\mcl{D}$ satisfies the $N$D heat equation, $\partial_{t}\mbf{u}(t)=\nabla^2\mbf{u}(t)=\sum_{i=1}^{N}\partial_{s_{i}}^2\mbf{u}(t)$, if and only if $\mbf{v}(t)=D^{\vec{2}}\mbf{u}(t)\in L_{2}^{n}$ satisfies the PIE $\partial_{t} \mcl{T}\mbf{v}(t)=\mcl{A}\mbf{v}(t)$, where $\mcl{A}\mbf{v}(s):=\sum_{j=1}^{N}\prod_{i\neq j}(\int_{[a_{i},b_{i}]}\mbs{G}_{i}(s_{i},\theta_{i}))\mbf{v}(\theta)d\theta$.
\end{cor}
\begin{proof}
	First, we define the parameters of the $\mcl D_i$ used in Eqn.~\eqref{eq:PDEdom_1D} as $B^{i}_{0,0}=C^{i}_{1,0}=I_{n}$ if $\mcl{D}_{i}=\mcl{D}_{\tnf{D}}$; $B^{i}_{0,0}=C^{i}_{1,1}=I_{n}$ if $\mcl{D}_{i}=\mcl{D}_{\tnf{N}1}$; and $B^{i}_{0,1}=C^{i}_{1,0}=I_{n}$ if $\mcl{D}_{i}=\mcl{D}_{\tnf{N}2}$; with $B^{i}_{j,k}=C^{i}_{j,k}=0_{n}$ for all other $j,k\in\{0,1\}$. Each $\mcl{D}_{i}$ is admissible as per Defn.~\ref{defn:admissible_BCs_1D}, and we can compute $K^{i}=(H^{i}_{a}+H^{i}_{b}\mbs{Q}(b_{i}-a_{i}))^{-1}H^{i}_{b}$ to find
	$K^{i}=\frac{1}{b_{i}-a_{i}}\smallbmat{0_{n}&0_{n}\\I_{n}&0_{n}}$ for $\mcl{D}_{i}=\mcl{D}_{\tnf{D}}$; $K^{i}=\smallbmat{0_{n}&0_{n}\\I_{n}&0_{n}}$ for $\mcl{D}_{i}=\mcl{D}_{\tnf{N1}}$;  and $K^{i}=\smallbmat{I_{n}&0_{n}\\0_{n}&0_{n}}$ for $\mcl{D}_{i}=\mcl{D}_{\tnf{N2}}$. 
	Decomposing the $K^{i}$ as in Defn.~\ref{defn:consistent_BCs}, the blocks $K^{i}_{k,p}\in\R^{n\times n}$ and $K^{j}_{\ell,q}\in\R^{n\times n}$ then trivially commute for all $i,j\in\{1:N\}$ and $k,\ell,p,q\in\{1,2\}$, and thus the $\mcl D_i$ are also consistent as per Defn.~\ref{defn:consistent_BCs}. 
	Now, defining $(\mcl{T}_{i}\mbf{v})(s_{i}):=\int_{a_{i}}^{b_{i}}\mbs{G}_{i}(s_{i},\theta_{i})\mbf{v}(\theta_{i})d\theta_{i}$, for each $i\in\{1:N\}$, we find by Cor.~\ref{cor:gohberg_PIE} and Cor.~\ref{cor:Tmap_1D} that $\mbf{u}=\mcl{T}_{i}\partial_{s_{i}}^{2}\mbf{u}$, $\mcl{T}_{i}\mbf{v}\in\hat{\mcl{D}}_{i}[i]$, and $\mbf{v}=\partial_{s_{i}}^{2}\mcl{T}_{i}\mbf{v}$ for all $\mbf{u}\in\hat{\mcl{D}}_{i}[i]$ and $\mbf{v}\in L_{2}^{n}[\Omega]$. By definition of the operators $\mcl{T}$ and $\mcl{A}$, it follows by Thm.~\ref{thm:Tmap} that $\mbf{u}=\mcl{T}D^{\vec{2}}\mbf{u}$, $\mcl{T}\mbf{v}\in\mcl{D}$, and $D^{\vec{2}}\mcl{T}\mbf{v}=\mbf{v}$ for all $\mbf{u}\in\mcl{D}$ and $\mbf{v}\in L_{2}^{n}[\Omega]$, and that $\mbf{u}(t)$ satisfies the $N$D heat equation if and only if $\mbf{v}(t):=D^{\vec{2}}\mbf{u}(t)$ satisfies $\partial_{t} \mcl{T}\mbf{v}(t)=\mcl{A}\mbf{v}(t)$.
\end{proof}

Cor.~\ref{cor:Tmap_DN_2} shows how an $N$D heat equation with mixed boundary conditions can be expressed as a PIE, providing an explicit expression for the operators $\{\mcl{T},\mcl{A}\}$ defining this PIE representation.
The following example illustrates how these operators are defined for a 2D reaction-diffusion equation, with mixed boundary conditions.
\begin{example}\label{ex:BCs_DN_2}
	To illustrate the construction of the PIE representation for a multivariate PDE, consider the following 2D reaction-diffusion equation from Example~\ref{ex:BCs_DN_0},
	\begin{align}\label{eq:example_heat_eq}
		\mbf{u}_{t}(t,x,y)&\!=\!\mbf{u}_{xx}(t,x,y)\!+\!\mbf{u}_{yy}(t,x,y)\!+\!r\mbf{u}(t,x,y),	\\
		\mbf{u}(t,0,y)&=\mbf{u}(t,1,y)=0,\quad \mbf{u}(t,x,0)=\mbf{u}_{y}(t,x,1)=0,\notag
	\end{align}
	where $\mbf{u}(t)\in S_{2}^{(2,2)}[[0,1]^2]$. Then, by Cor.~\ref{cor:Tmap_1D}, for any solution $\mbf{u}$ to the PDE, $\mbf{u}=\mcl{T}_{1}\mbf{u}_{xx}=\mcl{T}_{2}\mbf{u}_{yy}$, where
	\begin{align*}
		(\mcl{T}_{1}\mbf{v})(x,y)		&:=\!-\!\int_{0}^{x}\!\!\!\theta(1-x)\mbf{v}(\theta,y)d\theta -\int_{x}^{1}\!\!\!x(1-\theta)\mbf{v}(\theta,y)d\theta,	\\
		(\mcl{T}_{2}\mbf{v})(x,y)
		&:=-\int_{0}^{y}\eta\mbf{v}(x,\eta)d\eta -\int_{y}^{1}y\mbf{v}(x,\eta)d\eta.
	\end{align*}
	By Cor.~\ref{cor:Tmap_DN_2}, it follows that also $\mbf{u}_{yy}=\mcl{T}_{1}\mbf{u}_{xxyy}$, $\mbf{u}_{xx}=\mcl{T}_{2}\mbf{u}_{xxyy}$, and $\mbf{u}=\mcl{T}\mbf{u}_{xxyy}$, where
	\begin{equation*}
		(\mcl{T}\mbf{v})(x,y):=(\mcl{T}_{1}(\mcl{T}_{2}\mbf{v}))(x,y)=(\mcl{T}_{2}(\mcl{T}_{1}\mbf{v}))(x,y).
	\end{equation*}
	Conversely, for any $\mbf{v}\in L_{2}[[0,1]^2]$, we have $\mbf{v}=\partial_{x}^2\mcl{T}_{1}\mbf{v}$, $\mbf{v}=\partial_{y}^2\mcl{T}_{2}\mbf{v}$, and $\mbf{v}=\partial_{x}^{2}\partial_{y}^{2}\mcl{T}\mbf{v}$. 
	By Thm.~\ref{thm:PDE2PIE}, then, $\mbf{u}$ solves the PDE~\eqref{eq:example_heat_eq} if and only if $\mbf{v}=\mbf{u}_{xxyy}$ solves the PIE
	\begin{equation*}
		\partial_{t}\mcl{T}\mbf{v}(t)=\mcl{A}\mbf{v}(t)
		:=\mcl{T}_{2}\mbf{v}(t) +\mcl{T}_{1}\mbf{v}(t) +r\mcl{T}\mbf{v}(t).
	\end{equation*}
\end{example}

Example~\ref{ex:BCs_DN_2} illustrates how, using Thm.~\ref{thm:PDE2PIE}, a linear multivariate PDE with admissible and consistent domains can be equivalently represented as a PIE. Using this equivalence, stability properties of the PDE can be inferred from properties of the associated PIE. Verification of such stability properties using the PIE representation is significantly easier, since the fundamental state $\mbf{v}(t)$ of the PIE lies in the Hilbert space $L_{2}^{n}$, obviating the need for ad hoc manipulations to account for boundary conditions. 
In addition, as will be shown in the following section, the operators $\mcl{T},\mcl{A}$ defining the PIE representation belong to a $*$-algebra of Partial Integral (PI) operators, 
being closed under composition and adjoint operations.
Using these properties, stability of the PDE can be tested in the PIE representation by solving a convex optimization program, as will be shown in Section~\ref{sec:stability}.

\section{A $*$-Algebra of Multivariate PI Operators}\label{sec:PIs}

Having shown how a broad class of multivariate, linear PDEs can be equivalently represented as PIEs, we now show that this PIE representation retains many elements of the classical, matrix-parameterized, state-space representation of linear ordinary differential equations.
In particular, we show that the parameters $\mcl{T},\mcl{A}$ defining the PIE representation may be embedded within a parameterized $*$-algebra of integral operators---akin to how the parameters defining linear ordinary differential equations belong to the $*$-algebra of matrices.

Let us begin by recalling that, in constructing the operators defining the PIE representation of a given multivariate PDE, an inductive approach was used, defining each multivariate operator as the composition (product) and sum of univariate operators: $\mcl{T}:=\prod_{i=1}^{N}\mcl{T}_{i}$ and $\mcl{A}:=\sum_{\vec{0}\leq\alpha\leq\delta}\tnf{M}[\mbs{A}_{\alpha}]\,\prod_{i=1}^{N}\mcl{A}_{i,\alpha_{i}}$.
Here, each of the univariate operators $\mcl{T}_{i}$ and $\mcl{A}_{i,\alpha_{i}}$ is of the form $(\mcl{R}\mbf{v})(s)=\mbs{R}_{0}(s)\mbf{v}(s)+\int_{a}^{s}\mbs{R}_{1}(s,\theta)\mbf{v}(\theta)d\theta +\int_{s}^{b}\mbs{R}_{2}(s,\theta)\mbf{v}(\theta)d\theta$, for polynomial $\mbs{R}_{0},\mbs{R}_{1},\mbs{R}_{2}$. Such univariate operators have been referred to as 3-PI operators in, e.g.~\cite{shivakumar2022GPDE_Arxiv} (see also Defn.~\ref{defn:PI_1D}), and have been used to parameterize a class of 1D PIEs. Moreover, it has been shown that these univariate 3-PI operators form a $*$-algebra---being closed under summation, composition, and adjoint operations.

In this section, we show that in the multivariate setting, the operators $\mcl{T}$ and $\mcl{A}$ defining the PIE representation of a PDE similarly belong to a $*$-algebra of PI operators. In particular, in Subsection~\ref{subsec:PIs:PI_ops}, based on the inductive definition of $\mcl{T}$ and $\mcl{A}$, we define a class of $N$D PI operators as linear combinations of compositions of univariate 3-PI operators along each spatial direction. In Subsection~\ref{subsec:PIs:algebra}, we then prove that this class of operators is closed under compositions and adjoints, providing explicit formulae for performing these operations.

\subsection{A Parameterization of Multivariate 3-PI Operators}\label{subsec:PIs:PI_ops}

In this subsection, we define a parameterized class of multivariate PI operators on $L_{2}[\Omega]$ which includes the operators $\mcl T, \mcl{A}$, as constructed in Thm.~\ref{thm:PDE2PIE}. 
To define such multivariate PI operators, we first recall the following definition of univariate 3-PI operators from, e.g.~\cite{shivakumar2022GPDE_Arxiv}. 
\begin{defn}[Univariate 3-PI operators, $\PIset_{1}{[a,b]}$]\label{defn:PI_1D}
	For $[a,b]\subset\R$, we say that $\mcl{P}\in \mcl L(L_2^{n}[a,b],L_2^{m}[a,b])$ is a \textbf{univariate 3-PI operator} (denoted  $\mcl{P}\in\PIset_{1}^{m\times n}[a,b]$) if there exist $\mbs{P}_{0}\in \R^{m\times n}[s]$ and $\mbs{P}_{1},\mbs{P}_{2}\in \R^{m\times n}[s,\theta]$ such that
	\begin{equation*}\Resize{\linewidth}{
			(\mcl{P}\mbf{v})(s)
			\!=\!\mbs{P}_{0}(s)\mbf{v}(s)\! +\!\int_{a}^{s}\!\!\!\mbs{P}_{1}(s,\theta)\mbf{v}(\theta)d\theta \!+\!\int_{s}^{b}\!\!\!\mbs{P}_{2}(s,\theta)\mbf{v}(\theta)d\theta.}
	\end{equation*}
	We say that $\mcl{P}$ is defined by parameters $\{\mbs{P}_{0},\mbs{P}_{1},\mbs{P}_{2}\}\in \PIparam_{1}^{m\times n}[a,b]:= \R^{m\times n}[s]\times\R^{m\times n}[s,\theta]\times \R^{m\times n}[s,\theta]$.
\end{defn}

The class of univariate 3-PI operators includes both multiplier operators with polynomial multipliers and integral operators with polynomial semi-separable kernels on $L_{2}[a,b]$, where we say that a kernel, $\mbs{K}$, is \textit{polynomial semi-separable} if it may be expressed as
\begin{equation*}
	\mbs{K}(s,\theta):=
	\begin{cases}
		\mbs{P}_{1}(s,\theta),	&	s\geq\theta,	\\
		\mbs{P}_{2}(s,\theta),	&	s<\theta,
	\end{cases}
	\quad \text{for}~ \mbs{P}_{1},\mbs{P}_{2}\in \R^{m\times n}[s,\theta].
\end{equation*}
The restriction to polynomial parameters will allow us in Section~\ref{sec:stability} to verify properties of 3-PI operators using algorithms based on the monomial coefficients of these polynomial parameters. Note that, by inspection, the operators $\mcl{T}$ and $\mcl{A}$ defining the univariate PIE representation in Cor.~\ref{cor:gohberg_PIE} are univariate 3-PI operators.

We now define the class of $N$D 3-PI operators, by taking sums of compositions of univariate 3-PI operators.  
\begin{defn}[$N$D 3-PI operators, {$\PIset_{N}[\Omega]$}]\label{defn:PI_ND}
	Given $N\in\N$, $\Omega=\prod_{i=1}^{N}[a_{i},b_{i}]$, we say that $\mcl{P} \in \mcl L(L_2[\Omega])$ is a \textbf{$\boldsymbol{N}$D 3-PI operator} (denoted  $\mcl{P}\in\PIset_{N}[\Omega]$) if there exist $M\!\in\N$ and $\mcl P_{j,i}\!\in \PIset_{1}[a_i,\!b_i]$ such that $\mcl{P}=\!\sum_{j=1}^{M}\! \prod_{i=1}^{N}\!\mcl P_{j,i}$,
	where each univariate $\mcl P_{j,i}$ operates on variable $s_i$---i.e. 
	\begin{equation*}
		(\mcl P_{j,i}\mbf v)(s):=\bl(\mcl{P}_{j,i}\mbf v(s_1,\cdots,s_{i-1},\bullet,s_{i+1},\cdots s_N)\br)(s_i).
	\end{equation*}
	For matrix-valued operators, $\mcl{P} \in \mcl L(L_2^{n}[\Omega],L_2^{m}[\Omega])$, we say that $\mcl{P}$ is an $N$D 3-PI operator (denoted  $\mcl{P}\in\PIset_{N}^{m \times n}[\Omega]$) if $[\mcl{P}]_{k,\ell}\in \PIset_{N}[\Omega]$ for all $k\in\{1:m\}$ and $\ell\in\{1:n\}$.
\end{defn}

Defn.~\ref{defn:PI_ND} offers a natural parameterization of multivariate 3-PI operators, by taking compositions of univariate 3-PI operators along different directions---i.e. $\mcl{P}=\prod_{i=1}^{N}\mcl{P}_{i}$. Furthermore, by allowing summation of such operators, the class $\PIset_{N}$ also includes, e.g., any operator of the form
$(\mcl{P}\mbf{v})(s):=\mbs M(s)\mbf{v}(s)+\int_{\Omega}\mbs{K}(s,\theta)\mbf{v}(\theta)d\theta$,
for $\mbs{M}$ polynomial and $\mbs{K}$ polynomial semi-separable, 
where the definition of polynomial semi-separable kernels may be extended to multivariate domains as follows.
\begin{defn}\label{defn:semisep}
	For $N\in\N$ and $\Omega:=\prod_{i=1}^{N}[a_{i},b_{i}]$, we say that $\mbs{K}\in L_{2}^{m\times n}[\Omega\times\Omega]$ is \tbf{polynomial semi-separable} if for each $\alpha\in\{-1,1\}^{N}$ there exists polynomial $\mbs{K}_{\alpha}\in\R^{m\times n}[s,\theta]$ such that $\mbs{K}(s,\theta)=\mbs{K}_{\alpha}(s,\theta)$ for all $s,\theta\in\Omega$ for which $\alpha_{i}\cdot (s_{i}-\theta_{i})\leq 0$ for every $i\in\{1:N\}$.
\end{defn}

We may now observe that, given a multivariate PDE as in~\eqref{eq:PDE_standard}, the operators $\mcl{T}$ and $\mcl{A}$ defining the associated PIE representation in Thm.~\ref{thm:Tmap} are $N$D 3-PI operators. 

\begin{lem}\label{lem:PI_op_T}
	For $N\in\N$ and $\delta\in\N_{0}^{N}$, let $\mcl{T}$ and $\mcl{A}$ be as defined in Thm.~\ref{thm:Tmap}. Then $\mcl{T},\mcl{A}\in\PIset_{N}^{n\times n}[\Omega]$.
\end{lem}

The proof follows directly using the inductive definition of multivariate 3-PI operators 
(see Lem.~\ref{lem:PI_op_T_appx} in Appx.~\ref{appx:PIs} for a formal proof).
The following example illustrates the parameterization of multivariate 3-PI operators by examining the $\mcl{T}$ and $\mcl{A}$ as obtained from the PIE representation of the 2D PDE from Example~\ref{ex:BCs_DN_2}.
\begin{example}\label{ex:BCs_DN_3}
	Consider the reaction-diffusion equation
	\begin{align*}
		\mbf{u}_{t}(t,x,y)&=\mbf{u}_{xx}(t,x,y)+\mbf{u}_{yy}(t,x,y)+\mbf{u}(t,x,y),	\\
		\mbf{u}(t,0,y)&=\mbf{u}(t,1,y)=0, \quad \mbf{u}(t,x,0)=\mbf{u}_{y}(t,x,1)=0, 
	\end{align*}
	where $\mbf{u}(t)\in S_{2}^{\vec{2}}[[0,1]^2]$.
	Defining $\mcl{T}_{1},\mcl{T}_{2}$ as in Example~\ref{ex:BCs_DN_2}, we can equivalently represent the PDE as a PIE, $\partial_{t}\mcl{T}\mbf{v}(t)=\mcl{A}\mbf{v}(t)$, where $\mcl{T}=\mcl{T}_{1}\mcl{T}_{2}$ and $\mcl{A}=\!\mcl{T}_{2}+\mcl{T}_{1}+\mcl{T}$. Here, for each $i\in\{1,2\}$, the operator $\mcl{T}_{i}\in\PIset_{1}[0,1]$ is defined by parameters $\mbf{T}_{i}=\{0,\mbs{T}_{i}^{1},\mbs{T}_{i}^{2}\}\in\PIparam_{1}[0,1]$, where
	\begin{align*}
		\mbs{T}_{1}^{1}(x,\theta)&:=-\theta(1-x),	&
		\mbs{T}_{1}^{2}(x,\theta)&:=-x(1-\theta),		\\
		\mbs{T}_{2}^{1}(y,\eta)&:=-\eta,				&
		\mbs{T}_{2}^{2}(y,\eta)&:=-y.
	\end{align*}
	It follows that $\mcl{T},\mcl{A}\in\PIset_{2}[[0,1]^2]$. 
\end{example}

Having now defined a suitable class of $N$D 3-PI operators, in the following subsection we show that this class is a $*$-algebra, providing explicit expressions for composition and adjoint in terms of the parameters defining the operators.

\subsection{$N$D 3-PI Operators Form a $*$-Algebra}\label{subsec:PIs:algebra}

A $*$-algebra is a vector space equipped with a multiplication and involution operation which satisfy certain properties. It is well-known that both the space of multiplier operators with bounded multipliers and that of integral operators with square-integrable kernels are $*$-algebras, using composition ($\circ$) for the multiplication operation, and the adjoint ($*$) with respect to the $L_{2}$ inner product as involution operation. Since 3-PI operators are defined by (bounded) multiplier and integral operators, we can similarly define a multiplication and involution operation on $\PIset_{N}^{n\times n}[\Omega]$ as the composition and adjoint, respectively, so that for $\mcl{Q},\mcl{R}\in\PIset_{N}^{n\times n}[\Omega]$ we define $(\mcl{Q}\circ\mcl{R})\mbf{v}:=\mcl{Q}(\mcl{R}\mbf{v})$ and $\mcl{P}=\mcl{R}^*$ if $\ip{\mcl{P}\mbf{u}}{\mbf{v}}_{L_{2}}=\ip{\mbf{u}}{\mcl{R}\mbf{v}}_{L_{2}}$
for all $\mbf{u},\mbf{v}\in L_{2}^{n}$. In order to prove that $\PIset_{N}^{n\times n}[\Omega]$ is a $*$-algebra, then, we need only show that $\PIset_{N}^{n\times n}$ is a vector space which is closed under composition and adjoint operations. To reduce notational complexity, we will prove this fact only for the scalar-valued case, $n=1$, noting that this result readily generalizes to arbitrary $n\in\N$ using the standard matrix product and transpose operations.

To start, we note that (by definition) $\PIset_{N}$ is indeed a vector space, as shown in the following lemma.

\begin{lem}\label{lem:PI_sum_ND}
	For $N\in\N$ and $\Omega:=\prod_{i=1}^{N}[a_{i},b_{i}]$, if $\mcl{Q},\mcl{R}\in\PIset_{N}[\Omega]$, then $\lambda\mcl{Q}+\mu\mcl{R}\in\PIset_{N}[\Omega]$ for all $\lambda,\mu\in\R$.
\end{lem}
\begin{proof}
	The result follows from the definition of $\PIset_{N}$, and linearity of integral and multiplier operators. A full proof is given in Lem.~\ref{lem:PI_sum_ND_appx} in Appx.~\ref{appx:PIs}.
\end{proof}

Having established that the set $\PIset_{N}$ is a vector space, it remains to prove that it is closed under composition and adjoint operations. However, these properties readily follow from the inductive definition of multivariate 3-PI operators in terms of univariate 3-PI operators---the latter of which have already been proven to form a $*$-algebra. 
Specifically, recall the following result from~\cite{shivakumar2022GPDE_Arxiv}, proving that the composition of two univariate 3-PI operators is again a univariate 3-PI operator.
\begin{lem}\label{lem:PI_composition_1D}
	For $[a,b]\subseteq\R$,
	if $\mcl{Q},\mcl{R}\in\PIset_{1}[a,b]$, then $\mcl{P}:=\mcl{Q}\circ\mcl{R}\in\PIset_{1}[a,b]$. In particular, if $\mcl{Q}$ and $\mcl{R}$ are defined by parameters $\{\mbs{Q}_{0},\mbs{Q}_{1},\mbs{Q}_{2}\}\in\PIparam_{1}[a,b]$ and $\{\mbs{R}_{0},\mbs{R}_{1},\mbs{R}_{2}\}\in\PIparam_{1}[a,b]$, respectively, then $\mcl{P}$ is defined by parameters $\{\mbs{P}_{0},\mbs{P}_{1},\mbs{P}_{2}\}\in\PIparam_{1}[a,b]$, where $\mbs{P}_{0}(s):=\mbs{Q}_{0}(s)\mbs{R}_{0}(s)$ and
	\begin{align*}
		&~\mbs{P}_{1}(s,\theta):=\mbs{Q}_{0}(s)\mbs{R}_{1}(s,\theta) +\mbs{Q}_{1}(s,\theta)\mbs{R}_{0}(\theta) \\[-0.2em] &\qquad +{\textstyle\int_{a}^{\theta}}\mbs{Q}_{1}(s,\eta)\mbs{R}_{2}(\eta,\theta)d\eta +{\textstyle\int_{\theta}^{s}}\mbs{Q}_{1}(s,\eta)\mbs{R}_{1}(\eta,\theta)d\eta	\\
		&\qquad\qquad+{\textstyle\int_{s}^{b}}\mbs{Q}_{2}(s,\eta)\mbs{R}_{1}(\eta,\theta)d\eta, \\
		&~\mbs{P}_{2}(s,\theta):=\mbs{Q}_{0}(s)\mbs{R}_{2}(s,\theta) +\mbs{Q}_{2}(s,\theta)\mbs{R}_{0}(\theta) \\[-0.2em] &\qquad +{\textstyle\int_{a}^{s}}\mbs{Q}_{1}(s,\eta)\mbs{R}_{2}(\eta,\theta)d\eta +{\textstyle\int_{s}^{\theta}}\mbs{Q}_{2}(s,\eta)\mbs{R}_{2}(\eta,\theta)d\eta	\\
		&\qquad\qquad+{\textstyle\int_{\theta}^{b}}\mbs{Q}_{2}(s,\eta)\mbs{R}_{1}(\eta,\theta)d\eta.\\[-1.8\baselineskip]
	\end{align*}
\end{lem}
\begin{proof}
	We refer to Lem. 36 in~\cite{shivakumar2022GPDE_Arxiv} for a proof.
\end{proof}

Lem.~\ref{lem:PI_composition_1D} shows that the composition of univariate 3-PI operators again defines a univariate 3-PI operator, providing an explicit expression for the parameters defining this composition. Using this expression, and the inductive definition of multivariate PI operators, we then obtain the following result for the composition of $N$D 3-PI operators.
\begin{prop}[Composition of $N$D PI Operators]\label{prop:PI_composition_ND}
	For any $N\in\N$ and $\Omega=\prod_{i=1}^{N}[a_{i},b_{i}]$, if $\mcl{Q},\mcl{R}\in\PIset_{N}[\Omega]$, then $\mcl{P}:=\mcl{Q}\circ\mcl{R}\in\PIset_{N}[\Omega]$. In particular, if $\mcl{Q}=\sum_{j=1}^{M}\prod_{i=1}^{N}\mcl{Q}_{j,i}$ and $\mcl{R}=\sum_{k=1}^{K}\prod_{i=1}^{N}\mcl{R}_{k,i}$ where $\mcl{Q}_{j,i},\mcl{R}_{k,i}\in\PIset_{1}[a_{i},b_{i}]$ for $i\in\{1:N\}$, $j\in\{1:M\}$ and $k\in\{1:K\}$, then $\mcl{P}=\sum_{j=1}^{M}\sum_{k=1}^{K}\prod_{i=1}^{N}\mcl{Q}_{j,i}\circ\mcl{R}_{k,i}$, where the parameters defining $\mcl{Q}_{j,i}\circ\mcl{R}_{k,i}\in\PIset_{1}[a_{i},b_{i}]$ are given in Lem.~\ref{lem:PI_composition_1D}.
\end{prop}
\begin{proof}
	We prove the result for $M=K=1$, noting that the extension to arbitrary $M,K\in\N$ then follows by Lem.~\ref{lem:PI_sum_ND}.
	Let $\mcl{Q}=\prod_{i=1}^{N}\mcl{Q}_{i}$ and $\mcl{R}=\prod_{i=1}^{N}\mcl{R}_{i}$ for $\mcl{Q}_{i},\mcl{R}_{i}\in\PIset_{1}[a_{i},b_{i}]$ for $i\in\{1:N\}$. Then, for each $i,\ell\in\{1:N\}$, $\mcl{Q}_{i}$ operates only on variable $s_{i}$, and $\mcl{R}_{\ell}$ only on variable $s_{\ell}$. Since the operators are defined by scalar-valued parameters, this implies $\mcl{Q}_{i}\circ\mcl{R}_{\ell}=\mcl{R}_{\ell}\circ\mcl{Q}_{i}$ whenever $i\neq \ell$ (see Lem.~\ref{lem:PI_prod_commute} in Appx.~\ref{appx:PIs} for a formal proof). Using this fact, it follows that
	\begin{equation*}
		\prod_{i=1}^{N}\mcl{Q}_{i}\circ\prod_{\ell=1}^{N}\mcl{R}_{\ell}=\prod_{i=1}^{N}\mcl{Q}_{i}\circ\mcl{R}_{i},
	\end{equation*}
	(see Cor.~\ref{cor:PI_prod_commute} in Appx.~\ref{appx:PIs} for a formal proof). Letting $\mcl{P}_{i}:=\mcl{Q}_{i}\circ\mcl{R}_{i}$ for each $i\in\{1:N\}$, we find $\mcl{Q}\circ\mcl{R}=\prod_{i=1}^{N}\mcl{P}_{i}$. Since, by Lem.~\ref{lem:PI_composition_1D}, $\mcl{P}_{i}\in \PIset_{1}[a_{i},b_{i}]$, we conclude that $\mcl{P}\in\PIset_{N}[\Omega]$.
\end{proof}

Prop.~\ref{prop:PI_composition_ND} shows that, for any $\mcl{Q},\mcl{R}\in\PIset_{N}$, we can define $\mcl{P}\in\PIset_{N}$ such that $\mcl{P}=\mcl{Q}\circ\mcl{R}$---proving that $\PIset_{N}$ is closed under compositions.
It remains only to prove, then, that $\PIset_{N}$ is also closed under the adjoint on $L_{2}$. This is relatively trivial, based on the following result for the adjoint of univariate 3-PI operators from e.g.~\cite{shivakumar2022GPDE_Arxiv}. 
\begin{lem}\label{lem:PI_adjoint_1D}
	For $[a,b]\subseteq\R$, if $\mcl{R}\in\PIset_{1}[a,b]$, then $\mcl{R}^*\in\PIset_{1}[a,b]$. In particular, if $\mcl{R}$ is defined by parameters $\{\mbs{R}_{0},\mbs{R}_{1},\mbs{R}_{2}\}\in\PIparam_{1}[a,b]$, then $\mcl{R}^*$ is defined by parameters $\{\mbs{P}_{0},\mbs{P}_{1},\mbs{P}_{2}\}\in\PIparam_{1}[a,b]$, where
	$\mbs{P}_{0}(s):=\mbs{R}_{0}(s)$, $\mbs{P}_{1}(s,\theta):=\mbs{R}_{2}(\theta,s)$ and $\mbs{P}_{2}(s,\theta):=\mbs{R}_{1}(\theta,s)$.
\end{lem}
\begin{proof}
	We refer to Lem.~37 in~\cite{shivakumar2022GPDE_Arxiv} for a proof.
\end{proof}

Using the expression for the adjoint of a 1D 3-PI operator from Lem.~\ref{lem:PI_adjoint_1D}, the following lemma shows how we may similarly compute the adjoint of an $N$D 3-PI operator, proving that this adjoint is an $N$D 3-PI operator as well.
\begin{lem}\label{lem:PI_adjoint_ND}
	For $N\in\N$, $\Omega:=\prod_{i=1}^{N}[a_{i},b_{i}]$, if $\mcl{R}\in\PIset_{N}[\Omega]$, then $\mcl{R}^*\in\PIset_{N}[\Omega]$. In particular, if $\mcl{R}=\sum_{j=1}^{M}\prod_{i=1}^{N}\mcl{R}_{j,i}$ for some $\mcl{R}_{j,i}\in\PIset_{1}[a_{i},b_{i}]$, then $\mcl{R}^*=\sum_{j=1}^{M}\prod_{i=1}^{N}\mcl{R}_{j,i}^*$, where the parameters defining $\mcl{R}_{j,i}^*$ are given in Lem.~\ref{lem:PI_adjoint_1D}.
\end{lem}
\begin{proof}
	The result follows immediately from the definition of multivariate 3-PI operators, and the commutative properties of univariate 3-PI operators. 
	A full proof is given in Lem.~\ref{lem:PI_adjoint_ND_appx} in Appx.~\ref{appx:PIs}.
\end{proof}

By Lem.~\ref{lem:PI_adjoint_ND}, the adjoint of any multivariate 3-PI operator is again a multivariate 3-PI operator. Using this result, as well as Prop.~\ref{prop:PI_composition_ND} and Lem.~\ref{lem:PI_sum_ND}, it then follows that $\PIset_{N}[\Omega]$ in fact defines a $*$-algebra.

\begin{prop}\label{prop:*algebra}
	For any $N\in\N$ and $\Omega:=\prod_{i=1}^{N}[a_{i},b_{i}]$, $\PIset_{N}[\Omega]$ with multiplication defined by composition and involution defined by the adjoint on $L_{2}$, is a $*$-algebra.
\end{prop}
\begin{proof}
	To prove this result, we use the fact that $\mcl{L}(L_{2}[\Omega])$ is a $*$-algebra, with multiplication defined by composition, and involution defined by the adjoint on $L_{2}$. Since $\PIset_{N}[\Omega]\subseteq \mcl{L}(L_{2}[\Omega])$, the composition and adjoint on $L_{2}$ then also define suitable multiplication and involution operations, respectively, on $\PIset_{N}[\Omega]$. Since $\PIset_{N}[\Omega]$ is a vector space (by Lem.~\ref{lem:PI_sum_ND}) which is closed under composition and adjoint operations (by Prop.~\ref{prop:PI_composition_ND} and Lem.~\ref{lem:PI_adjoint_ND}, respectively), it follows that $\PIset_{N}[\Omega]$ is a $*$-algebra.
\end{proof}

Prop.~\ref{prop:*algebra} proves that the set $\PIset_{N}[\Omega]$ of $N$D 3-PI operators is a $*$-algebra, being closed under linear combinations, composition, and adjoint. This result also generalizes directly to $\PIset_{N}^{n\times n}$, using the standard formulae for matrix multiplication and transposition. In the following section, we will use the $*$-algebraic properties of $N$D 3-PI operators to show how stability of linear $N$D PDEs can be tested in the PIE representation using convex optimization.

\section{Stability Analysis in the PIE Representation using Semidefinite Programming}\label{sec:stability}

Having shown that any sufficiently well-posed linear $N$D PDE as in~\eqref{eq:PDE_standard} admits an equivalent PIE representation of the form $\partial_{t}\mcl{T}\mbf{v}= \mcl{A}\mbf{v}$ where $\mbf v \in L_2$, and having shown that the parameters $\mcl T,\mcl{A}$ lie in the $*$-algebra of bounded linear multivariate 3-PI operators, we now exploit this representation to obtain stability conditions which may be tested using convex optimization. Specifically, in Subsection~\ref{subsec:stability:LPI}, we will define a suitable notion of exponential stability and derive a sufficient condition for this stability as existence of a PI operator $\mcl P$ such that
\begin{equation*} 
	\mcl{P}^*\mcl{T}=\mcl{T}^*\mcl{P}\succeq \epsilon^2 \mcl{T}^*\mcl{T},\qquad
	\mcl{P}^*\mcl{A}+\mcl{A}^*\mcl{P}\preceq -2k\mcl{P}^*\mcl{T}, 
\end{equation*}
for some $\epsilon,k \ge 0$. Inequalities of this form, expressed in terms of PI operators and operator inequalities, are referred to as Linear PI Inequalities (LPIs). In Subsection~\ref{subsec:stability:SDP}, we then show how this LPI can be solved with semidefinite programming, by parameterizing positive semidefinite multivariate PI operators by positive semidefinite matrices.

\subsection{An LPI for Exponential Stability of PDEs}\label{subsec:stability:LPI}

To construct a PIE-based stability test for linear multivariate PDEs, we first define a notion of exponential stability. Specifically, we use the notion of \textit{exponential PIE to PDE stability}, which is slightly weaker than the classical notion of exponential $L_2$ stability. 
This weaker definition is used because the classical notion of exponential stability, when applied to PDEs in first-order form (i.e. $\partial_{t}\mbf{u}=\mscr{H} \mbf{u}$), does not account for certain notions of energy storage (e.g. strain) and is not satisfied for many commonly used PDEs such as wave equations (see e.g. Appx.~C.2).

\begin{defn}\label{defn:exponential_stability}
	A PDE defined by parameters $\{\mbs{A}_{\alpha},\mcl{D}_{i}\}$ as in~\eqref{eq:PDE_standard}, admitting a unique solution that depends continuously on the initial state, is said to be \textbf{exponentially PIE to PDE stable} with rate $k\geq 0$ and gain $M>0$ if for every initial state $\mbf{u}_{0}\in \mcl{D}\subseteq S_{2}^{\delta,n}$ and all $t\geq 0$, the associated solution $\mbf{u}(t)$ satisfies $\norm{\mbf{u}(t)}_{L_{2}}\leq Me^{-kt}\|D^{\delta}\mbf{u}_{0}\|_{L_{2}}$ .
\end{defn}

In contrast to classical notions of exponential stability, Defn.~\ref{defn:exponential_stability} bounds the solution in terms of the $L_2$ norm of the initial fundamental state, $\|D^\delta \mbf{u}_0\|_{L_2}$, rather than the initial PDE state, $\|\mbf{u}_{0}\|_{L_2}$. Since the $L_2$ norm of the PDE state is bounded by a constant multiple of the $L_{2}$ norm of the fundamental state, the classical notion of exponential stability then implies exponential PIE-to-PDE stability, but the converse is not true.

Having defined a suitable notion of exponential stability, suppose now we want to verify this stability property for a linear PDE as in~\eqref{eq:PDE_standard}, with an associated PIE representation of the form $\partial_{t}\mcl{T}\mbf{v}(t)=\mcl{A}\mbf{v}(t)$. Our stability test is based on existence of a Lyapunov functional $V:L_{2}^{n}\to \R$ such that $\frac{d}{dt}V(\mbf{v}(t))\leq -2kV(\mbf{v}(t))$ for any $\mbf v(t)$ which satisfies the PIE---implying $V(\mbf{v}(t))\leq e^{-2kt}V(\mbf{v}(0))$. The distinction between $L_2$ and PIE to PDE stability, then, lies only in the nature of the \textit{upper} bound on $V$---i.e. $V(\mbf{v})\leq C\norm{\mbf{v}}_{L_{2}}$ as opposed to $V(\mbf{v})\leq C\norm{\mcl T\mbf{v}}_{L_{2}}$. 
The following theorem shows how existence of such a functional $V$ can be tested by solving an LPI.

\begin{thm}\label{thm:LPI_stability}
	For $\{\mbs{A}_{\alpha},\mcl{D}_{i}\}$ defining a PDE as in~\eqref{eq:PDE_standard}, let associated operators $\mcl{T},\mcl{A}\in \PIset_{N}^{n\times n}$ be as in Thm.~\ref{thm:PDE2PIE}. For $\epsilon>0$ and $k\geq 0$, if there exists $\mcl{P}\in\PIset_{N}^{n\times n}$ such that
	\begin{equation}\label{eq:LPI_stability}
		\mcl{P}^*\mcl{T}=\mcl{T}^*\mcl{P}\succeq \epsilon^2 \mcl{T}^*\mcl{T},\qquad
		\mcl{P}^*\mcl{A}+\mcl{A}^*\mcl{P}\preceq -2k\mcl{P}^*\mcl{T},
	\end{equation}
	then the PDE defined by $\{\mbs{A}_{\alpha},\mcl{D}_{i}\}$ is exponentially PIE to PDE stable with rate $k$ and gain $M:=\sqrt{\|\mcl{P}^*\mcl{T}\|_{\tnf{op}}}/\epsilon$.
\end{thm}
\begin{proof}
	Let $\mcl{P}$ be such that~\eqref{eq:LPI_stability} is satisfied, and consider the functional $V: L_{2}^{n}[\Omega]\to\R$ defined by $V(\mbf{v}):=\ip{\mbf{v}}{\mcl{P}^*\mcl{T}\mbf{v}}_{L_{2}}$. Since $\mcl{P}^*\mcl{T}\succeq\epsilon^2 \mcl{T}^*\mcl{T}$, this functional satisfies
	\begin{equation*}
		\epsilon^2\norm{\mcl{T}\mbf{v}}_{L_{2}}^2
		\leq V(\mbf{v})
		\leq \|\mcl{P}^*\mcl{T}\|_{\tnf{op}} \norm{\mbf{v}}_{L_{2}}^2.
	\end{equation*}
	Now, let $\mbf{u}_{0}\in\bigcap_{i=1}^{N}\mcl{D}_{i}\subseteq S_{2}^{\delta,n}[\Omega]$ be an arbitrary initial state, and let $\mbf{u}$ be the associated solution to the PDE defined by $\{\mbs{A}_{\alpha},\mcl{D}_{i}\}$. Let $\mbf{v}=D^{\delta}\mbf{u}$. Then, by  Thm.~\ref{thm:PDE2PIE}, $\mbf{u}(t)=\mcl{T}\mbf{v}(t)$ and $\partial_{t}\mcl{T} \mbf{v}(t)=\mcl{A}\mbf v(t)$ for all $t\ge 0$, with $\mbf v(0)=D^{\delta}\mbf{u}_{0}$. Taking the temporal derivative of $V(\mbf{v}(t))$ and exploiting the algebraic properties of $\PIset_N$, we find
	\begin{align*}
		\frac{d}{dt}V(\mbf{v}(t))
		&=\ip{\mbf{v}(t)}{\partial_{t}\mcl{P}^*\mcl{T}\mbf{v}(t)}_{L_{2}} +\ip{\partial_{t}\mcl{P}^*\mcl{T}\mbf{v}(t)}{\mbf{v}(t)}_{L_{2}}	\\
		&=\ip{\mbf{v}(t)}{\mcl{P}^*\mcl{A}\mbf{v}(t)}_{L_{2}} +\ip{\mcl{P}^*\mcl{A}\mbf{v}(t)}{\mbf{v}(t)}_{L_{2}}	\\
		&=\ip{\mbf{v}(t)}{[\mcl{P}^*\mcl{A}+\mcl{A}^*\mcl{P}]\mbf{v}(t)}_{L_{2}}	\\
		&\leq -2k\ip{\mbf{v}(t)}{\mcl{P}^*\mcl T\mbf{v}(t)}_{L_{2}}
		=-2kV(\mbf{v}(t)).
	\end{align*}
	By Gr\"onwall-Bellman inequality, it follows that $V(\mbf{v}(t))\leq e^{-2k t}V(\mbf{v}(0))$, and therefore, for all $t\geq 0$,
	\begin{align*}
		&\norm{\mbf{u}(t)}_{L_{2}}
		=\norm{\mcl{T}\mbf{v}(t)}_{L_{2}}
		\leq \frac{1}{\epsilon}\sqrt{V(\mbf{v}(t))}		
		\leq \frac{1}{\epsilon} e^{-k t}\sqrt{V(\mbf{v}(0))}	\\
		&\hspace*{0.85cm}\leq \frac{\sqrt{\|\mcl{P}^*\mcl{T}\|_{\tnf{op}}}}{\epsilon}e^{-k t}\norm{\mbf{v}(0)}_{L_{2}}
		=Me^{-kt}\norm{D^{\delta}\mbf{u}_{0}}_{L_{2}}.
	\end{align*}
	Since this holds for all $\mbf{u}_{0}$, 
	we conclude that the PDE is exponentially PIE to PDE stable.
\end{proof}

Thm.~\ref{thm:LPI_stability} proves that exponential PIE to PDE stability of any PDE of the form in~\eqref{eq:PDE_standard} can be tested using the associated PIE representation, by testing for existence of a solution $\mcl{P}$ to the LPI in~\eqref{eq:LPI_stability}. 
In the following subsection, we will show how this LPI can be solved numerically using semidefinite programming.

\subsection{Solving LPIs using Semidefinite Programming}\label{subsec:stability:SDP}

Having provided a test for exponential PIE to PDE stability as feasibility of an LPI, we now briefly show how sufficient conditions for feasibility of this LPI may be tested using semidefinite programming (SDP). This approach is based on the fact that for any positive semidefinite matrix, $R\succeq 0$, and any $\mcl{Z} \in \PIset_{N}[\Omega]$, we have $\mcl{R}:=\mcl Z^* \tnf{M}[R] \mcl Z \succeq 0$ (where $\tnf{M}[R]$ is just standard matrix multiplication extended to $L_2^n$). Furthermore, using the composition maps defined previously, there is a linear map from the elements of $R$ to the coefficients of the polynomials defining $\mcl{R}$. Thus, for given $\mcl{Z} \in \PIset_{N}[\Omega]$, if we can find an operator $\mcl{P}$ and matrices $Q,R\succeq 0$ satisfying
\begin{align*}
	\mcl{P}^*\mcl{T}=\mcl{T}^*\mcl{P},\qquad
	\mcl{T}^*\mcl{P} -\epsilon^2 \mcl{T}^*\mcl{T}&=\mcl Z^* \tnf{M}[R] \mcl Z,	\\
	\mcl{P}^*\mcl{A}+\mcl{A}^*\mcl{P}+2k\mcl{P}^*\mcl{T}&=-\mcl Z^* \tnf{M}[Q] \mcl Z,
\end{align*}
then we may conclude exponential PIE to PDE stability. Such equalities may then be enforced using SDP. All that remains, therefore, is to choose a suitable $\mcl Z$ so as to render the equality constraints feasible. Specifically, $\mcl{Z}$ is selected by defining a vector of basis operators for $\PIset_{N}$. This is done inductively by first defining a vector of basis operators for $\PIset_{1}[a_{i},b_{i}]$ for each $i$, and then constructing the multivariate Kronecker product of these bases. From~\cite{shivakumar2022GPDE_Arxiv}, we recall that such a basis for $\PIset_1[a,b]$ is constructed using the degree-$d$ monomial basis vector, $\mbf{z}_d$, yielding an associated basis of polynomial 3-PI parameters, $\mbf{Z}_{d}\in\PIparam_{1}^{\mu(d)\times 1}[a,b]$ for $\mu(d):=d^2+4d+3$, as
\begin{equation}\label{eq:Zparams_1D}
	\mbf{Z}_{d}:=
	\bbbbl\{\bbbbl[\scalemath{0.9}{\mat{\mbf{z}_{d}(s)\\[-0.5em]0\\[-0.5em]0}}\bbbbr],\bbbbl[\scalemath{0.9}{\mat{0\\[-0.5em]\mbf{z}_{d}(s,\theta)\\[-0.5em]0}}\bbbbr],\bbbbl[\scalemath{0.9}{\mat{0\\[-0.5em]0\\[-0.5em] \mbf{z}_{d}(s,\theta)}}\bbbbr]\bbbbr\}.
\end{equation}
Defining $\mcl{Z}_{d}\in\PIset_{1}^{\mu(d)\times 1}[a,b]$ by such parameters $\mbf{Z}_{d}$, any $\mcl{P}\in\PIset_{1}[a,b]$ with polynomial parameters of degree at most $d$ can then be represented as $\mcl{P}=c^T\,\mcl{Z}_{d}$ for a unique vector of coefficients $c\in\R^{\mu(d)}$. Given such bases $\mcl{Z}_{d,i}\in\PIset_{1}^{\mu(d)\times 1}[a_{i},b_{i}]$ of univariate operators for $i\in\{1:N\}$---defined by parameters $\mbf{Z}_{d,i}\in\PIparam_{1}^{\mu(d)\times 1}[a_{i},b_{i}]$ of the form in~\eqref{eq:Zparams_1D}---a basis for the space of multivariate 3-PI operators is then constructed as
\begin{equation}\label{eq:Zop}
	\mcl{Z}_{d}:=\mcl{Z}_{d,1}\otimes\cdots\otimes\mcl{Z}_{d,N}\in\PIset_{N}^{\mu(d)^N\times 1}[\Omega],
\end{equation}
so that any $\mcl{P}\in\PIset_{N}[\Omega]$ (defined by parameters of degree at most $d$ in each variable $s_{i}$) can be represented as $\mcl{P}=c^T\,\mcl{Z}_{d}$. Note that while $\mcl{Z}_d$ here defines a basis for the \textit{linear} representation of multivariate 3-PI operators, $\mcl{Z}_{d}^* \tnf{M}[R] \mcl{Z}_{d}$ is a \textit{quadratic} form. However, by the composition rules of multivariate 3-PI operators, $\mcl{Z}_{d}^*\tnf{M}[R]\mcl{Z}_{d}$ will be defined by polynomial parameters of degree at most $2d+1$ in each variable. 
Then, based on the monomial ordering of $\mbf{z}_d$, we can define an associated matrix $A$, such that for any 
matrix $R$,
we have $\mcl{Z}_{d}^*\tnf{M}[R]\mcl{Z}_{d}=\tnf{vec}(R)^TA\mcl{Z}_{2d+1}$, where $\tnf{vec}(R)$ is the vector of columns of $R$. In this way, $\mcl{Z}_d$ also defines a basis for the quadratic representation of $\PIset_{N}[\Omega]$. 
The following corollary uses this basis to tighten the stability LPI in Thm.~\ref{thm:LPI_stability} to an optimization problem that can be solved using SDP.

\begin{cor}\label{cor:LPI_stability_SDP}
	For given $\{\mbs{A}_{\alpha},\mcl{D}_{i}\}$ defining a PDE as in~\eqref{eq:PDE_standard}, let $\mcl{T},\mcl{A}\in \PIset_{N}^{n\times n}$ be as defined in Thm.~\ref{thm:PDE2PIE}. Define $\mcl{Z}_{d}$ as in~\eqref{eq:Zop}, with $\mu(d):=d^2+4d+3$. For any $d,d'\in\N_{0}$, $\epsilon>0$ and $k\geq 0$, if there exist $P\in\R^{n\times n\mu(d)^{N}}$ and $R,Q\succeq 0$ such that $\mcl{P}=\tnf{M}[P](I_{n}\otimes\mcl{Z}_{d})$ satisfies
	\begin{align}\label{eq:SDP_stability}
		&\mcl{P}^*\mcl{T}=\mcl{T}^*\mcl{P}=\epsilon ^2 \mcl T^* \mcl T+(I_{n}\!\otimes\!\mcl{Z}_{d'})^* \tnf{M}[R](I_{n}\!\otimes\!\mcl{Z}_{d'}),	\\
		&\mcl{P}^*\mcl{A}+\mcl{A}^*\mcl{P}+2k\mcl{P}^*\mcl{T}=-(I_{n}\!\otimes\!\mcl{Z}_{d'})^* \tnf{M}[Q](I_{n}\!\otimes\!\mcl{Z}_{d'}),	\notag
	\end{align}
 then the PDE defined by $\{\mbs{A}_{\alpha},\mcl{D}_{i}\}$ is exponentially PIE to PDE stable with rate $k$ and gain $M:=\sqrt{\|\mcl{P}^*\mcl{T}\|_{\tnf{op}}}/\epsilon$.
\end{cor}
\begin{proof}
	Suppose there exist $P,Q,R$ satisfying the conditions of the corollary, with $\mcl{P}=\tnf{M}[P](I_{n}\otimes\mcl{Z}_{d})$. Since $R\succeq 0$, we then have $\mcl{P}^*\mcl{T}-\epsilon^2\mcl{T}^*\mcl{T}\succeq 0$, and therefore $\mcl{P}^*\mcl{T}\succeq\epsilon^2\mcl{T}^*\mcl{T}$. Furthermore, since $Q\succeq 0$, we have $\mcl{P}^*\mcl{A}+\mcl{A}^*\mcl{P}+2k\mcl{P}^*\mcl{T}\preceq 0$, and thus the LPI~\eqref{eq:LPI_stability} is feasible. By Thm.~\ref{thm:LPI_stability}, it follows that the PDE defined by $\{\mbs{A}_{\alpha},\mcl{D}_{i}\}$ is exponentially PIE to PDE stable, with rate $k$ and gain $M:=\sqrt{\|\mcl{P}^*\mcl{T}\|_{\tnf{op}}}/\epsilon$.
\end{proof}

Cor.~\ref{cor:LPI_stability_SDP} allows stability of a broad class of $N$D PDEs to be tested using SDP. We note, however, that the computational complexity associated with solving the resulting SDP scales with the size of the basis $\mcl Z_d$. Specifically, while the length of the univariate parameter basis $\mbf Z_d$ in~\eqref{eq:Zparams_1D} is proportional to $d^2$, the use of the Kronecker product implies the length of $\mcl Z_d$ scales as $d^{2N}$---implying the number of decision variables in $P,Q,R$ scale as $d^{4N}$.
For higher spatial dimension, then, solving the resulting SDP will require either very low degree or the use of specialized large-scale SDP solvers~\cite{zheng2020SparseSOS,andersen2021cvxopt,wang2025LowRankSDP}. As in the univariate case, however, computational complexity can be partially mitigated by reducing the number of monomials. 
Specifically, the degree $d$ in Cor.~\ref{cor:LPI_stability_SDP} should be chosen to balance accuracy and complexity---the required degree will depend strongly on the type of PDE being analyzed, the required accuracy in decay rate $k$, and the underlying difficulty of the stability problem. 
In addition, kernels may be eliminated from $\mcl{P}$ and multipliers do not appear in the expression for $\mcl{T}^*\mcl{P}$. 
Finally, it is often sufficient to use polynomial rather than polynomial semi-separable kernels (i.e. $\mbf{Z}_{i}^{d}=\{0,\mbs{R}_{i,1},\mbs{R}_{i,1}\}$).


\section{Software and Numerical Examples}\label{sec:Examples}
Having developed a framework for representation and computational stability analysis of a broad class of linear multivariate PDEs, we now demonstrate the application of this framework: providing efficient software construction of the PIE representation and testing of LPIs; presenting the PIE representation for 2D heat, wave, and plate equations; verifying accuracy of the stability test; and providing tight bounds on the associated rates of exponential decay.

\subsection{The PIETOOLS Software Package}\label{subsec:Examples:PIETOOLS}

While Thm.~\ref{thm:PDE2PIE} provides analytic expressions for the conversion of a PDE to a PIE, and while such expressions often provide insight into the structure of the associated PDE and PIE, evaluating such expressions may require substantial effort on the part of the user. Therefore, to facilitate the rapid prototyping and conversion of PDE models to a PIE representation, and to construct the associated stability test, such functions have been automated in the MATLAB software package PIETOOLS~\cite{shivakumar2025PIETOOLS}. A brief introduction and illustrative example of the most useful features of this package are provided here.

Specifically, PIETOOLS is a MATLAB software package whose essential function is to allow for the declaration and manipulation of PI operators and PI operator variables in a syntax similar to matrices---adapting approaches from the widely used SOSTOOLS~\cite{papachristodoulou2021SOSTOOLS} and YALMIP~\cite{lofberg2004yalmip} interfaces.
In addition, a central feature of PIETOOLS is the command-line interface for declaration of 1D and 2D ODE-PDE systems. Given such a PDE model, the software allows for: conversion to PIE; numerical simulation; stability analysis; optimal controller synthesis, etc. 

The command-line interface begins by declaration of PDE state variables using \texttt{pde\_var}, which creates objects which can be added (\texttt{+}), differentiated (\texttt{diff}) and evaluated at spatial positions (\texttt{subs}) in order to define the evolution equations and boundary conditions which constitute a PDE model. Using \texttt{convert}, such a PDE model may then be converted to a PIE, returning the associated PI operators $\{\mcl{T},\mcl{A}\}$.
For example, to compute the PIE representation of the following modified 2D heat equation,
\begin{align}\label{eq:heat_eq_PIETOOLS}
	&\!\!\mbf{u}_{t}(t,x,y)=(1\!+\!x)\mbf{u}_{xx}(t,x,y)\!+\!\mbf{u}_{yy}(t,x,y),~~	x,y\in[0,1],	\notag\\
	&\!\!\mbf{u}(t,0,y)=\!\mbf{u}(t,1,y)=\!\mbf{u}(t,x,0)=\!\mbf{u}_{y}(t,x,1)\!=\!0,	
\end{align}
we can use the following code

\begin{verbatim}
>> pvar t x y;    dom = [0,1; 0,1];
>> u = pde_var('state',[x;y],dom);
>> PDE = [diff(u,t)==(1+x)*diff(u,x,2)+diff(u,y,2);
              subs(u,x,0)==0;   subs(u,x,1)==0;
              subs(u,y,0)==0;   subs(diff(u,y),y,1)==0];
>> PIE = convert(PDE,'pie');
>> T = PIE.T;    A = PIE.A;
\end{verbatim}

Here, the \texttt{convert} function checks the declared \texttt{PDE} for the highest order of the derivative of the PDE state taken along each spatial direction---in this case finding a 2nd-order derivative along both the $x$- and $y$-direction ($\mbf{u}_{xx}$ and $\mbf{u}_{yy}$)---and selects the fundamental state as the corresponding highest-order mixed derivative---in this case $\mbf{v}=\mbf{u}_{xxyy}$.
After verifying consistency, the operators $\{\mcl{T},\mcl{A}\}$ defining the resulting PIE are then returned as \texttt{opvar2d} objects, \texttt{T,A}, for which operations such as addition (\texttt{+}), multiplication (\texttt{*}), and transpose (\texttt{'}) have been overloaded to allow for easy computation of the sum, composition, and adjoint of PI operators. Moreover, given these operators, an LPI such as the program in Cor.~\ref{cor:LPI_stability_SDP} can be declared as an \texttt{lpiprogram} structure, using a programming structure similar to SOSTOOLS~\cite{papachristodoulou2021SOSTOOLS}. This program structure can be modified to add indefinite and positive semidefinite PI operator decision variables using \texttt{lpivar} and \texttt{poslpivar}, and declare equality and inequality constraints using \texttt{lpi\_eq} and \texttt{lpi\_ineq}, after which an SDP solver such as SeDumi~\cite{sturm1999SeDuMi} or MOSEK~\cite{mosek} can be used to solve the program by calling \texttt{lpisolve}. For example, to solve the program in Cor.~\ref{cor:LPI_stability_SDP} with $d=4$, $\epsilon=10^{-1}$, and $k=1$ for the PDE in~\eqref{eq:heat_eq_PIETOOLS}, we can call


{
\begin{verbatim}
>> d = 4;  ep = 1e-1;  k = 1;
>> prog = lpiprogram([x;y],dom);
>> [prog,P] = lpivar(prog,T.dim,d);
>> prog = lpi_eq(prog,P'*T-T'*P);
>> prog = lpi_ineq(prog,T'*P-ep^2*T'*T);
>> prog = lpi_ineq(prog,-(P'*A+A'*P+2*k*P'*T));
>> prog = lpisolve(prog);
\end{verbatim}}

\begin{figure}[t]
	\centering
	\includegraphics[width=1\linewidth]{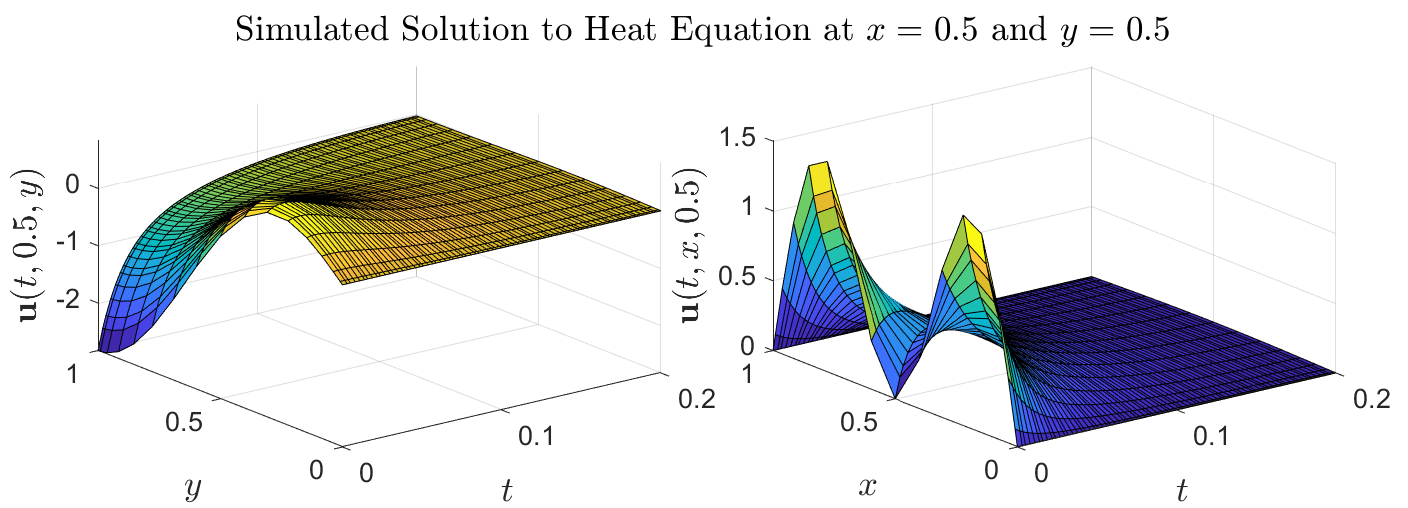}
	\vspace*{-0.4cm}
	\caption{
		Simulated solution $\mbf{u}(t,x,y)$ to the reaction-diffusion equation in~\eqref{eq:heat_eq} at $x=0.5$ and $y=0.5$, using $r=12$ and with $\mbf{u}(0,x,y)=\sqrt{2}\sin(\pi x)\sin(\frac{3\pi}{2}y)+\sqrt{2}\sin(3\pi x)\sin(\frac{\pi}{2}y)$.} 
\label{fig:Heat_Eq_Sim}	
\vspace*{0.1cm}
\end{figure}

A more optimized version of this LPI is also included as a script \texttt{PIETOOLS\_stability}.
We refer to the user manual~\cite{shivakumar2025PIETOOLS} for more details on how to declare PDEs, PIEs, and LPIs using PIETOOLS, as well as
for other applications of the software such as controller design. Note, however, that PIETOOLS does not currently support PDE models of spatial dimension greater than two. Therefore, the application of this software in the following subsection will be limited to analysis of heat, wave and plate equations in 2D.

\subsection{Numerical Examples}\label{subsec:Examples:examples}

In this subsection, we use the PIETOOLS software suite to verify exponential PIE to PDE stability of 2D heat, wave, and plate equations and to compute a maximal lower bound on the associated exponential rate of decay. For the wave and plate equations, second-order temporal derivatives are eliminated through introduction of auxiliary states to obtain a first-order PDE representation as in Eqn.~\eqref{eq:PDE_standard}.
For each example, stability was verified by using the scripts in Subsection~\ref{subsec:Examples:PIETOOLS} to solve the optimization problem in Cor.~\ref{cor:LPI_stability_SDP} with $\epsilon=10^{-1}$, using MOSEK~\cite{mosek} to solve the underlying SDP. Maximal lower bounds on the decay rate were calculated via bisection on the parameter $k$ in Cor.~\ref{cor:LPI_stability_SDP}.
Each PDE was simulated using the PIESIM routines included in PIETOOLS, expanding solutions using a basis of $17\times 17$ Chebyshev polynomials in space, and using the Crank-Nicolson scheme for time integration. 
All numerical tests were performed on a computer with Intel(R) Core(TM) i7-5960x CPU @3.00 GHz, with 128 GB RAM.
Analytic proofs of stability for the heat and wave equation are provided in Appx.~\ref{appx:Examples}.

\subsubsection{Reaction-Diffusion Equation}\label{subsec:Examples:examples:heat}

First, consider again the reaction-diffusion equation from Example~\ref{ex:BCs_DN_2},
\begin{align}\label{eq:heat_eq}
	\mbf{u}_{t}(t,x,y)&\!=\!\mbf{u}_{xx}(t,x,y)\!+\!\mbf{u}_{yy}(t,x,y)\!+\!r\mbf{u}(t,x,y),	\\
	\mbf{u}(t,0,y)&\!=\!\mbf{u}(t,1,y)\!=\!0,\quad \mbf{u}(t,x,0)\!=\!\mbf{u}_{y}(t,x,1)\!=\!0,	\notag
\end{align}
where $\mbf{u}(t)\in S_{2}^{(2,2)}[[0,1]^2]$. A PIE representation of this PDE is presented in Example~\ref{ex:BCs_DN_2}.
This PDE can be shown to be exponentially PIE to PDE stable if and only if $r\leq 1\frac{1}{4}\pi^2\approx 12.337$, with rate $k=1\frac{1}{4}\pi^2-r$---see Appx.~\ref{appx:Examples:heat}.
Using PIETOOLS to verify the stability LPI via Cor.~\ref{cor:LPI_stability_SDP} with $d=1$ and $k=0$, and performing bisection on the value of $r$, stability of the PDE can be verified for any $r\leq 1\frac{1}{4}\pi^2-10^{-3}$. Fixing $r$ and performing bisection on the value of $k$, exponential PIE to PDE stability can be verified with maximal rates $\hat{k}$ as in Table~\ref{tab:decay_HeatEq}, for several values of $r$, and for several values of the maximal monomial degree $d$. The mean computation time $\bar{t}$ needed to parse the stability LPI and solve the resulting semidefinite program for each value of $d$ is also displayed in the table. A simulated solution to the reaction-diffusion equation for $r=12$ is displayed in Fig.~\ref{fig:Heat_Eq_Sim}, at $x=0.5$ and at $y=0.5$, for initial state $\mbf{u}(0,x,y)=\sqrt{2}\sin(\pi x)\sin(\frac{3\pi}{2}y)+\sqrt{2}\sin(3\pi x)\sin(\frac{\pi}{2}y)$.

\begin{figure}[t]
	\centering
	\includegraphics[width=1\linewidth]{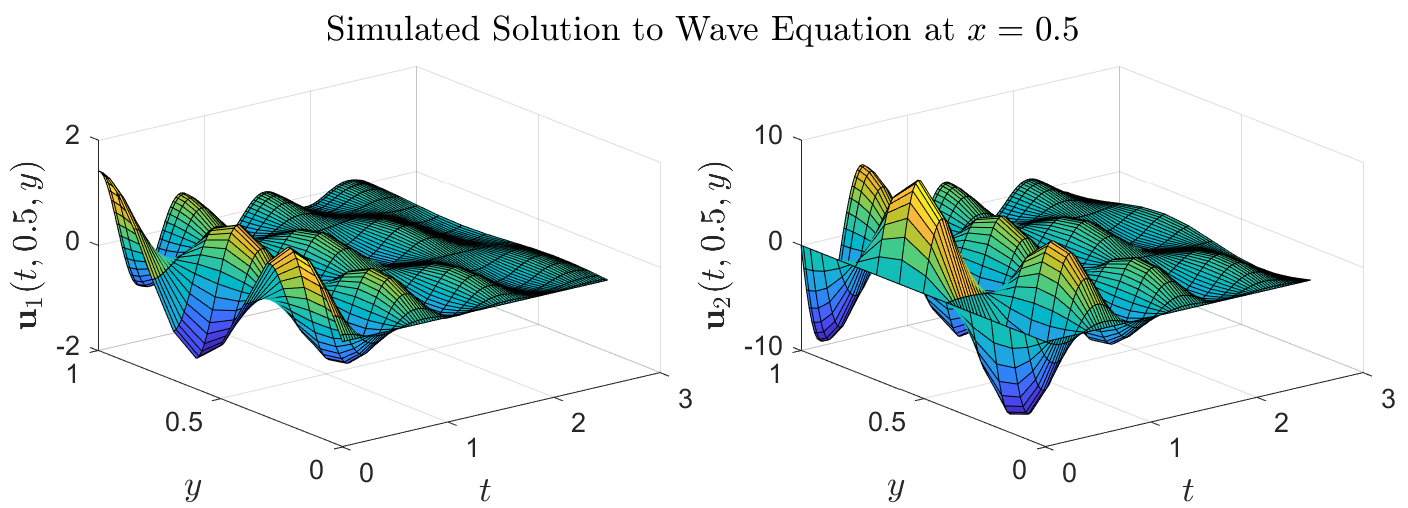}
	\vspace*{-0.4cm}
	\caption{Simulated solution to the damped wave equation
		in~\eqref{eq:wave_eq_u} at $x=0.5$, using $\kappa=1$ and starting with initial
		values $\mbf{u}_{1}(0,x,y):=\sin(\frac{\pi}{2}x)\sin(\frac{5\pi}{2}y)$ and $\mbf{u}_{2}(0,x,y)=0$.}
\label{fig:Wave_Eq_Sim}	
\vspace*{0.1cm}
\end{figure}

\begin{table}[h]
	\setlength{\tabcolsep}{4pt}
	\begin{tabular}{c|c||ccccc|c}
		& $d\,\backslash\, r$ & 0 & 4 & 8 & 12 & 12.3 & $\bar{t}$ (s)	\\[-0.3em]\hline
		\multirow{4}{*}{$\hat{k}$}&
		~$0$ & 12.336 & 8.3362 & 4.3340 & 0.33578 & 0.03613 & 43.44	\\[-0.4em]
		&~$1$ & 12.336 & 8.3363 & 4.3367 & 0.33666 & 0.03685	& 55.74 \\[-0.4em]
		&~$2$ & 12.337 &8.3368 & 4.3369 & 0.33700 & 0.03700	& 149.1 \\[-0.3em]\hline
		$k$ & & 12.337 & 8.3370 & 4.3370 & 0.33700 & 0.03700 & 
	\end{tabular}
	\caption{%
		Largest exponential decay rate $\hat{k}$ for which stability of the		reaction-diffusion equation in~\eqref{eq:heat_eq} was verified for several		values of the parameter $r>0$, using Cor.~\ref{cor:LPI_stability_SDP} with $d\in\{0:2\}$. The average time $\bar{t}$ required to numerically parse and solve the optimization program for each value of $d$ is also displayed, as
		well as the analytic rate of decay $k$ for each value of $r$.
	}\label{tab:decay_HeatEq}
\end{table}

\subsubsection{Wave Equation}\label{subsec:Examples:examples:wave}

Next, consider the 2D wave equation with stabilizing feedback, parameterized by $\kappa\in\R$,
\begin{align}\label{eq:wave_eq}
	&\mbs{\phi}_{tt}(t)=\mbs{\phi}_{xx}(t)+\mbs{\phi}_{yy}(t)-2\kappa\mbs{\phi}_{t}(t)-\kappa^2\mbs{\phi}(t),	\\
	&\mbs{\phi}(t,0,y)=\mbs{\phi}_{x}(t,1,y)=0,\quad \mbs{\phi}(t,x,0)=\mbs{\phi}_{y}(t,x,1)=0,	\notag 
\end{align}
where $\mbs{\phi}(t)\in S_{2}^{(2,2)}[[0,1]^2]$. Using the augmented state $\mbf{u}(t)=(\mbf{u}_{1}(t),\mbf{u}_{2}(t))=(\mbs{\phi}(t),\mbs{\phi}_{t}(t))\in S_{2}^{(2,2),2}$, we represent this PDE in the first-order form of Eqn.~\eqref{eq:PDE_standard_intro} as
\begin{align}\label{eq:wave_eq_u}
	&\partial_{t}\mbf{u}(t)=\bbbl[\scalemath{0.9}{\mat{0&1\\[-0.2em]\partial_{x}^2+\partial_{y}^2-\kappa^2&-2\kappa}}\bbbr]\mbf{u}(t),		\\
	&\mbf{u}(t,0,y)\!=\!\partial_{x}\mbf{u}(t,1,y)\!=\!0,\enspace
	\mbf{u}(t,x,0)\!=\!\partial_{y}\mbf{u}(t,x,1)\!=\!0,		\notag		
\end{align}
for $(x,y)\in[0,1]^2$.
For $\kappa\geq 0$, this PDE can be shown to be exponentially PIE to PDE stable in the sense of Defn.~\ref{defn:exponential_stability} with rate of decay $\kappa$---see Appx.~\ref{appx:Examples:wave}.

Although PIETOOLS can be used to find a PIE representation of this PDE directly, we may provide some additional insight into the structure of the PIE by noting that $\mbf u_1$ and $\mbf u_2$ satisfy the same regularity and boundary constraints. Thus, using Thm.~\ref{thm:Tmap}, we can define operators $\mcl{T}_{1},\mcl{T}_{2}$ and $\mcl{T}=\mcl{T}_{1}\mcl{T}_{2}$ such that $\mbf{u}_{i}(t)=\mcl{T}\partial_{x}^2\partial_{y}^2\mbf{u}_{i}(t)$, $\partial_{x}^2\mbf{u}_{i}(t)=\mcl{T}_{1}\partial_{x}^2\partial_{y}^2\mbf{u}_{i}(t)$, and $\partial_{y}^2\mbf{u}_{i}(t)=\mcl{T}_{2}\partial_{x}^2\partial_{y}^2\mbf{u}_{i}(t)$, for each $i\in\{1,2\}$.
It follows that $\mbf{u}(t)$ satisfies the PDE~\eqref{eq:wave_eq_u} if and only if $\mbf{v}(t)=\mbf{u}_{xxyy}(t)$ satisfies the PIE
\begin{equation*}
\partial_{t}\bbbl[\scalemath{0.9}{\mat{\mcl{T}&0\\[-0.2em]0&\mcl{T}}}\bbbr]\mbf{v}(t)=\bbbl[\scalemath{0.9}{\mat{0&\mcl{T}\\[-0.2em]\mcl{T}_{1}+\mcl{T}_{2}-\kappa^2\mcl{T}&-2\kappa\mcl{T}}}\bbbr]\mbf{v}(t).
\end{equation*}
Simulating this PIE for $\kappa=1$ and with $\mbf{u}_{1}(0,x,y)=\sin(\frac{\pi}{2}x)\sin(\frac{5\pi}{2}y)$ and $\mbf{u}_{2}(0,x,y)=0$, the obtained solution at $x=0.5$ is plotted in Fig.~\ref{fig:Wave_Eq_Sim}, showing that both $\mbf{u}_{1}(t)=\mbs{\phi}(t)$ and $\mbf{u}_{2}(t)=\mbs{\phi}_{t}(t)$ converge to zero. Using Cor.~\ref{cor:LPI_stability_SDP} with $d=1$, exponential stability can also be verified with PIETOOLS. Maximal lower bounds on the decay rate, $\hat{k}$, are provided in Tab.~\ref{tab:decay_WaveEq}, in each case tightly lower-bounding the true rate $k=\kappa$.
The mean computation time to parse the stability LPI and solve the resulting SDP for each value of $\kappa$ was 720 seconds, and the mean CPU time for solving the SDP was 477 seconds.

\begin{table}[h]
	\setlength{\tabcolsep}{2.8pt}
	\begin{tabular}{c|ccccccc}
		$\kappa=k$ & 1 & 2 & 3 & 4 & 5 & 6 & 7	\\[-0.4em]\hline
		$\hat{k}$ & 0.9999 & 1.9998 & 2.9874 & 3.9973 & 4.9980 & 5.9906 & 6.9645  
	\end{tabular}
	\caption{Largest lower bound on the decay rate, $\hat{k}$, for which exponential PIE to PDE stability of the wave equation in~\eqref{eq:wave_eq_u} was verified with Cor.~\ref{cor:LPI_stability_SDP}, using $d=1$, for several values of the parameter $\kappa>0$ (which is also the true decay rate, $k$).}\label{tab:decay_WaveEq}
\end{table}

\subsubsection{Kirchhoff Plate Equation}\label{subsec:Examples:examples:plate}

As a final example, consider the following Kirchhoff plate equation with structural damping, used to model the vertical deflection $\mbf{w}$ of a clamped plate,
\begin{align}\label{eq:Plate_Eq}
	\mbf{w}_{tt}(t)&=-\mbf{w}_{xxxx}(t)-2\mbf{w}_{xxyy}(t)-\mbf{w}_{yyyy}(t)\\
	&\qquad+\alpha_{0}(\mbf{w}_{txx}(t)+\mbf{w}_{tyy}(t)),	\notag	\\
	\mbf{w}(t,0,y)&=\mbf{w}_{x}(t,0,y)=\mbf{w}(t,1,y)=\mbf{w}_{x}(t,1,y)=0,	\notag\\
	\mbf{w}(t,x,0)&=\mbf{w}_{y}(t,x,0)=\mbf{w}(t,x,1)=\mbf{w}_{y}(t,x,1)=0.	\notag
\end{align}
Introducing $\mbf{u}(t)=(\mbf{w}(t),\mbf{w}_{t}(t))\in S_{2}^{(4,4),2}[[0,1]^2]$, we represent this PDE in the first-order form as 
\begin{align}\label{eq:Plate_Eq_u}
	&\partial_{t}\mbf{u}(t)=\bbbl[\scalemath{0.875}{\mat{0&1\\[-0.2em]\!-D^{(4,0)}\!-\!2D^{(2,2)}\!-\!D^{(0,4)}&\,\alpha_{0}\bl(D^{(2,0)}\!+\!D^{(0,2)}\br)}}\bbbr]\mbf{u}(t),	\notag\\
	&\mbf{u}(t,0,y)\!=\!\mbf{u}_{x}(t,0,y)\!=\!\mbf{u}(t,1,y)\!=\!\mbf{u}_{x}(t,1,y)\!=\!0,	\notag\\
	&\mbf{u}(t,x,0)\!=\!\mbf{u}_{y}(t,x,0)\!=\!\mbf{u}(t,x,1)\!=\!\mbf{u}_{y}(t,x,1)\!=\!0.	
\end{align}
As for the wave equation, we use Thm.~\ref{thm:Tmap} to define PI operators $\mcl{T}, \mcl R_i, \mcl Q_i$ such that $\mbf{v}(t)=D^{(4,4)}\mbf{u}(t)$ satisfies
\begin{equation*}
	\partial_{t}\bbbr[\scalemath{0.9}{\mat{\mcl{T}&0\\[-0.2em]0&\mcl{T}}}\bbbr]\mbf{v}(t)=\bbbl[\scalemath{0.9}{\mat{0&\mcl{T}\\[-0.2em]-\mcl{Q}_{1}-\mcl{Q}_{2}-2\mcl{Q}_{3}&\alpha_{0}\bl(\mcl{R}_{1}+\mcl{R}_{2}\br)}}\bbbr]\mbf{v}(t).
\end{equation*}
Applying Cor.~\ref{cor:LPI_stability_SDP} to this PIE, using $d=0$, the PDE~\eqref{eq:Plate_Eq_u} with $\alpha_{0}=0.2$ was found to be exponentially PIE to PDE stable with rate at least $k=3.6328$. 
Here, although stability properties of plate equations such as that in~\eqref{eq:Plate_Eq} have been studied in the literature, including in e.g.~\cite{liu1997PlateEq,fridman2010PlateEq,hajjej2024KirchoffPlate}, these results commonly prove exponential decay of a different energy functional $E(\mbf{u}(t))$, rather than of the norm $\norm{\mbf{u}(t)}_{L_{2}}$.
Therefore, to assess the tightness of the lower bound $k=3.6328$ on the exponential decay rate, solutions to the PDE were also simulated using the PIESIM software. 
The norm $\|\mbf{u}(t)\|_{L_{2}}$ of the simulated solutions is displayed in Fig.~\ref{fig:Plate_Eq_Sim}, for three initial states $\mbf{u}^{i}:=\hat{\mbf{u}}^{i}/\|\hat{\mbf{u}}^{i}\|_{L_{2}}$, where $\hat{\mbf{u}}^{i}=(\hat{\mbf{u}}^{i}_{1},\hat{\mbf{u}}^{i}_{2})$ for
\begin{align}\label{eq:Plate_Eq_InitialStates}
	\hat{\mbf{u}}^{1}_{1}(x,y)&:=(\cos(2\pi x)-1)(\cos(2\pi y)-1),\quad
	\hat{\mbf{u}}^{1}_{2}(x,y):=0,	\notag\\
	\hat{\mbf{u}}^{2}_{1}(x,y)&:=\hat{\mbf{u}}^{1}_{1}(x,2y),\hspace*{2.0cm}
	\hat{\mbf{u}}^{2}_{2}(x,y):=2\hat{\mbf{u}}^{1}_{1}(2x,y),	\notag\\
	\hat{\mbf{u}}^{3}_{1}(x,y)&:=x^2(1-x)^2y^2(1-y)^2,\quad
	\hat{\mbf{u}}^{3}_{2}(x,y):=2\hat{\mbf{u}}^{3}_{1}(x,y).	
\end{align}
These numerical simulations are consistent with the computed lower bound on the exponential decay rate.

\begin{figure}[t]
	\centering
	\includegraphics[width=1\linewidth]{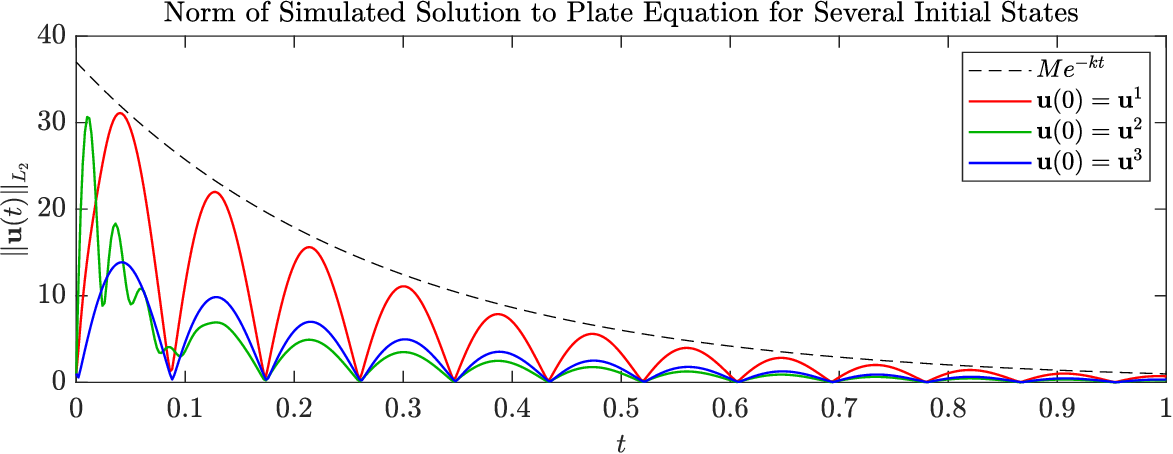}
	\vspace*{-0.4cm}
	\caption{Evolution of the norm $\|\mbf{u}(t)\|_{L_{2}}$ of simulated solutions to\\[-0.1em]
		the Kirchhoff plate equation in~\eqref{eq:Plate_Eq_u} for $\alpha_{0}=0.2$ and for initial	states $\mbf{u}^{i}:=\hat{\mbf{u}}^{i}/\|\hat{\mbf{u}}^{i}\|_{L_{2}}$, with $\hat{\mbf{u}}^{i}$ as in~\eqref{eq:Plate_Eq_InitialStates} for $i\in\{1,2,3\}$. An exponential bound $Me^{-kt}$ for $M=37$ is also plotted, with		 $k=3.6328$ corresponding to the largest rate for which exponential PIE to PDE stability was verified using PIETOOLS.}
	\label{fig:Plate_Eq_Sim}	
\end{figure}

\section{Conclusion}

We have shown how a broad class of coupled, linear PDEs on a hyper-rectangular spatial domain can be equivalently represented as Partial Integral Equations (PIEs). This representation was constructed inductively, based on a representation of the multivariate domain (including boundary conditions and Sobolev regularity constraints) as the intersection of lifted 1D domains. For such a domain decomposition, we proposed an algebraic consistency condition which was shown to be necessary and sufficient for invertibility of the multivariate spatial differential operator, $D^{\delta}$, on the domain of the PDE---thus providing a bijection between the PDE domain and the Hilbert space $L_2$. Furthermore, an expression for this inverse was constructed inductively as the composition of univariate integral operators with polynomial semi-separable kernels. This inverse operator was then shown to be embedded in a *-algebra of multivariate Partial Integral (PI) operators, allowing for the construction of an equivalent PIE representation of the evolution of the PDE. Given this representation, a notion of exponential PIE to PDE stability was defined and a sufficient condition for such stability was formulated as feasibility of a set of linear PI operator inequalities. Feasibility of such operator inequalities was then verified by constructing a basis for the PI operators and using positive matrices to enforce positivity. Software implementation was proposed which automates construction of the PIE representation and verification of the associated stability conditions. Finally, accuracy of the resulting stability analysis was verified by applying the stability test to 2D heat, wave, and plate equations.

\bibliography{bibfile}
\bibliographystyle{plain}


%
%
%

\clearpage

\onecolumn

\appendix

\section{Bijection Between the PDE Domain and Fundamental State Space}\label{appx:proofs}

In Section~\ref{sec:Tmap}, an explicit expression was derived for the inverse of the multivariate differential operator $D^{\delta}:\mcl{D}\to L_{2}[\Omega]$, on the PDE domain $\mcl{D}\subseteq S_{2}^{\delta}[\Omega]$. Specifically, it was shown that if we can decompose $\mcl{D}=\bigcap_{i=1}^{N}\mcl{D}_{i}$, for lifted univariate domains $\mcl{D}_{i}$, then the inverse to $D^{\delta}:\mcl{D}\to L_{2}[\Omega]$ can be expressed as the product $\mcl{T}=\prod_{i=1}^{N}\mcl{T}_{i}:L_{2}[\Omega]\to\mcl{D}$ of the inverse $\mcl{T}_{i}$ to each $\partial_{s_{i}}^{\delta_{i}}:\mcl{D}_{i}\to L_{2}[\Omega]$. The proof of this result, however, relies on the commutability of the operators $\mcl{T}_{i},\partial_{s_{j}}^{\delta_{j}}$, and of the operators $\mcl{T}_{i},\mcl{T}_{j}$, for distinct $i,j$. In this appendix, we show that these commutative properties are indeed satisfied under the conditions presented in Section~\ref{sec:Tmap}---in particular when the domains $\mcl{D}_{i}$ are consistent---providing some of the proofs and details omitted from that section.

To start, consider a univariate domain $\mcl{D}\subseteq S_{2}^{d,n}[a,b]$ as in~\eqref{eq:PDEdom_1D}, lifted to the multivariate setting as $\mcl{D}\mapsto\mcl{D}_{i}=\hat{\mcl{D}}[i]\subseteq S_{2}^{\delta_{i}\cdot\tnf{e}_{i},n}[\Omega]$ using $d\mapsto\delta_{i}$ and $[a,b]\mapsto[a_{i},b_{i}]$. Recall from Cor.~\ref{cor:gohberg_PIE} and Cor.~\ref{cor:Tmap_1D} that if $\mcl{D}$ is admissible as per Defn.~\ref{defn:admissible_BCs_1D}, we can define the inverse to $\partial_{s_{i}}^{\delta_{i}}:\mcl{D}_{i}\to L_{2}^{n}[\Omega]$ as an integral operator $\mcl{T}_{i}:L_{2}^{n}[\Omega]\to \mcl{D}_{i}$ with polynomial semi-separable kernel. As such, for any distinct $i,j\in\{1:N\}$, we can define a right-inverse to $\partial_{s_{i}}^{\delta_{i}}\partial_{s_{j}}^{\delta_{j}}:\mcl{D}_{i}\cap \mcl{D}_{j}\to L_{2}^{n}[\Omega]$ as $\mcl{T}_{j}\mcl{T}_{i}$, satisfying $\partial_{s_{i}}^{\delta_{i}}\partial_{s_{j}}^{\delta_{j}}\,\mcl{T}_{j}\mcl{T}_{i}\mbf{v}=\partial_{s_{i}}^{\delta_{i}}\, \mcl{T}_{i}\mbf{v}=\mbf{v}$ for all $\mbf{v}\in L_{2}^{n}[\Omega]$. In order to show that this operator also satisfies $\mcl{T}_{j}\mcl{T}_{i}\,\partial_{s_{i}}^{\delta_{i}}\partial_{s_{j}}^{\delta_{j}}\mbf{u}=\mbf{u}$ for $\mbf{u}\in\mcl{D}_{i}\cap\mcl{D}_{j}$, however, we first need to prove that $\partial_{s_{j}}^{\delta_{j}}\mbf{u}\in \mcl{D}_{i}$, for which we have the following lemma.

\begin{lem}\label{lem:PDEdom_diff_appx} 
	Let $N\in\N$, $\delta\in\N_{0}^{N}$, and $\Omega=\prod_{i=1}^{N}[a_{i},b_{i}]$.
	For given $B^{i}_{j,k},C^{i}_{j,k}$, let associated domains $\mcl{D}_{i}$ be of the form in~\eqref{eq:PDEdom_1D} (letting $B_{j,k}=B^{i}_{j,k}$, $C_{j,k}=C^{i}_{j,k}$, $[a,b]=[a_{i},b_{i}]$ and $d=\delta_{i}$), and lifted to the multivariate setting as defined in Subsection~\ref{subsec:notation_sobolev} (so that $\mcl{D}_{i}=\hat{\mcl{D}}[i]\subseteq S_{2}^{\delta_{i}\cdot\tnf{e}_{i},n}[\Omega]$) for each $i\in\{1:N\}$. 
	If $\mbf{u}\in\mcl{D}:=\mcl{D}_{1}\cap\cdots\cap\mcl{D}_{N}$, then $\partial_{s_{j}}^{\alpha_{j}}\cdots\partial_{s_{N}}^{\alpha_{N}}\mbf{u}\in\mcl{D}_{1}\cap\cdots\cap\mcl{D}_{j-1}$ for all $j\in\{1:N\}$ and $\vec{0}\leq\alpha\leq\delta$.
\end{lem}
\begin{proof}
	Fix arbitrary $\vec{0}\leq\alpha\leq\delta$.
	We prove the result by induction on $j$, ranging from $N$ to $1$. For the base case, $j=N$, fix arbitrary $\mbf{u}\in\mcl{D}$. Then $\mbf{u}\in S_{2}^{\delta_{i}\cdot\tnf{e}_{i},n}[\Omega]$ for all $i\in\{1:N\}$, and therefore $\partial_{s_{N}}^{\alpha_{N}}\mbf{u}\in S_{2}^{\delta_{i}\cdot\tnf{e}_{i},n}[\Omega]$ for all $i\in\{1:N-1\}$. In addition, by definition of the domains $\mcl{D}_{i}$, we have
	\begin{equation*}
		\sum_{k=0}^{\delta_{i}-1}\bbl[B^{i}_{\ell,k} (\partial_{s_{i}}^{k}\mbf{u})(s)|_{s_{i}=a_{i}} +C^{i}_{\ell,k}(\partial_{s_{i}}^{k}\mbf{u})(s)|_{s_{i}=b_{i}}\bbr]=0,\qquad \forall \ell\in\{0:\delta_{i}-1\},
	\end{equation*} 
	for all $i\in\{1:N-1\}$. It follows that also
	\begin{align*}
		&\sum_{k=0}^{\delta_{i}-1}\bbl[B^{i}_{\ell,k} (\partial_{s_{i}}^{k}\partial_{s_{N}}^{\alpha_{N}}\mbf{u})(s)|_{s_{i}=a_{i}} +C^{i}_{\ell,k}(\partial_{s_{i}}^{k}\partial_{s_{N}}^{\alpha_{N}}\mbf{u})(s)|_{s_{i}=b_{i}}\bbr]	\\
		&\hspace*{2.0cm}=\partial_{s_{N}}^{\alpha_{N}}\sum_{k=0}^{\delta_{i}-1}\bbl[B^{i}_{\ell,k} (\partial_{s_{i}}^{k}\mbf{u})(s)|_{s_{i}=a_{i}} +C^{i}_{\ell,k}(\partial_{s_{i}}^{k}\mbf{u})(s)|_{s_{i}=b_{i}}\bbr]
		=0,
	\end{align*}
	for every $\ell\in\{0:\delta_{i}-1\}$, and therefore $\partial_{s_{N}}^{\alpha_{N}}\mbf{u}\in \mcl{D}_{i}$, for all $i\in\{1:N-1\}$. Thus, $\partial_{s_{N}}^{\alpha_{N}}\mbf{u}\in \mcl{D}_{1}\cap\cdots\cap \mcl{D}_{N-1}$.
	
	Now, suppose for induction the lemma statement holds for all $j\in\{J:N\}$, for some $J\in\{1:N\}$, and let $j=J-1$. Then, for any $\mbf{u}\in\mcl{D}$, we have $\partial_{s_{j+1}}^{\alpha_{j+1}}\cdots\partial_{s_{N}}^{\alpha_{N}}\mbf{u}\in \mcl{D}_{1}\cap\cdots\cap \mcl{D}_{j}\subseteq S_{2}^{\delta_{1},n}[\Omega]\cap\cdots\cap S_{2}^{\delta_{j}\cdot\tnf{e}_{j},n}[\Omega]$, by the induction hypothesis. It follows that $\partial_{s_{j}}^{\alpha_{j}}\cdots\partial_{s_{N}}^{\alpha_{N}}\mbf{u}\in S_{2}^{\delta_{i}\cdot\tnf{e}_{i},n}[\Omega]$ for all $i\in\{1:j-1\}$. In addition, by definition of the domains $\mcl{D}_{i}$, we have
	\begin{equation*}
		\sum_{k=0}^{\delta_{i}-1}\bbl[B^{i}_{\ell,k} \bl(\partial_{s_{i}}^{k}(\partial_{s_{j+1}}^{\alpha_{j+1}}\cdots\partial_{s_{N}}^{\alpha_{N}}\mbf{u})\br)(s)|_{s_{i}=a_{i}} +C^{i}_{\ell,k}\bl(\partial_{s_{i}}^{k}(\partial_{s_{j+1}}^{\alpha_{j+1}}\cdots\partial_{s_{N}}^{\alpha_{N}}\mbf{u})\br)(s)|_{s_{i}=b_{i}}\bbr]=0,\qquad \forall \ell\in\{0:\delta_{i}-1\},
	\end{equation*}
	for all $i\in\{1:j\}$. It follows that also
	\begin{align*}
		&\sum_{k=0}^{\delta_{i}-1}\bbl[B^{i}_{\ell,k} \bl(\partial_{s_{i}}^{k}(\partial_{s_{j}}^{\alpha_{j}}\cdots\partial_{s_{N}}^{\alpha_{N}}\mbf{u})\br)(s)|_{s_{i}=a_{i}} +C^{i}_{\ell,k}\bl(\partial_{s_{i}}^{k}(\partial_{s_{j}}^{\alpha_{j}}\cdots\partial_{s_{N}}^{\alpha_{N}}\mbf{u})\br)(s)|_{s_{i}=b_{i}}\bbr]	\\
		&\qquad=\partial_{s_{j}}^{\alpha_{j}}\sum_{k=0}^{\delta_{i}-1}\bbl[B^{i}_{\ell,k} \bl(\partial_{s_{i}}^{k}(\partial_{s_{j+1}}^{\alpha_{j+1}}\cdots\partial_{s_{N}}^{\alpha_{N}}\mbf{u})\br)(s)|_{s_{i}=a_{i}} +C^{i}_{\ell,k}\bl(\partial_{s_{i}}^{k}(\partial_{s_{j+1}}^{\alpha_{j+1}}\cdots\partial_{s_{N}}^{\alpha_{N}}\mbf{u})\br)(s)|_{s_{i}=b_{i}}\bbr]
		=0,
	\end{align*}
	for all $\ell\in\{0:\delta_{j}\}$, and therefore $\partial_{s_{j}}^{\alpha_{j}}\cdots\partial_{s_{N}}^{\alpha_{N}}\mbf{u}\in\mcl{D}_{i}$, for all $i\in\{1:j-1\}$. We find that $\partial_{s_{j}}^{\alpha_{j}}\cdots\partial_{s_{N}}^{\alpha_{N}}\mbf{u}\in\mcl{D}_{1}\cap\cdots\cap\mcl{D}_{j-1}$, whence the result holds by induction.
\end{proof}

Lem.~\ref{lem:PDEdom_diff_appx} proves that, for any $\mbf{u}\in \mcl{D}_{1}\cap\cdots\cap\mcl{D}_{N}$, we have $\partial_{s_{j}}^{\delta_{j}}\cdots\partial_{s_{N}}^{\delta_{N}}\mbf{u}\in \mcl{D}_{1}\cap\cdots\cap\mcl{D}_{j-1}$ for all $j\in\{1:N\}$. As such, if both $\mcl{D}_{i}$ and $\mcl{D}_{j}$ are admissible as per Defn.~\ref{defn:admissible_BCs_1D}, then for any $\mbf{u}\in \mcl{D}_{i}\cap \mcl{D}_{j}$ we have by Cor.~\ref{cor:gohberg_PIE} and Cor.~\ref{cor:Tmap_1D} that $\partial_{s_{j}}^{\delta_{j}}\mbf{u}=\mcl{T}_{i}\partial_{s_{i}}^{\delta_{i}}(\partial_{s_{j}}^{\delta_{j}}\mbf{u})$, and thus $\mcl{T}_{j}\mcl{T}_{i}\,\partial_{s_{i}}^{\delta_{i}}\partial_{s_{j}}^{\delta_{j}}\mbf{u}=\mbf{u}$. More generally, the following corollary proves that $\mcl{T}:=\mcl{T}_{N}\cdots\mcl{T}_{1}$ defines a left-inverse to $D^{\delta}:\mcl{D}\to L_{2}^{n}$, as well as a right-inverse to the extension of $D^{\delta}$ to the range of $\mcl{T}$.

\begin{cor}\label{cor:LRinverse_appx}
	Let $\mcl{T}:=\mcl{T}_{N}\cdots\mcl{T}_{1}$ and $\mcl{D}:=\mcl{D}_{1}\cap\cdots\cap\mcl{D}_{N}$, for $\mcl{T}_{i}$ as in Cor.~\ref{cor:Tmap_1D}. Then the following statements hold:
	\begin{enumerate}
		\item $\mcl T D^\delta \mbf{u}=\mbf{u}$ for any $\mbf{u}\in \mcl D$.
		\item $D^{\delta}\mcl{T} \mbf{v}=\mbf{v}$ for any $\mbf{v}\in L_2^{n}$.
	\end{enumerate}
\end{cor}
\begin{proof}
	For the first statement, we note that, by Lem.~\ref{lem:PDEdom_diff_appx}, we have $\partial_{s_{j+1}}^{\delta_{j+1}}\cdots\partial_{s_{N}}^{\delta_{N}}\mbf{u}\in \mcl{D}_{1}\cap\cdots\cap\mcl{D}_{j}\subseteq \mcl{D}_{j}$ for all $\mbf{u}\in\mcl{D}$ and $j\in\{1:N\}$. Since, by Cor.~\ref{cor:Tmap_1D}, $\mcl{T}_{j}\partial_{s_{j}}^{\delta_{j}}\mbf{u}=\mbf{u}$ for all $\mbf{u}\in\mcl{D}_{j}$, it follows that for every $\mbf{u}\in\mcl{D}$,
	\begin{align*}
		\mcl{T}D^{\delta}\mbf{u}
		&=\mcl{T}_{N}\cdots\mcl{T}_{2}(\mcl{T}_{1}\partial_{s_{1}}^{\delta_{1}})(\partial_{s_{2}}^{\delta_{2}}\cdots\partial_{s_{N}}^{\delta_{N}}\mbf{u})	\\
		&=\mcl{T}_{N}\cdots\mcl{T}_{3}(\mcl{T}_{2}\partial_{s_{2}}^{\delta_{2}})(\partial_{s_{3}}^{\delta_{3}}\cdots \partial_{s_{N}}^{\delta_{N}}\mbf{u})	
		=~\cdots~
		=\mcl{T}_{N}\partial_{s_{N}}^{\delta_{N}}\mbf{u}
		=\mbf{u}.
	\end{align*}
	Now, for the second statement, note that $\mcl{T}_{j-1}\cdots\mcl{T}_{1}\mbf{v}\in L_{2}^{n}$ for all $\mbf{v}\in L_{2}^{n}$ and $j\in\{1:N\}$. Since, by Cor.~\ref{cor:Tmap_1D}, $\partial_{s_{j}}^{\delta_{j}}\mcl{T}_{j}\mbf{v}=\mbf{v}$ for all $\mbf{v}\in L_{2}^{n}$, it follows that
	\begin{align*}
		D^{\delta}\mcl{T}\mbf{v}
		&=\partial_{s_{1}}^{\delta_{1}}\cdots\partial_{s_{N-1}}^{\delta_{N-1}}(\partial_{s_{N}}^{\delta_{N}}\mcl{T}_{N})(\mcl{T}_{N-1}\cdots\mcl{T}_{1}\mbf{v})	\\
		&=\partial_{s_{1}}^{\delta_{1}}\cdots\partial_{s_{N-2}}^{\delta_{N-2}}(\partial_{s_{N-1}}^{\delta_{N-1}}\mcl{T}_{N-1})(\mcl{T}_{N-2}\cdots\mcl{T}_{1}\mbf{v})
		=~\cdots~	
		=\partial_{s_{1}}^{\delta_{1}}\mcl{T}_{1}\mbf{v}
		=\mbf{v}.
	\end{align*}
\end{proof}

Cor.~\ref{cor:LRinverse_appx} shows that $\mcl{T}:=\mcl{T}_{N}\cdots\mcl{T}_{1}$ defines both a left- and right-inverse to $D^{\delta}:\mcl{D}\to L_{2}^{n}$ on the image of $D^{\delta}$. In order for $\mcl{T}$ to define an inverse to $D^{\delta}:\mcl{D}\to L_{2}^{n}$, then, it remains to prove that the image of $D^{\delta}$ is indeed $L_{2}^{n}$, or equivalently, that $\mcl{T}\mbf{v}\in \mcl{D}$ for all $\mbf{v}\in L_{2}^{n}$. Unfortunately, admissibility of the domains $\mcl{D}_{i}$ is insufficient to guarantee that this is indeed the case. Indeed, as shown in Subsection~\ref{subsec:Tmap:consistency} for $\Omega=[0,1]^2$ and $\delta=(d,d)$, a unique inverse to $\partial_{x}^{d}\partial_{y}^{d}:\mcl{D}_{1}\cap \mcl{D}_{2}\to L_{2}^{n}[[0,1]^2]$ may not exist if the boundary conditions defining $\mcl{D}_{1}$ and $\mcl{D}_{2}$ impose conflicting constraints at any of the corners of the domain, i.e. $(x,y)\in\{(0,0),(1,0),(0,1),(1,1)\}$. A necessary condition on the parameters $\{B^{1}_{\ell,k},C^{1}_{\ell,k}\}$ and $\{B^{2}_{\ell,k},C^{2}_{\ell,k}\}$ to avoid such conflicting conditions was presented in Lem.~\ref{lem:compatibility_Du(a,c)_2D}, which we restate and prove here.

\begin{lem}\label{lem:compatibility_Du(a,c)_2D_appx}
	Suppose $B^{i}_{j,k},C^{i}_{j,k}$ as in Eqn.~\eqref{eq:PDEdom_1D} define admissible domains $\mcl{D}\subseteq S_{2}^{d,n}[0,1]$ as per Defn.~\ref{defn:admissible_BCs_1D} (using $B_{j,k}=B^{i}_{j,k}$, $C_{j,k}=C^{i}_{j,k}$, $[a,b]=[0,1]$), lifted to the multivariate setting as $\mcl{D}\mapsto\mcl{D}_{i}$, and let $[H^{i}_a]_{j,k}=B^{i}_{j,k}$, $[H^{i}_b]_{j,k}=C^{i}_{j,k}$. Define associated $K^{i}:=\bl(H_{a}^{i}+H_{b}^{i}\mbs{Q}(b_{i}-a_{i})\br)^{-1}H_{b}^{i}\in\R^{n d\times n d}$ for $\mbs{Q}$ as in Defn.~\ref{defn:admissible_BCs_1D}, and decompose the matrix $K^{i}$ into blocks as
	\begin{equation}\label{eq:M_decomp_2D_appx} 
		K^{i}=\bmat{K^{i}_{1,1}&\hdots&K^{i}_{1,d}\\\vdots&\ddots&\vdots\\K^{i}_{d,1}&\hdots&K^{i}_{d,d}},\qquad
		\text{where}\qquad
		K^{i}_{k,\ell}\in\R^{n\times n},~\forall k,\ell\in\{1:d\}.
	\end{equation} 
	If there exists an operator $\mcl{T}:L_{2}^{n}[[0,1]^2]\to \mcl{D}:=\mcl{D}_{i}\cap\mcl{D}_{j}$ such that $\partial_{x}^{d}\partial_{y}^{d}\mcl{T}\mbf{v}=\mbf{v}$ for all $\mbf{v}\in L_{2}^{n}[[0,1]^2]$, then
	\begin{equation}\label{eq:M_commute_necessity_2D_appx}
		K^{1}_{k,p}K^{2}_{\ell,q}=K^{2}_{\ell,q}K^{1}_{k,p},\qquad \forall k,p,\ell,q\in\{1:d\}.
	\end{equation}
\end{lem}
\begin{proof}
	Suppose there exists an operator $\mcl{T}:L_{2}^{n}[[0,1]^2]\to \mcl{D}$ such that $\partial_{x}^{d}\partial_{y}^{d}\mcl{T}\mbf{v}=\mbf{v}$ for all $\mbf{v}\in L_{2}^{n}[[0,1]^2]$. For any $\mbf{v}\in L_{2}^{n}[\Omega]$, let $\mbf{u}=\mcl{T}\mbf{v}$. Then, $\mbf{v}=\partial_{x}^{d}\partial_{y}^{d}\mbf{u}$ and $\mbf{u}\in \mcl{D}$. By Lem.~\ref{lem:PDEdom_diff_appx}, it follows that also $\partial_{y}^{d}\mbf{u}\in \mcl{D}_{1}$ and $\partial_{x}^{d}\mbf{u}\in \mcl{D}_{2}$.
	
	Now, fix arbitrary $k,\ell\in\{1:d\}$. To prove that~\eqref{eq:M_commute_necessity_2D_appx} holds, we will use the fact that $\mbf{u}\in \mcl{D}_{2}$ and $\partial_{y}^{d}\mbf{u}\in \mcl{D}_{1}$, and the fact that $\mbf{u}\in \mcl{D}_{1}$ and $\partial_{x}^{d}\mbf{u}\in \mcl{D}_{2}$, to derive two distinct expressions for the same corner value, $(\partial_{x}^{k-1}\partial_{y}^{\ell-1}\mbf{u})(0,0)$, in terms of $\mbf{v}=\partial_{x}^{d}\partial_{y}^{d}\mbf{u}$. First, given that $\mbf{u}\in \mcl{D}_{2}$, we remark that by the proof of Cor.~\ref{cor:gohberg_PIE}, $\mbf{u}$ must satisfy
	\begin{align*}
		0&=H_{a}^{2}\bl(\partial_{y}^{(0:d-1)}\mbf{u}\br)(x,0) +H^{2}_{b}\bl(\partial_{y}^{(0:d-1)}\mbf{u}\br)(x,1)		\\
		&=\bl(H^{2}_{a}+H^{2}_{b}\mbs{Q}(1)\br)\bl(\partial_{y}^{(0:d-1)}\mbf{u}\br)(x,0) +\int_{0}^{1}H^{2}_{b}\mbf{e}_{d}(1-y)\partial_{y}^{d}\mbf{u}(x,y)\,dy.
	\end{align*}
	where $\mbs{Q}$ is as in Defn.~\ref{defn:admissible_BCs_1D}, and where
	\begin{equation*}
		\mbf{e}_d(z):=\bmat{\frac{z^{d-1}}{(d-1)!}\\\vdots\\ \half z^2\\ z \\1}\otimes I_n,\hspace*{1.0cm}
		\partial_{y}^{(0:d-1)}\mbf{u}:=\bmat{\mbf{u}\\\partial_{y}\mbf{u}\\\vdots\\\partial_{y}^{d-1}\mbf{u}}.
	\end{equation*}
	Since, by assumption, the domains $\mcl{D}_{1}$ and $\mcl{D}_{2}$ are admissible, it follows by Defn.~\ref{defn:admissible_BCs_1D} that the matrix $(H^{2}_{a}+H^{2}_{j,k}\mbs{Q}(1))\in\R^{nd\times nd}$ is invertible, and thus $\mbf{u}$ must satisfy
	\begin{equation*}
		\bl(\partial_{y}^{(0:d-1)}\mbf{u}\br)(x,0) =-\int_{0}^{1}\bl(H^{2}_{a}+H^{2}_{b}\mbs{Q}(1)\br)^{-1} H^{2}_{b}\mbf{e}_{d}(1-y)\partial_{y}^{d}\mbf{u}(x,y)\,dy
		=-\int_{0}^{1}K^{2}\mbf{e}_{d}(1-y)\,\partial_{y}^{d}\mbf{u}(x,y)\,dy,
	\end{equation*}
	Taking the derivative $\partial_{x}^{k-1}$ of both sides of this identity, and remarking that $\partial_{y}^{\ell-1}\mbf{u}=(\tnf{e}_{\ell}\otimes I_{n})^T \partial_{y}^{(0:d-1)}\mbf{u}$ for $\tnf{e}_{\ell}\in\R^{d}$ the $\ell$th standard Euclidean basis vector, it follows that the derivative $\partial_{x}^{k-1}\partial_{y}^{\ell-1}\mbf{u}$ must satisfy
	\begin{align*}
		\bl(\partial_{x}^{k-1}\partial_{y}^{\ell-1}\mbf{u}\br)(x,0)
		=(\tnf{e}_{\ell}\otimes I_{n})^T\partial_{x}^{k-1}\bl(\partial_{y}^{(0:d-1)}\mbf{u}\br)(x,0)
		=-\int_{0}^{1}K^{2}_{(\ell,:)} \mbf{e}_{d}(1-y)\,\partial_{x}^{k-1}\partial_{y}^{d}\mbf{u}(x,y)\,dy.
	\end{align*}
	where now $K^{2}_{(\ell,:)}:=(\tnf{e}_{\ell}\otimes I_{n})^T K^{2}$.
	Here, since $\partial_{y}^{d}\mbf{u}\in\mcl{D}_{1}$, by Cor.~\ref{cor:gohberg_PIE} and Cor.~\ref{cor:Tmap_1D} we can express $\partial_{x}^{k-1}(\partial_{y}^{d}\mbf{u})=\mcl{R}_{1}^{k-1}\partial_{x}^{d}(\partial_{y}^{d}\mbf{u})=\mcl{R}_{1}^{k-1}\mbf{v}$, where the operator $\mcl{R}_{1}^{k}$ is defined by
	\begin{equation*}
		(\mcl{R}_{1}^{k-1}\mbf{v})(x,y):=\int_{0}^{1} \mbs{G}_{k-1}(x,\theta)\mbf{v}(\theta,y)  d\theta, \qquad \mbs{G}_{k}(x,\theta):=
		\begin{cases}
			\mbf{c}_{k-1}(x)^T(I_{nd}-K^{1})\mbf{e}_d(1-\theta)&\theta \leq x,\\
			-\mbf{c}_{k-1}(x)^TK^{1}\mbf{e}_d(1-\theta)& x < \theta,
		\end{cases}
	\end{equation*}
	where
	\begin{equation*}
		\mbf{e}_1(z):=\bmat{1 \\ z\\\half z^2\\ \vdots \\\frac{z^{d-1}}{(d-1)!}}\otimes I_n,
		\qquad 
		\mbf{c}_{k-1}(z):=\partial_{z}^{k-1}\mbf{e}_{1}(z)
		=\bmat{0_{k-1}\\1\\z\\\vdots\\\frac{z^{d-k}}{(d-k)!}}\otimes I_{n}.
	\end{equation*}
	It follows that
	\begin{equation*}
		\bl(\partial_{x}^{k-1}\partial_{y}^{\ell-1}\mbf{u}\br)(x,0)
		=-\int_{0}^{1}K^{2}_{(\ell,:)}\mbf{e}_{d}(1-y)\bl(\mcl{R}_{1}^{k-1}\mbf{v}\br)(x,y)\, dy.
	\end{equation*}
	Evaluating this expression at $x=0$, we note that, by definition of $\mcl{R}_{1}^{k-1}$, we have
	\begin{equation*}
		\bl(\mcl{R}_{1}^{k-1}\mbf{v}\br)(0,y)
		=-\int_{0}^{1} \mbf{c}_{k-1}(0-0)^TK^{1}\mbf{e}_{d}(1-x)\,\mbf{v}(x,y)\,dx	
		=-\int_{0}^{1} K_{(k,:)}^{1} \mbf{e}_{d}(1-x)\,\mbf{v}(x,y)\,dx.
	\end{equation*}
	Using this expression it follows that
	\begin{align*}
		\bl(\partial_{x}^{k-1}\partial_{y}^{\ell-1}\mbf{u}\br)(0,0)
		&=-\int_{0}^{1}K^{2}_{(\ell,:)}\mbf{e}_{d}(1-y)\bl(\mcl{R}_{1}^{k-1}\mbf{v}\br)(0,y)\, dy	\\
		&=\int_{0}^{1}K^{2}_{(\ell,:)}\mbf{e}_{d}(1-y)\bbbl(\int_{0}^{1}K^{1}_{(k,:)} \mbf{e}_{d}(1-x)\mbf{v}(x,y)\,dx\bbbr) dy	\\
		&=\int_{0}^{1}\int_{0}^{1}\bl(K^{2}_{(\ell,:)}\mbf{e}_{d}(1-y)\br)\bl(K^{1}_{(k,:)} \mbf{e}_{d}(1-x)\br)\mbf{v}(x,y)\,dy\,dx.
	\end{align*}
	Here, to derive this expression for $\partial_{x}^{k-1}\partial_{y}^{\ell-1}\mbf{u}(0,0)$, we only used the fact that $\mbf{u}\in\mcl{D}_{2}$ and $\partial_{y}^{d}\mbf{u}\in \mcl{D}_{1}$. Conversely, then, using the fact that $\mbf{u}\in \mcl{D}_{1}$ and $\partial_{x}^{d}\mbf{u}\in \mcl{D}_{2}$, we can apply similar reasoning to find that also
	\begin{equation*}
		\bl(\partial_{x}^{k-1}\partial_{y}^{\ell-1}\mbf{u}\br)(0,0)
		=\int_{0}^{1}\int_{0}^{1}\bl(K^{1}_{(k,:)} \mbf{e}_{d}(1-x)\br)\bl(K^{2}_{(\ell,:)}\mbf{e}_{d}(1-y)\br)\mbf{v}(x,y)\,dy\,dx.
	\end{equation*}
	Since $\mbf{u}\in \mcl{D}\subseteq S_{2}^{(d,d),n}[\Omega]$ and $k,\ell\leq d$, the derivative $\partial_{x}^{k-1}\partial_{y}^{\ell-1}\mbf{u}$ is continuous, and therefore we must have
	\begin{align*}
		&\int_{0}^{1}\int_{0}^{1}\bl(K^{2}_{(\ell,:)}\mbf{e}_{d}(1-y)\br)\bl(K^{1}_{(k,:)} \mbf{e}_{d}(1-x)\br)\mbf{v}(x,y)\,dy\,dx	\\
		&\hspace*{2.0cm}=\bl(\partial_{x}^{k-1}\partial_{y}^{\ell-1}\mbf{u}\br)(0,0)
		=\int_{0}^{1}\int_{0}^{1}\bl(K^{1}_{(k,:)} \mbf{e}_{d}(1-x)\br)\bl(K^{2}_{(\ell,:)}\mbf{e}_{d}(1-y)\br)\mbf{v}(x,y)\,dy\,dx.
	\end{align*}
	Since this holds for arbitrary $\mbf{v}\in L_{2}^{n}[\Omega]$, it follows that in fact
	\begin{equation}\label{eq:EQME_commute_2D_appx}
		\bl(K^{2}_{(\ell,:)}\mbf{e}_{d}(y)\br)\bl(K^{1}_{(k,:)}\mbf{e}_{d}(x)\br)
		-\bl(K^{1}_{(k,:)}\mbf{e}_{d}(x)\br)\bl(K^{2}_{(\ell,:)}\mbf{e}_{d}(y)\br)
		=0\qquad \tnf{for a.e.}~ x,y\in[0,1].			\tag{$\star$}	
	\end{equation}
	Finally, to see that this implies~\eqref{eq:M_commute_necessity_2D_appx}, we decompose $K^{i}$ and $K^{j}$ as in~\eqref{eq:M_decomp_2D_appx}, and invoke the definition of $\mbf{e}_{d}$ to find
	\begin{equation*}
		K^{1}_{(k,:)}\mbf{e}_{d}(z)
		=\bmat{K^{1}_{k,1}&K^{1}_{k,2}&\hdots&K^{1}_{k,d}}\bmat{\frac{z^{d-1}}{(d-1)!}I_{n}\\\frac{z^{d-2}}{(d-2)!}I_{n}\\\vdots\\I_{n}}
		=\sum_{p=1}^{d}\frac{1}{(d-p)!}\,K^{1}_{k,p}\, z^{d-p}.
	\end{equation*}
	Similarly, we can express $K^{2}_{(\ell,:)}\mbf{e}_{d}(z)=\sum_{q=1}^{d}\frac{1}{(d-q)!}K^{2}_{\ell,q} z^{d-q}$, for all $z\in\R$. It follows that
	\begin{align*}
		&\bl(K^{2}_{(\ell,:)}\mbf{e}_{d}(y)\br)\bl(K^{1}_{(k,:)}\mbf{e}_{d}(x)\br)
		-\bl(K^{1}_{(k,:)}\mbf{e}_{d}(x)\br)\bl(K_{(\ell,:)}^{2}\mbf{e}_{d}(y)\br)	\\
		&\hspace*{4.0cm} =\sum_{p,q=1}^{d}\frac{1}{(d-p)!(d-q)!}\bbl[K^{1}_{k,p}K^{2}_{\ell,q}-K^{2}_{\ell,q}K^{1}_{k,p}\bbr] x^{d-p} y^{d-q},
	\end{align*}
	for all $x,y\in[0,1]$.
	By ~\eqref{eq:EQME_commute_2D_appx}, it follows that $K^{i}_{k,p}K^{j}_{\ell,q}=K^{j}_{\ell,q}K^{i}_{k,p}$ for all $p,q\in\{1:d\}$.
\end{proof}

Lem.~\ref{lem:compatibility_Du(a,c)_2D_appx} proves that the differential operator $\partial_{x}^{d}\partial_{y}^{d}:\mcl{D}_{1}\cap \mcl{D}_{2}\to L_{2}^{n}[[0,1]^2]$ admits a right-inverse only if a suitable commutative condition is satisfied, ensuring that the boundary conditions defining $\mcl{D}_{1}$ and $\mcl{D}_{2}$ can be simultaneously satisfied at the corner $(x,y)=(0,0)$. We refer to this commutative property as consistency of the domains $\mcl{D}_{1}$ and $\mcl{D}_{2}$---see also Defn.~\ref{defn:consistent_BCs}. In order to prove that consistency of the PDE domains is also sufficient for $\partial_{x}^{d}\partial_{y}^{d}:\mcl{D}_{1}\cap \mcl{D}_{2}\to L_{2}^{n}[[0,1]^2]$ to be right-invertible, we remark that the operators $\mcl{T}=\mcl{T}_{2}\mcl{T}_{1}$ and $\tilde{\mcl{T}}=\mcl{T}_{1}\mcl{T}_{2}$ both already satisfy $\partial_{x}^{d}\partial_{y}^{d}\mcl{T}=\partial_{x}^{d}\partial_{y}^{d}\tilde{\mcl{T}}=I_{n}$, though $\mcl{T}\mbf{v}\notin \mcl{D}_{1}$ and $\tilde{\mcl{T}}\mbf{v}\notin \mcl{D}_{2}$ for general $\mbf{v}\in L_{2}^{n}[[0,1]^2]$. As such, a right-inverse to $\partial_{x}^{d}\partial_{y}^{d}:\mcl{D}_{1}\cap \mcl{D}_{2}\to L_{2}^{n}[[0,1]^2]$ exists only if the operators $\mcl{T}_{1}$ and $\mcl{T}_{2}$ commute. The following lemmas, generalizing Lem.~\ref{lem:M_commute_T} in Sec.~\ref{subsec:Tmap:consistency}, prove that this condition is equivalent to consistency of the domains $\mcl{D}_{1}$ and $\mcl{D}_{2}$.

\begin{lem}\label{lem:M_commute_Tparam_appx}
	For $i,j\in\{1:N\}$, let $\mbs{T}_{i}(x,\theta):=-\mbf{e}_{1}(x-a_{i})^T K^{i}\mbf{e}_{\delta_{i}}(b_{i}-\theta)$ for some $K^{i}\in\R^{n\delta_{i}\times n\delta_{i}}$, and where $\mbf{e}_{1},\mbf{e}_{\delta_{i}}$ are as in Cor.~\ref{cor:gohberg_PIE} (using $d=\delta_{i}$). Then $\mbs{T}_{i}\mbs{T}_{j}=\mbs{T}_{j}\mbs{T}_{i}$ if and only if $K^{i}_{k,p}K^{j}_{\ell,q}=K^{j}_{\ell,q}K^{i}_{k,p}$ for all $k,p\in\{1:\delta_{i}\}$ and $\ell,q\in\{1:\delta_{j}\}$.
\end{lem}
\begin{proof}
	By definition, we have
	\begin{align*}
		\mbs{T}_{i}(x+a_{i},b_{i}-\theta)&:=-\mbf{e}_{1}(x)K^{i}\mbf{e}_{\delta_{i}}(\theta)		\\
		&=\bmat{I_{n}\\ xI_{n}\\\vdots\\\frac{x^{\delta_{i}-1}}{(\delta_{i}-1)!}I_{n}}^T\bmat{K^{i}_{1,1}&K^{i}_{1,2}&\hdots&K^{i}_{1,\delta_{i}}\\ K^{i}_{2,1}&K^{i}_{2,2}&\hdots&K^{i}_{2,\delta_{i}}\\\vdots&\vdots&\ddots&\vdots\\K^{i}_{\delta_{i},1}&K^{i}_{\delta_{i},2}&\hdots&K^{i}_{\delta_{i},\delta_{i}}}\bmat{\frac{\theta^{\delta_{i}-1}}{(\delta_{i}-1)!}I_{n}\\ \frac{\theta^{\delta_{i}-2}}{(\delta_{i}-2)!}I_{n}\\\vdots\\I_{n}}	\\
		&=\sum_{k,p=1}^{\delta_{i}}\frac{1}{(k-1)!(\delta_{i}-p)!}\,K^{i}_{k,p}\, x^{k-1}\theta^{\delta_{i}-p},
	\end{align*}
	for all $x,\theta\in[0,b_{i}-a_{i}]$. It follows that
	\begin{align*}
		&\mbs{T}_{i}(x+a_{i},b_{i}-\theta)\mbs{T}_{j}(y+a_{j},b_{j}-\eta)
		-\mbs{T}_{j}(y+a_{j},b_{j}-\eta)\mbs{T}_{i}(x+a_{i},b_{i}-\theta)	\\
		&\hspace*{1.0cm}=\sum_{k,p=1}^{\delta_{i}}\sum_{\ell,q=1}^{\delta_{j}}\frac{1}{(k-1)!(\delta_{i}-p)!}\frac{1}{(\ell-1)!(\delta_{j}-q)!}\,
		\bbl[K^{i}_{k,p}K^{j}_{\ell,q}-K^{j}_{\ell,q}K^{i}_{k,p}\bbr]\, x^{k-1}\theta^{\delta_{i}-p} y^{\ell-1}\eta^{\delta_{j}-q},	
	\end{align*}
	for all $x,\theta\in[0,b_{i}-a_{i}]$ and $y,\eta\in[0,b_{j}-a_{j}]$. It follows that $\mbs{T}_{i}\mbs{T}_{j}=\mbs{T}_{j}\mbs{T}_{i}$ if and only if $K^{i}_{k,p}K^{j}_{\ell,q}=K^{j}_{\ell,q}K^{i}_{k,p}$ for all $k,p,\ell,q$. 
\end{proof}

\begin{lem}\label{lem:M_commute_T_appx}
	For $\delta\in\N_{0}^{N}$ and $\Omega=\prod_{i=1}^{N}[a_{i},b_{i}]$, let $\mcl{D}_{i}\subseteq S_{2}^{\delta_{i}\cdot\tnf{e}_{i},n}[\Omega]$ for each $i\in\{1:N\}$ be admissible, and define associated operators $\mcl{T}_{i}$ as in Cor.~\ref{cor:gohberg_PIE} and Cor.~\ref{cor:Tmap_1D} and matrices $K^{i}$ as in Defn.~\ref{defn:consistent_BCs}. 
	Then, for any $i,j\in\{1:N\}$ such that $i\neq j$, we have
	\begin{equation*}
		\mcl{T}_{i}\circ\mcl{T}_{j}=\mcl{T}_{j}\circ\mcl{T}_{i}\qquad
		\Leftrightarrow\qquad
		K^{i}_{k,p}K^{j}_{\ell,q}=K^{j}_{\ell,q}K^{i}_{k,p},\quad\forall k,p\in \{1:\delta_{i}\}, \ell,q\in\{1:\delta_{j}\},
	\end{equation*}
	where
	$K^{i}_{k,p}:=(\tnf{e}_{k}\otimes I_{n})^T K^{i}(\tnf{e}_{p}\otimes I_{n})$ for $\tnf{e}_{k}\in\R^{\delta_{i}}$ the $k$th standard Euclidean basis vector.
\end{lem}
\begin{proof}
	Fix arbitrary $i,j\in\{1:N\}$ such that $i\neq j$. If either $\delta_{i}=0$ or $\delta_{j}=0$, then $\mcl{T}_{i}=I_{n}$ or $\mcl{T}_{j}=I_{n}$, and the result holds trivially. Otherwise, let parameters $K^{i}$ be as in Defn.~\ref{defn:consistent_BCs}, and let $\mbs{T}_{i}(x,\theta):=-\mbf{e}_{1}(x-a_{i})^T K^{i}\mbf{e}_{\delta_{i}}(b_{i}-\theta)$ for $\mbf{e}_{1},\mbf{e}_{d}$ as in Cor.~\ref{cor:gohberg_PIE}. Then, by definition of $\mcl{T}_{i}$ in Cor.~\ref{cor:Tmap_1D}, we can express $\mcl{T}_{i}=\mcl{Q}_{i}+\mcl{R}_{i}$, where for any $\mbf{v}\in L_{2}^{n}[\Omega]$,
	\begin{equation*}
		(\mcl{Q}_{i}\mbf{v})(s)
		:=\int_{a_{i}}^{b_{i}}\mbs{T}_{i}(s_{i},\theta_{i})\mbf{v}(s)|_{s_{i}=\theta_{i}}d\theta_{i},\qquad\text{and}\qquad (\mcl{R}_{i}\mbf{v})(s):=\int_{a_{i}}^{s_{i}}\frac{(s_{i}-\theta_{i})^{\delta_{i}-1}}{(\delta_{i}-1)!}\mbf{v}(s)|_{s_{i}=\theta_{i}}d\theta_{i},
	\end{equation*}
	for all $s=(s_{1},\hdots,s_{n})\in\Omega$. Here, we note by inspection that for any $\mbs{T}_{i}$, we have $\mcl{Q}_{i}\mcl{R}_{j}=\mcl{R}_{j}\mcl{Q}_{i}$ and $\mcl{R}_{i}\mcl{R}_{j}=\mcl{R}_{j}\mcl{R}_{i}$. Hence the $\mcl{T}_{i}$ will commute if and only if the $\mcl{Q}_{i}$ commute. By definition of $\mcl{Q}_{i}$, we have for every $\mbf{v}\in L_{2}^{n}[\Omega]$,
	\begin{align*}
		\bbl(\bl(\mcl{Q}_{i}\mcl{Q}_{j}-\mcl{Q}_{j}\mcl{Q}_{i}\br)\mbf{v}\bbr)(s)&=\int_{a_{i}}^{b_{i}}\mbs{T}_{i}(s_{i},\theta_{i})\left[\int_{a_{j}}^{b_{j}}\mbs{T}_{j}(s_{j},\theta_{j})\mbf{v}(s)|_{s_{i}=\theta_{i},s_{j}=\theta_{j}}\,d\theta_{j}\right]d\theta_{i}	\\
		&\quad -\int_{a_{j}}^{b_{j}}\mbs{T}_{j}(s_{j},\theta_{j})\left[\int_{a_{i}}^{b_{i}}\mbs{T}_{i}(s_{i},\theta_{i})\mbf{v}(s)|_{s_{i}=\theta_{i},s_{j}=\theta_{j}}\,d\theta_{i}\right]d\theta_{j}	\\
		&=\int_{a_{i}}^{b_{i}}\int_{a_{j}}^{b_{j}}\bbl[\mbs{T}_{i}(s_{i},\theta_{i})\mbs{T}_{j}(s_{j},\theta_{j})-\mbs{T}_{j}(s_{j},\theta_{j})\mbs{T}_{i}(s_{i},\theta_{i})\bbr]\mbf{v}(s)|_{s_{i}=\theta_{i},s_{j}=\theta_{j}}\,d\theta_{j} d\theta_{i}.
	\end{align*}
	Therefore, $\mcl{T}_{i}\circ\mcl{T}_{j}=\mcl{T}_{j}\circ\mcl{T}_{i}$ if and only if $\mbs{T}_{i}\mbs{T}_{j}=\mbs{T}_{j}\mbs{T}_{i}$ except on a set of measure zero. However, since the $\mbs{T}_{i}$ are polynomial, by Lemma~\ref{lem:M_commute_Tparam_appx},  $\mbs{T}_{i}\mbs{T}_{j}=\mbs{T}_{j}\mbs{T}_{i}$ if and only if $K^{i}_{k,p}K^{j}_{\ell,q}=K^{j}_{\ell,q}K^{i}_{k,p}$ for all $k,p,\ell,q$, concluding the proof.
\end{proof}

Lem.~\ref{lem:M_commute_T_appx} proves that for any distinct $i,j\in\{1:N\}$, the operators $\mcl{T}_{i}$ and $\mcl{T}_{j}$ commute if and only if the $n\times n$ sub-blocks of the matrices $K^{i}$ and $K^{j}$ commute---i.e. if the domains $\mcl{D}_{i}$ are consistent as per Defn.~\ref{defn:consistent_BCs}. Assuming admissible and consistent domains, then, not only can we define an inverse $\mcl{T}_{i}:L_{2}^{n}[\Omega]\to \mcl{D}_{i}$ to each $\partial_{s_{i}}^{\delta_{i}}:\mcl{D}_{i}\to L_{2}^{n}[\Omega]$, but these operators $\mcl{T}_{i}$ and $\mcl{T}_{j}$ also commute for $i\neq j$. Using this fact, Thm.~\ref{thm:Tmap} proves that for admissible and consistent $\mcl{D}_{i}$, the inverse of the differential operator $D^{\delta}:\bigcap_{i=1}^{N}\mcl{D}_{i}\to L_{2}^{n}[\Omega]$ takes the form $\mcl{T}=\prod_{i=1}^{N}\mcl{T}_{i}:L_{2}^{n}[\Omega]\to \mcl{D}$. In proving this result, we also relied on the fact that for $i\neq j$, the integral operator $\mcl{T}_{i}$ commutes with the differential operator $\partial_{s_{j}}^{\delta_{j}}$, as proven in the following lemma.

\begin{lem}\label{lem:Top_orth_Dop}
	For $N\in \N$ and $\delta \in \N_{0}^N$, let $\mcl{D}:=\bigcap_{i=1}^{N}\mcl{D}_{i} \subset S_{2}^{\delta,n}$ with $\mcl{D}_{i}\subset S_2^{\delta_i,n}$ admissible and consistent.
	
	Let $\mcl{T}_{i},\mcl{A}_{i,j}$ be defined as in Cor.~\ref{cor:gohberg_PIE} and extended to multivariate space using Cor.~\ref{cor:Tmap_1D} (letting $\mcl{T}\mapsto\mcl{T}_{i}$ and $\mcl{A}\mapsto\mcl{A}_{i,j}$), for $d=\delta_{i}$, $\mcl{D}=\mcl{D}_{i}$ and $\{\mbs{H}_{k}\}$ where $\mbs{H}_{k}:=\bbbl\{\mat{I_{n},\hspace*{0.25cm} k=j,\\ 0,\quad k\neq j,}$ for each $k\in\{0:d\}$. 
	Then, for any $i,j\in\{1:N\}$ such that $i\neq j$ and all $d\in\{0:\delta_{i}\}$ and $k\in\{0:\delta_{j}\}$, we have
	\begin{equation*}
		\partial_{s_{i}}^{d}\mcl{T}_{j}\mbf{u}=\mcl{T}_{j}\partial_{s_{i}}^{d}\mbf{u},
		\qquad \text{and}\qquad
		\partial_{s_{i}}^{d}\mcl{A}_{j,k}\mbf{u}=\mcl{A}_{j,k}\partial_{s_{i}}^{d}\mbf{u},
	\end{equation*}
	for all $\mbf{u}\in S_{2,i}^{d,n}[\Omega]:=\bl\{\mbf{u}\in L_{2}^{n}[\Omega]\mid \partial_{s_{i}}^{\ell}\mbf{u}\in L_{2}^{n},~\forall 0\leq \ell\leq d\br\}$.
\end{lem}
\begin{proof}
	We prove the result only for the operators $\mcl{A}_{j,k}$ (not for $\mcl{T}_{j}$), noting that $\mcl{T}_{j}=\mcl{A}_{j,0}$.
	
	Fix arbitrary $i,j\in\{1:N\}$ such that $i\neq j$, and let $d\in\{0:\delta_{i}\}$ and $k\in\{0:\delta_{j}\}$. Fix further arbitrary $\mbf{u}\in S_{2,i}^{d,n}[\Omega]$. If $d=0$ or $k=d$, we have $\partial_{s}^{d}=I_{n}$ or $\mcl{A}_{j,k}=I_{n}$ (by the definition in Cor.~\ref{cor:gohberg_PIE}), and therefore $\partial_{s_{i}}^{d}\mcl{A}_{j,k}\mbf{u}=\mcl{A}_{j,k}\partial_{s_{i}}^{d}\mbf{u}$ holds trivially. Otherwise, if $k\neq d$ and $d\neq 0$, then by definition of the operator $\mcl{A}_{j,k}$ in Cor.~\ref{cor:gohberg_PIE},
	\begin{align*}
		\bl(\mcl{A}_{j,k}\mbf{v}\br)(s)
		&=\int_{a_{j}}^{b_{j}}\mbs{R}_{j,k}(s_{j},\theta_{j})\,\mbf{v}(s_{1},\hdots,s_{j-1},\theta_{j},s_{j+1},\hdots,s_{N})\,d\theta_{j}	\\
		&\qquad +\int_{a_{j}}^{s_{j}}\frac{(s_{j}-\theta_{j})^{k-1}}{(k-1)!}\,\mbf{v}(s_{1},\hdots,s_{j-1},\theta_{j},s_{j+1},\hdots,s_{N})\,d\theta_{j},
	\end{align*}
	where $\mbs{R}_{j,k}(x,y):=-\mbf{c}_{k}(x-a_{j})K\mbf{e}_{d}(b_{j}-y)$ for $K$ as defined in Cor.~\ref{cor:gohberg_PIE}, and where
	\begin{equation*}
		\mbf{e}_d(z)=\bmat{\frac{z^{d-1}}{(d-1)!}&\cdots& \half z^2& z &1}^T\otimes I_n,
		\qquad
		\mbf{c}_{k}(z)
		=\bmat{0_{k}&1&z&\cdots&\frac{z^{d-k-1}}{(d-1)!}}^T\otimes I_{n}.
	\end{equation*}
	It follows that
	\begin{align*}
		\bl(\mcl{A}_{j,k}\partial_{s_{i}}^{d}\mbf{u}\br)(s)
		&=\int_{a_{j}}^{b_{j}}\mbs{R}_{j,k}(s_{j},\theta_{j})\,\partial_{s_{i}}^{d}\mbf{u}(s_{1},\hdots,s_{j-1},\theta_{j},s_{j+1},\hdots,s_{N})\,d\theta_{j} 		\\
		&\qquad +\int_{a_{j}}^{s_{j}}\frac{(s_{j}-\theta_{j})^{k-1}}{(k-1)!}\,\partial_{s_{i}}^{d}\mbf{u}(s_{1},\hdots,s_{j-1},\theta_{j},s_{j+1},\hdots,s_{N})\,d\theta_{j}	\\
		&=\partial_{s_{i}}^{d}\left(\int_{a_{j}}^{b_{j}}\mbs{R}_{j,k}(s_{j},\theta_{j})\mbf{u}(s_{1},\hdots,s_{j-1},\theta_{j},s_{j+1},\hdots,s_{N})\,d\theta_{j} \right)		\\
		&\qquad +\partial_{s_{i}}^{d}\left(\int_{a_{j}}^{s_{j}}\frac{(s_{j}-\theta_{j})^{k-1}}{(k-1)!}\mbf{u}(s_{1},\hdots,s_{j-1},\theta_{j},s_{j+1},\hdots,s_{N})\,d\theta_{j}	\right)
		=\bl(\partial_{s_{i}}^{d}\mcl{A}_{j,k}\mbf{u}\br)(s).
	\end{align*}
\end{proof}


\section{The $*$-Algebra of $N$D 3-PI Operators}\label{appx:PIs}


In Section~\ref{sec:PIs}, a class of multivariate 3-PI operators was defined, denoted by $\PIset_{N}[\Omega]$, generalizing the definition of univariate 3-PI operators to the $N$D setting. In this appendix, we show that the class of multivariate 3-PI operators includes a particular class of integral operators, namely that of integral operators with polynomial semi-separable kernels. In addition, we provide some of the proofs omitted from Subsection~\ref{subsec:PIs:algebra}, proving that the class of multivariate 3-PI operators is indeed closed under linear combinations, as well as composition and adjoint operations---thus defining a $*$-algebra.

To start, we introduce the following convenient notation to denote the univariate and multivariate 3-PI operators defined by given parameters.
\begin{defn}[Parameterization of ND 3-PI operators, $\PI_{N}{[\mbf{P}]}$]\label{defn:PI_ND_appx}
	For $[a,b]\subset\R$, define $\PI_{1}:\PIparam_{1}^{m\times n}[a,b]\to\PIset_{1}^{m\times n}[a,b]$ such that $\PI_{1}[\mbf{P}]$ is the univariate 3-PI operator defined by $\mbf{P}=\{\mbs{P}_{0},\mbs{P}_{1},\mbs{P}_{2}\}\in \PIparam_{1}^{m\times n}[a,b]:=\R^{m\times n}[s]\times\R^{m\times n}[s,\theta]\times \R^{m\times n}[s,\theta]$, i.e.
	\begin{equation*}
			(\PI_{1}[\mbf{P}]\mbf{v})(s)
			:=\mbs{P}_{0}(s)\mbf{v}(s) +\int_{a}^{s}\mbs{P}_{1}(s,\theta)\mbf{v}(\theta)\, d\theta +\int_{s}^{b}\mbs{P}_{2}(s,\theta)\mbf{v}(\theta)\, d\theta,\qquad s\in[a,b],
	\end{equation*}
	for $\mbf{v}\in L_{2}^{n}[a,b]$.
	For $\Omega=\prod_{i=1}^{N}[a_{i},b_{i}]$, define $\PIparam_{N}[\Omega]:=\prod_{i=1}^{N}\PIparam_{1}[a_{i},b_{i}]$, and define the map $\PI_{N}:\PIparam_{N}[\Omega]\to\PIset[\Omega]$ from parameters $\mbf{P} =\{\mbf{P}_i\}_{i=1}^N \in\PIparam_{N}[\Omega]$ to an associated $N$D 3-PI operator as $\PI_{N}[\mbf{P}]:=\prod_{i=1}^{N} \PI_{1}[\mbf{P}_{i}]$, where now the $i$th univariate operator acts only along the $i$th variable,
	\begin{equation*}
		(\PI_{1}[\mbf{P}_{i}]\mbf{v})(s):=\bl(\PI_{1}[\mbf{P}_{i}]\mbf v(s_1,\cdots,s_{i-1},\bullet,s_{i+1},\cdots s_N)\br)(s_i),\qquad s\in\Omega.
	\end{equation*}
\end{defn}

We remark that, by Defn.~\ref{defn:PI_ND}, any $N$D 3-PI operator $\mcl{P} \in \PIset_{N}[\Omega]$ can be expressed as a sum of products of 3-PI operators along each dimension $i\in\{1,\hdots,N\}$---i.e. $\mcl{P} = \sum_j \PI_{N}[\mbf{P}_j]=\sum_{j}\prod_{i}\PI_{N}[\mbf{P}_{j,i}]$ for some $\mbf{P}_j:=\{\mbf{P}_{j,i}\}_{i=1}^{N} \in \PIparam_{N}[\Omega]$. In turn, any univariate 3-PI operator is given by the sum of a multiplier operator defined by a polynomial, and an integral operator with polynomial semi-separable kernel, so that $\mcl{R}\in\PIset_{1}[a,b]$ if $\mcl{R}\mbf{v}(s)=\mbs{R}(s)\mbf{v}(s)+\int_{a}^{b}\mbs{K}(s,\theta)\mbf{v}(\theta)d\theta$ for some $\mbs{R}\in\R[s]$ and some $\mbs{K}$ of the form
\begin{equation*}
	\mbs{K}(s,\theta)=\begin{cases}
		\mbs{K}_{1}(s,\theta),	&	\theta\leq s,	\\
		\mbs{K}_{2}(s,\theta),	&	\theta>s,
	\end{cases}
\end{equation*}
for $\mbs{K}_{1},\mbs{K}_{2}$ polynomial. Taking sums of compositions of such 3-PI operators along different intervals $[a_{i},b_{i}]$, then, the proposed class of $N$D 3-PI operators will include (but is not limited to) any operator of the form 
\begin{align*}
	(\mcl{R}\mbf{v})(s)=\sum_{j=1}^{M}\left(\prod_{i=1}^{N}\mbs{R}_{i,j}(s_{i})\right) \mbf{v}(s),\qquad
	\text{and}\qquad
	(\mcl{K}\mbf{v})(s)=\sum_{j=1}^{M}\left(\prod_{i=1}^{N}\int_{a_{i}}^{b_{i}}\mbs{K}_{i,j}(s_{i},\theta_{i}) \right)\mbf{v}(\theta)d\theta,
\end{align*}
where $\mbs{R}_{i,j}$ are polynomial, and $\mbs{K}_{i,j}$ are polynomial semi-separable. Note here that, any multivariate polynomial can be expressed as $\mbs{R}(s)=\sum_{j=1}^{M}\prod_{i=1}^{N}\mbs{R}_{i,j}(s_{i})$ for some univariate polynomials $\mbs{R}_{i,j}$. As such, the class $\PIset_{N}[\Omega]$ will include any multiplier operator defined by a polynomial. In addition, defining multivariate polynomial semi-separable functions as in Defn.~\ref{defn:semisep}, any such polynomial semi-separable function can be expressed as $\sum_{j=1}^{M}\prod_{i=1}^{N}\mbs{K}_{i,j}$ for some univariate polynomial semi-separable functions. Consequently, any integral operator with polynomial semi-separable kernel will be a multivariate 3-PI operator, yielding the following result.

\begin{lem}\label{lem:PI_semisep_appx}
	Suppose that 
	\[
	(\mcl{K}\mbf{v})(s):=\int_{\Omega}\mbs{K}(s,\theta)\mbf{v}(\theta)d\theta
	\]
	where $\mbs{K}(s,\theta)\in L_{2}[\Omega\times\Omega]$ is multivariate polynomial semi-separable. Then $\mcl{K}\in\PIset_{N}[\Omega]$.
\end{lem}
\begin{proof} 
	Suppose that $\mbs{K}$ is multivariate polynomial semi-separable. Then there exist polynomials $\mbs{K}_{\alpha}\in\R[s,\theta]$ for $\alpha\in\{-1,1\}^{N}$ such that
	\begin{align*}
		\mbs{K}(s,\theta)=\mbs{K}_{\alpha}(s,\theta)	\qquad
		\forall (s,\theta)\in \{(s,\theta)\in\Omega\mid \alpha_{i}(s_{i}-\theta_{i})\leq 0,~\forall i\in\{1:N\}\}.
	\end{align*}
	Given these polynomials, we can then express
	\begin{equation*}
		(\mcl{K}\mbf{v})(s)
		=\int_{\Omega}\mbs{K}(s,\theta)\mbf{v}(\theta)d\theta
		=\sum_{\alpha\in\{-1,1\}^{N}} \int_{\Omega}\mbf{I}_{\alpha}(s,\theta)\mbs{K}_{\alpha}(s,\theta)\mbf{v}(\theta)d\theta
	\end{equation*}
	where we define the indicator function
	\begin{equation*}
		\mbf{I}_{\alpha}(s,\theta):=\prod_{i=1}^{N}\mbf{I}_{\alpha_{i}}(s_{i},\theta_{i})
		\qquad\text{where}\qquad
		\mbf{I}_{\alpha_{i}}(s_{i},\theta_{i})
		:=\begin{cases}
			1,	& \alpha_{i}\cdot (s_{i}-\theta_{i})\leq 0,	\\
			0,	&	\tnf{else},
		\end{cases}
		\quad\forall i\in\{1:N\}.
	\end{equation*}
	Now, for each $\alpha\in\{-1,1\}^{N}$, since $\mbs{K}_{\alpha}$ is polynomial, we can define $\mbs{R}_{\alpha,i,j}\in\R[s_{i},\theta_{i}]$ such that $\mbs{K}_{\alpha}=\sum_{j=1}^{M}\prod_{i=1}^{N}\mbs{R}_{\alpha,i,j}$ for some $M\in\N$ sufficiently large (e.g. expanding $\mbs{K}_{\alpha}$ as a sum of monomials). Given these univariate polynomials, $\mbs{R}_{\alpha,i,j}$, define the 3-PI parameters $\mbf{R}_{\alpha,i,j}^{\beta,\gamma}\in\PIparam_{1}[a_{i},b_{i}]$ by
	\begin{equation*}
		\mbf{R}_{\alpha,i,j}:=
		\begin{cases}
			\{0,\mbs{R}_{\alpha,i,j},0\},	&	\alpha_{i}=-1,	\\
			\{0,0,\mbs{R}_{\alpha,i,j}\},	&	\alpha_{i}=1,
		\end{cases}
	\end{equation*}
	Then, by definition
	\begin{align*}
		\int_{\Omega} \mbf{I}_{\alpha}(s,\theta)\mbs{K}_{\alpha}(s,\theta)\mbf{v}(\theta)d\theta
		&=\int_{\Omega}\bbbbl(\sum_{j=1}^{M}\prod_{i=1}^{N}\mbf{I}_{\alpha_{i}}(s_{i},\theta_{i})\mbs{R}_{\alpha,i,j}(s_{i},\theta_{i})\bbbbr)\mbf{v}(\theta)d\theta	\\
		&=\sum_{j=1}^{M}\left[\bbbl(\prod_{i=1}^{N}\PI_{1}\bl[\mbf{R}_{\alpha,i,j}\br]\bbbr)\mbf{v}\right](s).
	\end{align*}
	Since $\prod_{i=1}^{N}\PI_{1}\bl[\mbf{R}_{\alpha,i,j}\br]\in\PIset_{N}[\Omega]$ for all $\alpha$ and $j$, it follows that also $\mcl{K}\in\PIset_{N}[\Omega]$.
\end{proof}

Lem.~\ref{lem:PI_semisep_appx} shows that the class of $N$D PI operators includes a particular class of integral operators, of which the kernel is polynomial semi-separable. By construction, the operator $\mcl{T}$ in Thm.~\ref{thm:Tmap}, defining the inverse to $D^{\delta}:\mcl{D}\to L_{2}^{n}[\Omega]$, will be of this form. However, the class of $N$D 3-PI operators also allows for mixed multiplier and integral operators along different spatial directions, e.g. $\int_{a}^{x}\mbs{R}(x,y,\theta)\mbf{v}(\eta,y)d\eta$ and $\int_{y}^{d}\mbs{R}(x,y,\eta)\mbf{v}(x,\zeta)d\zeta$. This ensures that the operator $\mcl{A}$ defining the PIE representation in Thm.~\ref{thm:PDE2PIE} is also an $N$D PI operator, as proven in the following lemma.

\begin{lem}\label{lem:PI_op_T_appx}
	For $N\in\N$ and $\delta\in\N_{0}^{N}$, let $\mcl{T}$ and $\mcl{A}$ be as defined in Thm.~\ref{thm:Tmap}. Then $\mcl{T},\mcl{A}\in\PIset_{N}^{n\times n}[\Omega]$.
\end{lem}
\begin{proof}
	For the case $n=1$, the result follows by inspection. In particular, by definition, $\mcl{T}=\prod_{i=1}^{N}\mcl{T}_{i}$ and $\mcl{A}=\sum_{\vec{0}\leq\alpha\leq\delta}\tnf{M}[\mbs{A}_{\alpha}]\prod_{i=1}^{N}\mcl{A}_{i,\alpha_{i}}$, where the operators $\mcl{T}_{i}$ and $\mcl{A}_{i,\alpha_{i}}$ are of the form of the operators $\mcl{T}$ and $\mcl{A}$ in Cor.~\ref{cor:gohberg_PIE}, respectively. By definition, then, we find that $\mcl{T}_{i}$ and $\mcl{A}_{i,\alpha_{i}}$ are univariate 3-PI operators, wherefore $\mcl{T}$ is a multivariate 3-PI operator. Furthermore, since $\tnf{M}[\mbs{A}_{\alpha}]\in\PIset_{N}$ for all $\mbs{A}_{\alpha}\in\R[s]$ and $\PIset_{N}$ is closed under compositions (by Prop.~\ref{prop:PI_composition_ND}), it follows that also $\mcl{A}$ is a multivariate 3-PI operator. 
	
	For the case $n>1$, we note that for each $k,\ell\in\{1:n\}$, we can express 
	\begin{align*}
		[\mcl{T}]_{k,\ell}
		=\left[\prod_{i=1}^{N}\mcl{T}_{i}\right]_{k,\ell}
		=\sum_{j_{1}=1}^{n}\sum_{j_{2}=1}^{n}\cdots \sum_{j_{N-1}=1}^{n}[\mcl{T}_{1}]_{k,j_{1}}[\mcl{T}_{2}]_{j_{1},j_{2}}\cdots[\mcl{T}_{N}]_{j_{N-1},\ell}.
	\end{align*}
	Since, as established for the case $n=1$, each $[\mcl{T}_{i}]_{j_{i-1},j_{i}}$ is a univariate 3-PI operator, it follows that each $[\mcl{T}]_{k,\ell}$ is a multivariate 3-PI operator, and therefore $\mcl{T}\in\PIset_{N}^{n\times n}[\Omega]$. By similar reasoning, it follows that also $\mcl{A}\in\PIset_{N}^{n\times n}[\Omega]$.
\end{proof}

By Lem.~\ref{lem:PI_op_T}, given a multivariate PDE as in~\eqref{eq:PDE_standard}, the operators $\mcl{T}$ and $\mcl{A}$ defining the associated PIE representation as per Thm.~\ref{thm:Tmap} are in fact $N$D 3-PI operators. Having thus shown that the class $\PIset_{N}[\Omega]$ is suitable for parameterization of $N$D PIEs, we now prove that this class in fact defines a $*$-algebra. To start, we show that $\PIset_{N}[\Omega]$ is a vector space, for which we have the following lemma.
\begin{lem}\label{lem:PI_sum_ND_appx}
	For $N\in\N$ and $\Omega:=\prod_{i=1}^{N}$, if $\mcl{Q},\mcl{R}\in\PIset_{N}[\Omega]$, then $\lambda\mcl{Q}+\mu\mcl{R}\in\PIset_{N}[\Omega]$ for all $\lambda,\mu\in\R$.
\end{lem}
\begin{proof}
	Fix arbitrary $\mcl{Q}=\sum_{j=1}^{M}\PI_{N}[\mbf{Q}_{j}]\in\PIset_{N}[\Omega]$, where for each $j\in\{1:M\}$, we have $\mbf{Q}_{j}:=\{\mbf{Q}_{j,i}\}_{i=1}^{N}$ for $\mbf{Q}_{j,i}\in\PIparam_{1}[a_{i},b_{i}]$. Suppose that $\mbf{Q}_{j,1}=\{\mbf{Q}_{j,1}^{0},\mbf{Q}_{j,1}^{1},\mbf{Q}_{j,1}^{2}\}\in\PIparam_{1}[a_{1},b_{1}]$, so that $\mbf{Q}_{j,1}^{0}\in\R[s_{1}]$ and $\mbf{Q}_{j,1}^{1},\mbf{Q}_{j,1}^{2}\in\R[s_{1},\theta_{1}]$. For $\lambda\in\R$, define $\lambda\mbf{Q}_{j,1}:=\{\lambda\mbf{Q}_{j,1}^{0},\lambda\mbf{Q}_{j,1}^{1},\lambda\mbf{Q}_{j,1}^{2}\}\in\PIparam_{1}[a_{1},b_{1}]$. Then, by definition of univariate 3-PI operators (Defn.~\ref{defn:PI_1D}) and linearity of multiplier and integral operators, $\lambda\PI_{1}[\mbf{Q}_{j,1}]\mbf{v}=\PI_{1}[\lambda\mbf{Q}_{j,1}]\mbf{v}$ for all $\mbf{v}\in L_{2}$. Define further $\lambda\mbf{Q}_{j}:=\{\lambda\mbf{Q}_{j,1},\mbf{Q}_{j,2},\hdots,\mbf{Q}_{j,N}\}\in\PIparam_{N}[\Omega]$. Then, for each $j\in\{1:M\}$
	\begin{equation*}
		\lambda\mcl{Q}=\lambda\sum_{j=1}^{M}\PI_{N}[\mbf{Q}_{j}]
		=\sum_{j=1}^{M}\lambda\prod_{i=1}^{N}\PI_{1}[\mbf{Q}_{j,i}]
		=\sum_{j=1}^{M}\PI_{1}[\lambda\mbf{Q}_{j,1}] \PI_{1}[\mbf{Q}_{j,2}]\cdots \PI_{1}[\mbf{Q}_{j,N}]
		=\sum_{j=1}^{M}\PI_{N}[\lambda\mbf{Q}_{j}]
		\in\PIset_{N}[\Omega].
	\end{equation*}
	Finally, fix arbitrary $\mcl{R}=\sum_{k=1}^{K}\PI_{N}[\mbf{R}_{k}]\in\PIset_{N}[\Omega]$. Then, for all $\lambda,\mu\in\R$,
	\begin{equation*}
		\lambda\mcl{Q}+\mu\mcl{R}
		=\sum_{j=1}^{M}\PI_{N}[\lambda\mbf{Q}_{j}] +\sum_{k=1}^{K}\PI_{N}[\mu\mbf{R}_{k}]
		\in\PIset_{N}[\Omega].
	\end{equation*}
\end{proof}

Lem.~\ref{lem:PI_sum_ND_appx} shows that the class of multivariate 3-PI operators is closed under linear combinations, and is thus a vector space. Next, we show that this vector space is also closed under compositions. Here, the composition of two multivariate 3-PI operators can be defined inductively, using the composition rules for univariate 3-PI operators. In particular, if $\mcl{Q}=\prod_{i=1}^{N}\PI_{1}[\mbs{Q}_{i}]$ and $\mcl{R}=\prod_{i=1}^{N}\PI_{1}[\mbs{R}_{i}]$ for $\mbs{Q}_{i},\mbs{R}_{i}\in\PIparam_{1}[a_{i},b_{i}]$, then $\mcl{Q}\circ\mcl{R}=\prod_{i=1}^{N}(\PI_{1}[\mbs{Q}_{i}]\circ\PI_{1}[\mbs{R}_{i}])$. In order for this to hold, however, the different operators $\PI_{1}[\mbs{Q}_{i}]$ and $\PI_{1}[\mbs{R}_{j}]$ for $i\neq j$ must commute. The following lemma proves that this is indeed the case.

\begin{lem}\label{lem:PI_prod_commute}
	Let $N\in\N$ and $\Omega=\prod_{i=1}^{N}[a_{i},b_{i}]$. For $\mbf{R}\in\PIparam_{1}[a_{i},b_{i}]$ and $\mbf{Q}\in\PIparam_{1}[a_{j},b_{j}]$ let $\mcl{R}$ and $\mcl{Q}$ be the multivariate extensions of $\PI_{1}[\mbf{R}]$ and $\PI_{1}[\mbf{Q}]$, respectively, so that e.g. $(\mcl{R}\mbf{v})(s):=(\PI_{1}\mbf{v}(s_{1},\hdots,s_{i-1},\bullet,s_{i+1},\hdots,s_{N}))(s_{i})$ for $\mbf{v}\in L_{2}[\Omega]$. If $i\neq j$, then $(\mcl{R}\mcl{Q})\mbf{v}=(\mcl{Q}\mcl{R})\mbf{v}$ for all $\mbf{v}\in L_{2}[\Omega]$.
\end{lem}
\begin{proof}
	Let $\mbf{R}=\{\mbs{R}_{0},\mbs{R}_{1},\mbs{R}_{2}\}\in\PIparam[a_{i},b_{i}]$ and $\mbf{Q}=\{\mbs{Q}_{0},\mbs{Q}_{1},\mbs{Q}_{2}\}\in\PIparam[a_{j},b_{j}]$.
	To prove the result, we first note that by Defn.~\ref{defn:PI_1D}, for any $\mbf{v}\in L_{2}[\Omega]$, we can express $\mcl{R}\mbf{v}=\mcl{R}_{0}\mbf{v}+\mcl{R}_{1}\mbf{v}+\mcl{R}_{2}\mbf{v}$ and $\mcl{Q}\mbf{v}=\mcl{Q}_{0}\mbf{v}+\mcl{Q}_{1}\mbf{v}+\mcl{Q}_{2}\mbf{v}$, where we define
	\begin{align*}
		&(\mcl{R}_{0}\mbf{v})(s):=\mbs{R}_{0}(s_{i})\mbf{v}(s),	&
		&(\mcl{Q}_{0}\mbf{v})(s):=\mbs{Q}_{0}(s_{j})\mbf{v}(s),	\\
		&(\mcl{R}_{1}\mbf{v})(s):=\int_{a_{i}}^{s_{i}}\mbs{R}_{1}(s_{i},\theta_{i})\mbf{v}(s)|_{s_{i}=\theta_{i}}d\theta_{i},	&
		&(\mcl{Q}_{1}\mbf{v})(s):=\int_{a_{j}}^{s_{j}}\mbs{Q}_{1}(s_{j},\theta_{j})\mbf{v}(s)|_{s_{j}=\theta_{j}}d\theta_{j},	\\
		&(\mcl{R}_{2}\mbf{v})(s):=\int_{s_{i}}^{b_{i}}\mbs{R}_{2}(s_{i},\theta_{i})\mbf{v}(s)|_{s_{i}=\theta_{i}}d\theta_{i},	&
		&(\mcl{Q}_{2}\mbf{v})(s):=\int_{s_{j}}^{b_{j}}\mbs{Q}_{2}(s_{j},\theta_{j})\mbf{v}(s)|_{s_{j}=\theta_{j}}d\theta_{j},
	\end{align*}
	for all $s=(s_{1},\hdots,s_{N})\in\Omega$. Here, we remark that
	\begin{align*}
		\bl(\mcl{R}_{0}\mcl{Q}_{0}\mbf{v}\br)(s)
		=\mbs{R}_{0}(s_{i})\mbs{Q}_{0}(s_{j})\mbf{v}(s)
		=\mbs{Q}_{0}(s_{j})\mbs{R}_{0}(s_{i})\mbf{v}(s)
		=\bl(\mcl{Q}_{0}\mcl{R}_{0}\mbf{v}\br)(s),
	\end{align*}
	as well as
	\begin{align*}
		\bl(\mcl{R}_{0}\mcl{Q}_{1}\mbf{v}\br)(s)
		&=\mbs{R}_{0}(s_{i})\int_{a_{j}}^{s_{j}}\mbs{Q}_{1}(s_{j},\theta_{j})\mbf{v}(s)|_{s_{j}=\theta_{j}}d\theta_{j}		\\
		&=\int_{a_{j}}^{s_{j}}\mbs{Q}_{1}(s_{j},\theta_{j})\mbs{R}_{0}(s_{i})\mbf{v}(s)|_{s_{j}=\theta_{j}}d\theta_{j}
		=\bl(\mcl{Q}_{1}\mcl{R}_{0}\mbf{v}\br)(s),
	\end{align*}
	and
	\begin{align*}
		\bl(\mcl{R}_{1}\mcl{Q}_{1}\mbf{v}\br)(s)
		&=\int_{a_{i}}^{s_{i}}\mbs{R}_{1}(s_{i},\theta_{i})\left[\int_{a_{j}}^{s_{j}}\mbs{Q}_{1}(s_{j},\theta_{j})\mbf{v}(s)|_{s_{i}=\theta_{i},s_{j}=\theta_{j}}d\theta_{j}\right]d\theta_{i}		\\
		&=\int_{a_{i}}^{s_{i}}\int_{a_{j}}^{s_{j}}\left[\mbs{R}_{1}(s_{i},\theta_{i})\mbs{Q}_{1}(s_{j},\theta_{j})\mbf{v}(s)|_{s_{i}=\theta_{i},s_{j}=\theta_{j}}\right]d\theta_{j} d\theta_{i}		\\
		&=\int_{a_{j}}^{s_{j}}\mbs{Q}_{1}(s_{j},\theta_{j})\left[\int_{a_{i}}^{s_{i}}\mbs{R}_{1}(s_{i},\theta_{i})\mbf{v}(s)|_{s_{i}=\theta_{i},s_{j}=\theta_{j}}d\theta_{i}\right]d\theta_{j}
		=\bl(\mcl{Q}_{1}\mcl{R}_{1}\mbf{v}\br)(s),
	\end{align*}
	By similar reasoning, we find that $\mcl{R}_{\ell}\mcl{Q}_{k}\mbf{v}=\mcl{Q}_{k}\mcl{R}_{\ell}\mbf{v}$, for all $k,\ell\in\{0,1,2\}$.
	It follows that
	\begin{align*}
		\mcl{R}\mcl{Q}\mbf{v}
		&=\bl(\mcl{R}_{0}+\mcl{R}_{1}+\mcl{R}_{2}\br)\bl(\mcl{Q}_{0}+\mcl{Q}_{1}+\mcl{Q}_{2}\br)\mbf{v}	\\
		&=\bl(\mcl{Q}_{0}+\mcl{Q}_{1}+\mcl{Q}_{2}\br)\bl(\mcl{R}_{0}+\mcl{R}_{1}+\mcl{R}_{2}\br)\mbf{v}
		=\mcl{Q}\mcl{R}\mbf{v}.
	\end{align*}
\end{proof}

Using Lem.~\ref{lem:PI_prod_commute}, the following corollary proves that the composition of two $N$D PI operators can indeed be expressed using the composition of the defining 1D PI operators, as used in the proof of Prop.~\ref{prop:PI_composition_ND}.

\begin{cor}\label{cor:PI_prod_commute}
	For given $N\in\N$ and $\Omega=\prod_{i=1}^{N}[a_{i},b_{i}]$, let $\mbf{Q}=\{\mbf{Q}_{i}\}_{i=1}^{N}\in\PIparam_{N}[\Omega]$ and $\mbf{R}=\{\mbf{R}_{i}\}_{i=1}^{N}\in\PIparam_{N}[\Omega]$, where $\mbf{Q}_{i}\in\PIparam_{1}[a_{i},b_{i}]$ and $\mbf{R}_{i}\in\PIparam_{1}[a_{i},b_{i}]$ for each $i\in\{1:N\}$. Then
	\begin{equation*}
		\left(\prod_{i=1}^{N}\PI_{1}[\mbf{Q}_{i}]\right)\circ \left(\prod_{j=1}^{N}\PI_{1}[\mbf{R}_{j}]\right)=\prod_{i=1}^{N}\bl(\PI_{1}[\mbf{Q}_{i}]\circ\PI_{1}[\mbf{R}_{i}]\br).
	\end{equation*}
\end{cor}
\begin{proof}	
	By Lem.~\ref{lem:PI_prod_commute} we have that for all $\mbf{u}\in L_{2}[\Omega]$,
	\begin{align*}
		&\left(\prod_{i=1}^{N}\PI_{1}[\mbf{Q}_{i}]\right)\left(\prod_{j=1}^{N}\PI_{1}[\mbf{R}_{j}]\right)\mbf{u}
		=\bbl(\PI_{1}[\mbf{Q}_{1}]\cdots \PI_{1}[\mbf{Q}_{N}]\,\PI_{1}[\mbf{R}_{1}]\cdots \PI_{1}[\mbf{R}_{N}]\bbr)\mbf{u}		\\
		&\hspace*{3.0cm}=\bl(\PI_{1}[\mbf{Q}_{1}]\PI_{1}[\mbf{R}_{1}]\br)\cdots\bl(\PI_{1}[\mbf{Q}_{N}] \PI_{1}[\mbf{R}_{N}]\br)\mbf{u}
		=\left(\prod_{i=1}^{N}\bl(\PI_{1}[\mbf{Q}_{i}]\PI_{1}[\mbf{R}_{i}]\br)\right)\mbf{u}.
	\end{align*}
\end{proof}

Using Cor.~\ref{cor:PI_prod_commute}, Prop.~\ref{prop:PI_composition_ND} proves that the composition of two multivariate 3-PI operators is again a multivariate 3-PI operator, so that $\PIset_{N}$ forms an algebra. Finally, to prove that $\PIset_{N}$ in fact forms a $*$-algebra, we show that this set is also closed under the adjoint operation. Specifically, we have the following result, corresponding to Lem.~\ref{lem:PI_adjoint_ND} in Subsection~\ref{subsec:PIs:algebra}

\begin{lem}\label{lem:PI_adjoint_ND_appx}
	For $N\in\N$ and $\Omega:=\prod_{i=1}^{N}[a_{i},b_{i}]\subseteq\R^{N}$, if $\mcl{R}\in\PIset_{N}[\Omega]$, then $\mcl{R}^*\in\PIset_{N}[\Omega]$. In particular, if $\mcl{R}=\sum_{j=1}^{M}\PI_{N}[\mbf{R}_{j}]$, where $\mbf{R}_{j}:=\{\mbf{R}_{j,i}\}_{i=1}^{N}\in\PIparam_{N}[\Omega]$, then $\mcl{R}^*=\sum_{j=1}^{M}\PI_{N}[\mbf{R}_{j}^{\adj}]$, where
	\begin{equation*}
		\mbf{R}_{j}^{\adj}:=\{\mbf{R}_{j,i}^{\adj}\}_{i=1}^{N},
	\end{equation*}
	where the univariate adjoint parameters $\mbf{R}_{j,i}^{\adj}\in\PIparam_{1}[a_{i},b_{i}]$ are as defined in Lem.~\ref{lem:PI_adjoint_1D}.
\end{lem}
\begin{proof}
	To prove the result we first remark that, by Lem.~\ref{lem:PI_adjoint_1D}, for each $i\in\{1:N\}$ and all $\mbf{R}\in\PIparam_{1}[a_{i},b_{i}]$ we have
	\begin{equation*}
		\int_{a_{i}}^{b_{i}}\mbf{u}(s)\,\bl(\PI_{1}[\mbf{R}]\mbf{v}\br)(s)\,ds_{i}	\\
		=\int_{a_{i}}^{b_{i}}\bl(\PI_{1}[\mbf{R}^{\adj}]\mbf{u}\br)(s)\, \mbf{v}(s)\, d s_{i},\qquad \forall \mbf{u},\mbf{v}\in L_{2}[\Omega],
	\end{equation*}
	where we define $(\PI_{1}[\mbf{R}]\mbf{v})(s):=\bl(\PI_{1}[\mbf{R}]\mbf{v}(s_{1},\hdots,s_{i-1},\bullet,s_{i+1},\hdots,s_{N})\br)(s_{i})$. Now, let $\mcl{R}=\sum_{j=1}^{M}\PI_{N}[\mbf{R}_{j}]$, where $\mbf{R}_{j}:=\{\mbf{R}_{j,i}\}_{i=1}^{N}\in\PIparam_{N}[\Omega]$.
	Then, for all $j\in\{1:M\}$,
	\begin{align*}
		&\ip{\mbf{u}}{\PI_{N}[\mbf{R}_{j}]\mbf{v}}_{L_{2}[\Omega]}
		=\int_{\Omega}\bbbl[\mbf{u}(s) \bl(\PI_{N}[\mbf{R}_{j}]\mbf{v}\br)(s)\bbbr]\,ds \\
		&=\int_{a_{1}}^{b_{1}}\cdots \int_{a_{N}}^{b_{N}}\bbbl[
		\mbf{u}(s_{1},\hdots,s_{N}) \bl(\PI_{1}[\mbf{R}_{j,1}]\cdots\PI_{1}[\mbf{R}_{j,N}]\mbf{v}\br)(s_{1},\hdots,s_{N})\bbbr] \,ds_{N}\cdots ds_{1} \\
		&=\int_{a_{1}}^{b_{1}}\cdots \int_{a_{N}}^{b_{N}}\bbbl[
		\bl(\PI_{1}[\mbf{R}_{j,1}^{\adj}]\mbf{u}\br)(s_{1},\hdots,s_{N})\, \bl(\PI_{1}[\mbf{R}_{j,2}]\cdots\PI_{1}[\mbf{R}_{j,N}]\mbf{v}\br)(s_{1},\hdots,s_{N})\bbbr]\,ds_{N}\cdots ds_{1} \\
		&=\int_{a_{1}}^{b_{1}}\cdots \int_{a_{N}}^{b_{N}}\bbbl[
		\bl(\PI_{1}[\mbf{R}_{j,2}^{\adj}]\PI_{1}[\mbf{R}_{j,1}^{\adj}]\mbf{u}\br)(s_{1},\hdots,s_{N})\, \bl(\PI_{1}[\mbf{R}_{j,3}]\cdots\PI_{1}[\mbf{R}_{j,N}]\mbf{v}\br)(s_{1},\hdots,s_{N})\bbbr]\,ds_{N}\cdots ds_{1} \\
		&~\vdots\\
		&=\int_{a_{1}}^{b_{1}}\cdots \int_{a_{N}}^{b_{N}}\bbbl[
		\br(\PI_{1}[\mbf{R}_{j,N}^{\adj}]\cdots\PI_{1}[\mbf{R}_{j,1}^{\adj}]\mbf{u}\br)(s_{1},\hdots,s_{N})\,  \mbf{v}(s_{1},\hdots,s_{N})\bbbr]\,ds_{N}\cdots ds_{1}	\\
		&=\int_{a_{1}}^{b_{1}}\cdots \int_{a_{N}}^{b_{N}}\bbbl[
		\br(\PI_{1}[\mbf{R}_{j,1}^{\adj}]\cdots\PI_{1}[\mbf{R}_{j,N}^{\adj}]\mbf{u}\br)(s_{1},\hdots,s_{N})\,  \mbf{v}(s_{1},\hdots,s_{N})\bbbr]\,ds_{N}\cdots ds_{1}	\\
		&=\int_{\Omega}\bbbl(\bl(\PI_{N}[\mbf{R}_{j}^{\adj}]\mbf{u}\br)(s)\mbf{v}(s)\bbbr)\,ds	\\
		&=\ip{\PI_{N}[\mbf{R}_{j}^{\adj}]\mbf{u}}{\mbf{v}}_{L_{2}[\Omega]},
	\end{align*}
	where we remark that $\PI_{1}[\mbf{R}_{j,N}^{\adj}]\cdots\PI_{1}[\mbf{R}_{j,1}^{\adj}]=\PI_{1}[\mbf{R}_{j,1}^{\adj}]\cdots\PI_{1}[\mbf{R}_{j,N}^{\adj}]$ by Lem.~\ref{lem:PI_prod_commute}.
	By linearity of the inner product, it follows that if $\mcl{R}:=\sum_{i=1}^{M}\PI_{N}[\mbf{R}_{j}]$, then
	\begin{equation*}
		\ip{\mbf{u}}{\mcl{R}\mbf{v}}_{L_{2}}
		=\sum_{i=1}^{M}\ip{\mbf{u}}{\PI_{N}[\mbf{R}_{j}]\mbf{v}}_{L_{2}}
		=\sum_{i=1}^{M}\ip{\PI_{N}[\mbf{R}_{j}^{\adj}]\mbf{u}}{\mbf{v}}_{L_{2}}
		=\ip{\mcl{R}^*\mbf{u}}{\mbf{v}}_{L_{2}}.
	\end{equation*}
\end{proof}


\section{Exponential PIE to PDE Stability of Classical PDE Examples}\label{appx:Examples}

In this appendix, we verify exponential stability of the reaction-diffusion equation from Subsec.~\ref{subsec:Examples:examples:heat} and the wave equation from Subsec.~\ref{subsec:Examples:examples:wave}. For each example, an explicit solution is constructed using separation of variables, and it is proven that the norm of this solution is bounded by an exponentially decaying function, providing an explicit value of the decay rate. Since the reaction-diffusion and wave equations are classical examples of PDEs for which existence and uniqueness of solutions have been well-studied in the literature, several details in the derivation of the solutions will be omitted, and it will not be shown that the obtained solutions are indeed unique. We refer to standard textbooks on PDEs such as~\cite{evans2022PDEs} for more details and proofs. 

\subsection{Exponential PIE to PDE Stability of the Reaction-Diffusion Equation}\label{appx:Examples:heat}

Consider the following reaction-diffusion equation, with Dirichlet-Neumann boundary conditions
\begin{align}\label{eq:reaction_diffusion_appx}
	\mbf{u}_{t}(t,x,y)&=\nu[\mbf{u}_{xx}(t,x,y)+\mbf{u}_{yy}(t,x,y)]+r\mbf{u}(t,x,y),	&	(x,y)&\in[0,L_{x}]\times[0,L_{y}],	\\
	\mbf{u}(t,0,y)&=\mbf{u}(t,L_{x},y)=0,\qquad \mbf{u}(t,x,0)=\mbf{u}_{y}(t,x,L_{y})=0.		\notag
\end{align}
We solve this PDE using separation of variables. In particular, suppose $\mbf{u}(t,x,y)=T(t)U(x,y)$ for some function $T$ that depends only on time, and a function $U$ that varies only in space. Substituting the relation $\mbf{u}(t,x,y)=T(t)U(x,y)$ into the PDE, we find that it must satisfy
\begin{equation*}
	T'(t)U(x,y)=\nu T(t)[\partial_{x}^2 U(x,y)+\partial_{y}^2 U(x,y)] +rT(t)U(x,y).
\end{equation*}
Dividing both sides by $T(t)U(x,y)$, and rearranging the terms, this yields
\begin{equation*}
	\frac{T'(t)}{T(t)}-r=\nu\frac{\partial_{x}^2 U(x,y)}{U(x,y)} +\nu\frac{\partial_{y}^2 U(x,y)}{U(x,y)},
\end{equation*}
where now the left-hand side depends only on $t$, and the right-hand side depends only on $x,y$. It follows that both sides must be constant in time and space, whence solutions must satisfy
\begin{equation*}
	\frac{T'(t)}{T(t)}-r=-\lambda\qquad \text{and}\qquad
	\nu\frac{\partial_{x}^2 U(x,y)}{U(x,y)} +\nu\frac{\partial_{y}^2 U(x,y)}{U(x,y)}=-\lambda,
\end{equation*}
for some constant $\lambda\in\R$. Solving the first of these equations, we obtain
\begin{equation*}
	T(t)=T(0)e^{[r-\lambda]t}.
\end{equation*}
For the second equation, suppose further that $U(x,y)=X(x)Y(y)$ for some functions $X:[0,L_{x}]\to\R$ and $Y:[0,L_{y}]\to\R$. Substituting this identity into the equation, it follows that these functions must satisfy
\begin{equation*}
	\nu\frac{X''(x)}{X(x)} =-\nu\frac{Y''(y)}{Y(y)}-\lambda.
\end{equation*}
Here, the left-hand side depends only on $x$, and the right-hand side depends only on $y$, implying that both sides must be constant in $x$ and $y$. Thus, the functions $X$ and $Y$ must satisfy
\begin{equation*}
	\nu\frac{X''(x)}{X(x)}=-\mu,\qquad \text{and}\qquad
	-\nu\frac{Y''(y)}{Y(y)}-\lambda=-\mu,
\end{equation*}
for some constant $\mu\in\R$. To solve these problems, recall that the function $u(t,x,y)=T(t)X(x)Y(y)$ must further satisfy the boundary conditions
\begin{align*}
	T(t)X(0)Y(y)&=u(t,0,y)=0,	&	T(t)X(L_{x})Y(y)&=u(t,L_{x},y)=0,	\\
	T(t)X(x)Y(0)&=u(t,x,0)=0,	&	T(t)X(x)Y'(L_{y})&=u_{y}(t,x,L_{y})=0.
\end{align*}
Excluding trivial solutions, these boundary conditions can only be satisfied if
\begin{equation*}
	X(0)=X(L_{x})=0,\qquad\text{and}\qquad Y(0)=Y'(L_{y})=0.
\end{equation*}
Thus, we obtain two boundary-value problems
\begin{align*}
	X''(x)&=-\frac{\mu}{\nu} X(x),	&	&\text{and}	&
	Y''(y)&=-\frac{\lambda-\mu}{\nu}Y(y),		\\
	X(0)&=X(L_{x})=0,		&	&	&
	Y(0)&=Y'(L_{y})=0.
\end{align*}
These are both standard 1D Sturm-Liouville problems, of which solutions are well-known.
In particular, we note that non-trivial solutions exist only for $\mu=\mu_{m}:=\nu\frac{m^2\pi^2}{L_{x}^2}$ and $\lambda-\mu=\lambda_{m,n}-\mu_{m}:=\nu\frac{(n-1/2)^2\pi^2}{L_{y}^2}$ for $m,n\in\N$, taking the form
\begin{equation*}
	X_{m}(x)=\sin\left(\frac{m\pi}{L_{x}}x\right),\qquad 
	Y_{n}(y)=\sin\left(\frac{(n-1/2)\pi}{L_{y}}y\right),\qquad \forall m,n\in\N.
\end{equation*}
Here, the functions $\mbf{u}_{m,n}(x,y):=X_{m}(x)Y_{n}(y)$ form a basis for functions on $L_{2}[[0,L_{x}]\times[0,L_{y}]]$ satisfying the imposed Dirichlet and Neumann boundary conditions, and any solution to the original PDE takes the form
\begin{equation*}
	\mbf{u}(t,x,y)=\sum_{m=1}^{\infty}\sum_{n=1}^{\infty}a_{m,n}e^{[r-\lambda_{m,n}]t}\mbf{u}_{m,n}(x,y)
	=\sum_{m=1}^{\infty}\sum_{n=1}^{\infty}a_{m,n}e^{[r-\lambda_{m,n}]t} \sin\left(\frac{m\pi}{L_{x}}x\right) \sin\left(\frac{(n-1/2)\pi}{L_{y}}y\right),
\end{equation*}
for some coefficients $a_{m,n}$ defined by the initial conditions, and where
\begin{equation*}
	\lambda_{m,n} =\mu_{m}+\nu\frac{(n-1/2)^2\pi^2}{L_{y}^2} =\nu\pi^2\left[\frac{m^2}{L_{x}^2}+\frac{(n-1/2)^2}{L_{y}^2}\right],\qquad \forall m,n\in\N.
\end{equation*}
It follows that
\begin{equation*}
	\mbf{u}_{xxyy}(t,x,y)
	=\sum_{m=1}^{\infty}\sum_{n=1}^{\infty}\frac{m^2\pi^2}{L_{x}^2}\frac{(n-1/2)^2\pi^2}{L_{y}^2}a_{m,n}e^{[r-\lambda_{m,n}]t} \sin\left(\frac{m\pi}{L_{x}}x\right) \sin\left(\frac{(n-1/2)\pi}{L_{y}}y\right).
\end{equation*}
By orthogonality of the basis functions $\mbf{u}_{m,n}$ with respect to the $L_{2}$ inner product, we find
\begin{align*}
	\norm{\mbf{u}(t)}_{L_{2}}^2&=\sum_{m=1}^{\infty}\sum_{n=1}^{\infty}\frac{L_{x}L_{y}}{4}a_{m,n}^2e^{2[r-\lambda_{m,n}]t} 	\\
	&\leq \sum_{m=1}^{\infty}\sum_{n=1}^{\infty}\left(m^2\pi^2(n-1/2)^2\pi^2\right)^2 \frac{L_{x}L_{y}}{4}a_{m,n}^2e^{2[r-\lambda_{\min}]t} 
	=L_{x}^2 L_{y}^2 e^{2[r-\lambda_{\min}]t}\norm{\mbf{u}_{xxyy}(0)}_{L_{2}}^2,
\end{align*}
where $\lambda_{\min}:= \min_{m,n\in\N}\lambda_{m,n}=\nu\pi^2\left[\frac{1}{L_{x}^2}+\frac{1}{4L_{y}^2}\right]$.
Thus, the PDE is exponentially PIE to PDE stable whenever $r\leq \lambda_{\min}$, with rate of decay $k=\lambda_{\min}-r$.

\subsection{Exponential PIE to PDE Stability of the Wave Equation}\label{appx:Examples:wave}

Consider the following wave equation with feedback, with Dirichlet-Neumann boundary conditions
\begin{align}\label{eq:wave_eq_appx}
	\mbf{u}_{tt}(t,x,y)&=\mbf{u}_{xx}(t,x,y)+\mbf{u}_{yy}(t,x,y)-2\kappa\mbf{u}_{t}(t,x,y)-\kappa^2\mbf{u}(t,x,y),	&	(x,y)&\in[0,L_{x}]\times[0,L_{y}],	\\
	\mbf{u}(t,0,y)&=\mbf{u}_{x}(t,L_{x},y)=0,\qquad \mbf{u}(t,x,0)=\mbf{u}_{y}(t,x,L_{y})=0,		\notag
\end{align}
where $\kappa\geq 0$.
Suppose solutions take the form $\mbf{u}(t,x,y)=T(t)U(x,y)$ for some function $T$ that depends only on time, and a function $U$ that varies only in space. Substituting the relation $\mbf{u}(t,x,y)=T(t)U(x,y)$ into the PDE, we find that $T$ and $U$ must satisfy
\begin{equation*}
	T''(t)U(x,y)= T(t)[\partial_{x}^2 U(x,y)+\partial_{y}^2 U(x,y)] -2\kappa T'(t)U(x,y) -\kappa^2T(t)U(x,y).
\end{equation*}
Dividing both sides by $T(t)U(x,y)$, and re-arranging the terms, this yields
\begin{equation*}
	\frac{T''(t)+2\kappa T'(t)}{T(t)}+\kappa^2=\frac{\partial_{x}^2 U(x,y)}{U(x,y)} +\frac{\partial_{y}^2 U(x,y)}{U(x,y)},
\end{equation*}
where now the left-hand side depends only on $t$, and the right-hand side depends only on $x,y$. It follows that both sides must be constant in time and space, whence solutions must satisfy
\begin{equation*}
	\frac{T''(t)+2\kappa T'(t)}{T(t)}+\kappa^2=-\lambda,\qquad \text{and}\qquad
	\frac{\partial_{x}^2 U(x,y)}{U(x,y)} +\frac{\partial_{y}^2 U(x,y)}{U(x,y)}=-\lambda,
\end{equation*}
for some constant $\lambda\in\R$. Solving the first of these equations, we find that if $\lambda>0$, solutions take the form
\begin{equation*}
	T(t)=\gamma e^{-\kappa t}\sin\left(t\sqrt{\lambda}\right)+\delta e^{-\kappa t}\cos\left(t\sqrt{\lambda}\right),
\end{equation*}
for some constants $\gamma,\delta$ depending on $T(0)$ and $T'(0)$. 
For the second equation, suppose further that $U(x,y)=X(x)Y(y)$ for some functions $X:[0,L_{x}]\to\R$ and $Y:[0,L_{y}]\to\R$. Substituting this identity into the equation, it follows that these functions must satisfy
\begin{equation*}
	\frac{X''(x)}{X(x)} =-\frac{Y''(y)}{Y(y)}-\lambda.
\end{equation*}
Here, the left-hand side depends only on $x$, and the right-hand side depends only on $y$, implying that both sides must be constant in $x$ and $y$. Thus, the functions $X$ and $Y$ must satisfy
\begin{equation*}
	\frac{X''(x)}{X(x)}=-\mu,\qquad \text{and}\qquad
	-\frac{Y''(y)}{Y(y)}-\lambda=-\mu,
\end{equation*}
for some constant $\mu\in\R$. To solve these problems, recall that the function $u(t,x,y)=T(t)X(x)Y(y)$ must further satisfy the boundary conditions
\begin{align*}
	T(t)X(0)Y(y)&=u(t,0,y)=0,	&	T(t)X'(L_{x})Y(y)&=u_{x}(t,L_{x},y)=0,	\\
	T(t)X(x)Y(0)&=u(t,x,0)=0,	&	T(t)X(x)Y'(L_{y})&=u_{y}(t,x,L_{y})=0.
\end{align*}
Excluding trivial solutions, these boundary conditions can only be satisfied if
\begin{equation*}
	X(0)=X'(L_{x})=0,\qquad\text{and}\qquad Y(0)=Y'(L_{y})=0.
\end{equation*}
Thus, we obtain the boundary-value problems
\begin{align*}
	X''(x)&=-\mu X(x),	&	&\text{and}	&
	Y''(y)&=-(\lambda-\mu)Y(y),		\\
	X(0)&=X'(L_{x})=0,		&	&	&
	Y(0)&=Y'(L_{y})=0.
\end{align*}
These are again standard Sturm-Liouville problems, which we can solve to find that a solution exists only if $\mu=\mu_{m}=(m-\frac{1}{2})^2\frac{\pi^2}{L_{x}^2}$ and $\lambda_{m,n}-\mu_{m}=(n-\frac{1}{2})^2\frac{\pi^2}{L_{y}^2}$ for $m,n\in\N$, with solutions taking the form
\begin{equation*}
	X_{m}(x)=\alpha\sin(\sqrt{\mu_{m}} x),\qquad
	Y_{n}(y)=\beta\sin(\sqrt{\lambda_{m,n}-\mu_{m}} y),\qquad \forall m,n\in\N,
\end{equation*}
for arbitrary $\alpha,\beta\in\R$. Thus, non-trivial solutions to the PDE exist only for
\begin{equation*}
	\lambda_{m,n}:=\mu_{m}+\left(n-\frac{1}{2}\right)^2\frac{\pi^2}{L_{y}^2}
	=\pi^2\left[\frac{(m-1/2)^2}{L_{x}^2}+\frac{(n-1/2)^2}{L_{y}^2}\right],\qquad \forall m,n\in\N.
\end{equation*}
Note that the functions $\mbf{u}_{m,n}(x,y):=X_{m}(x)Y_{n}(y)$ for $m,n\in\N$ define a basis for the space of functions satisfying the imposed Dirichlet and Neumann boundary conditions. Reintroducing the dependence on time, $T(t)$, we find that solutions to the wave equation take the form
\begin{align*}
	\mbf{u}(t,x,y)&=\sum_{m,n=1}^{\infty}a_{m,n}e^{-\kappa t}\sin\left(\sqrt{\lambda_{m,n}}t\right)\sin\bl(\sqrt{\mu_{m}} x\br)\sin\left(\sqrt{\lambda_{m,n}-\mu_{m}} y\right)	\\
	&\qquad  +\sum_{m,n=1}^{\infty}b_{m,n}e^{-\kappa t}\cos\left(\sqrt{\lambda_{m,n}}t\right)\sin\bl(\sqrt{\mu_{m}} x\br)\sin\left(\sqrt{\lambda_{m,n}-\mu_{m}} y\right),
\end{align*}
for coefficients $a_{m,n},b_{m,n}$ determined by $\mbf{u}(0,x,y)$ and $\mbf{u}_{t}(0,x,y)$. Given these coefficients, the temporal derivative of the solution is then given by 
\begin{align*}
	\mbf{u}_{t}(t,x,y)&=-\kappa\mbf{u}(t) +\sum_{m,n=1}^{\infty}a_{m,n}\sqrt{\lambda_{m,n}} e^{-\kappa t}\cos\left(\sqrt{\lambda_{m,n}}t\right)\sin\bl(\sqrt{\mu_{m}} x\br)\sin\left(\sqrt{\lambda_{m,n}-\mu_{m}} y\right)	\\
	&\qquad  -\sum_{m,n=1}^{\infty}b_{m,n}\sqrt{\lambda_{m,n}}e^{-\kappa t}\sin\left(\sqrt{\lambda_{m,n}}t\right)\sin\bl(\sqrt{\mu_{m}} x\br)\sin\left(\sqrt{\lambda_{m,n}-\mu_{m}} y\right)	\\
	&=-\sum_{m,n=1}^{\infty}\bbbl[\kappa a_{m,n}+b_{m,n}\sqrt{\lambda_{m,n}}\bbbr]e^{-\kappa t}\sin\left(\sqrt{\lambda_{m,n}}t\right)\sin\bl(\sqrt{\mu_{m}} x\br)\sin\left(\sqrt{\lambda_{m,n}-\mu_{m}} y\right)	\\
	&\qquad  -\sum_{m,n=1}^{\infty}\bbbl[\kappa b_{m,n}-a_{m,n}\sqrt{\lambda_{m,n}}\bbbr]e^{-\kappa t}\cos\left(\sqrt{\lambda_{m,n}}t\right)\sin\bl(\sqrt{\mu_{m}} x\br)\sin\left(\sqrt{\lambda_{m,n}-\mu_{m}} y\right).
\end{align*}
By orthogonality of the spatial basis functions with respect to the $L_{2}$-inner product, it follows that
\begin{align*}
	\norm{\mbf{u}(t)}_{L_{2}}^2
	&=\frac{L_{x}L_{y}}{4}e^{-2\kappa t}\sum_{m,n=1}^{\infty}\left[a_{m,n}\sin\left(\sqrt{\lambda_{m,n}}t\right) +b_{m,n}\cos\left(\sqrt{\lambda_{m,n}}t\right)\right]^2,	\\
	\norm{\mbf{u}_{t}(t)}_{L_{2}}^2
	&=  \frac{L_{x}L_{y}}{4}e^{-2\kappa t}\sum_{m,n=1}^{\infty}\left[\bbbl(\kappa a_{m,n}+b_{m,n}\sqrt{\lambda_{m,n}}\bbbr)\sin\left(\sqrt{\lambda_{m,n}}t\right) +\bbbl(\kappa b_{m,n}-a_{m,n}\sqrt{\lambda_{m,n}}\bbbr)\cos\left(\sqrt{\lambda_{m,n}}t\right)\right]^2.	
\end{align*}
Given these expressions, we note that solutions do not satisfy $\norm{\bmat{\mbf{u}(t)\\\mbf{u}_{t}(t)}}_{L_{2}}\leq M e^{-k t}\norm{\bmat{\mbf{u}(0)\\\mbf{u}_{t}(0)}}_{L_{2}}$ for any $M\geq 1$ and $k\geq 0$. Indeed, to illustrate, consider $L_{x}=L_{y}=1$, with the initial condition $\mbf{u}(0,x,y)=\sin((j-\frac{1}{2})\pi x)\sin((j-\frac{1}{2})\pi y)$ for some $j\in\N$. Then $a_{m,n}=0$ for all $m,n\in\N$, $b_{j,j}=1$, and $b_{m,n}=0$ for all other $m,n\in\N$. It follows that, the norm of the solution at each time $t$ is given by
\begin{align*}
	\norm{\mbf{u}(t)}_{L_{2}}^{2}
	&=\frac{L_{x}L_{y}}{4}e^{-2\kappa t}\cos^2\bl(\sqrt{2}(j-1/2)\pi t\br),	\\
	\norm{\mbf{u}_{t}(t)}_{L_{2}}^{2}
	&=\frac{L_{x}L_{y}}{4}e^{-2\kappa t}\left[ \sqrt{2}(j-1/2)\pi\sin\bl(\sqrt{2}(j-1/2)\pi t\br) +\kappa \cos\bl(\sqrt{2}(j-1/2)\pi t\br)\right]^2.
\end{align*}
Therefore, at $t=0$, we have
\begin{align*}
	\norm{\bmat{\mbf{u}(0)\\\mbf{u}_{t}(0)}}_{L_{2}}^{2}
	=\norm{\mbf{u}(0)}_{L_{2}}^{2} +\norm{\mbf{u}_{t}(0)}_{L_{2}}^{2}
	=\frac{L_{x}L_{y}}{4} +\frac{L_{x}L_{y}}{4}\kappa^2
	=\frac{L_{x}L_{y}}{4}(1+\kappa^2),
\end{align*}
whereas at $t=T_{j}:=\frac{1}{2\sqrt{2}(j-1/2)}$, we have
\begin{align*}
	\norm{\bmat{\mbf{u}(T_{j})\\\mbf{u}_{t}(T_{j})}}_{L_{2}}^{2}
	=\norm{\mbf{u}(T_{j})}_{L_{2}}^{2} +\norm{\mbf{u}_{t}(T_{j})}_{L_{2}}^{2}
	=\frac{L_{x}L_{y}}{2}(j-1/2)^2\pi^2 e^{-2\kappa T_{j}}.
\end{align*}
It follows that, for any $M\geq 1$, there exists $j\in\N$ such that $\norm{\bmat{\mbf{u}(T_{j})\\\mbf{u}_{t}(T_{j})}}_{L_{2}}^{2}> M\norm{\bmat{\mbf{u}(0)\\\mbf{u}_{t}(0)}}_{L_{2}}^{2}$, and thus the wave equation cannot be exponentially stable in the classical sense. However, we can prove that the PDE is exponentially PIE to PDE stable, in the sense of Defn.~\ref{defn:exponential_stability}. In particular, note that
\begin{align*}
	\mbf{u}_{xxyy}(0,x,y)&=\sum_{m,n=1}^{\infty}b_{m,n}\mu_{m}(\lambda_{m,n}-\mu_{m})\sin\bl(\sqrt{\mu_{m}} x\br)\sin\left(\sqrt{\lambda_{m,n}-\mu_{m}} y\right),		\\
	\mbf{u}_{txxyy}(0,x,y)&=-\sum_{m,n=1}^{\infty}\bbbl[\kappa b_{m,n}-a_{m,n}\sqrt{\lambda_{m,n}}\bbbr]\mu_{m}(\lambda_{m,n}-\mu_{m})\sin\bl(\sqrt{\mu_{m}} x\br)\sin\left(\sqrt{\lambda_{m,n}-\mu_{m}} y\right).
\end{align*}
By orthogonality of the spatial basis functions, it follows then that
\begin{align*}
	\norm{\bmat{\mbf{u}_{xxyy}(0)\\\mbf{u}_{txxyy}(0)}}_{L_{2}}^{2}
	=\frac{L_{x}L_{y}}{4}\sum_{m,n=1}^{\infty}\left[b_{m,n}^2 +\bbbl(\kappa b_{m,n}-a_{m,n}\sqrt{\lambda_{m,n}}\bbbr)^2\right]\mu_{m}^2(\lambda_{m,n}-\mu_{m})^2.
\end{align*}
Thus, we can bound the norm of the solution to the wave equation at each time as
\begin{align*}
	\norm{\bmat{\mbf{u}(t)\\\mbf{u}_{t}(t)}}_{L_{2}}^{2}
	&=\frac{L_{x}L_{y}}{4}e^{-2\kappa t}\sum_{m,n=1}^{\infty}\left[a_{m,n}\sin\left(\sqrt{\lambda_{m,n}}t\right) +b_{m,n}\cos\left(\sqrt{\lambda_{m,n}}t\right)\right]^2	\\
	&\hspace*{-1.0cm} +\frac{L_{x}L_{y}}{4}e^{-2\kappa t}\sum_{m,n=1}^{\infty}\left[\bbbl(\kappa a_{m,n}+b_{m,n}\sqrt{\lambda_{m,n}}\bbbr)\sin\left(\sqrt{\lambda_{m,n}}t\right) +\bbbl(\kappa b_{m,n}-a_{m,n}\sqrt{\lambda_{m,n}}\bbbr)\cos\left(\sqrt{\lambda_{m,n}}t\right)\right]^2		\\
	&\leq \frac{L_{x}L_{y}}{4}e^{-2\kappa t}\sum_{m,n=1}^{\infty}\left[a_{m,n}^2 +b_{m,n}^2 +\bbl(\kappa a_{m,n}+b_{m,n}\sqrt{\lambda_{m,n}}\bbr)^2 +\bbl(\kappa b_{m,n}-a_{m,n}\sqrt{\lambda_{m,n}}\bbr)^2\right]			\\
	&\leq e^{-2\kappa t}\norm{\bmat{\mbf{u}_{xxyy}(0)\\\mbf{u}_{txxyy}(0)}}_{L_{2}}^{2} +\frac{L_{x}L_{y}}{4}e^{-2\kappa t}\sum_{m,n=1}^{\infty}\left[a_{m,n}^2 +\bbl(\kappa a_{m,n}+b_{m,n}\sqrt{\lambda_{m,n}}\bbr)^2\right].
\end{align*}
Here, we note that for any $L_{x},L_{y}$, we can find a constant $C$ such that $\lambda_{m,n}\leq C\mu_{m}^2(\lambda_{m,n}-\mu_{m})^2$ for all $m,n\in\N$. It follows that
\begin{align*}
	\sum_{m,n=1}^{\infty}\left[a_{m,n}^2 +\bbl(\kappa a_{m,n}+b_{m,n}\sqrt{\lambda_{m,n}}\bbr)^2\right]
	&\leq \sum_{m,n=1}^{\infty}\left[a_{m,n}^2 +\bl(\kappa^2+\lambda_{m,n}\br)\bl(a_{m,n}^2+b_{m,n}^2\br)\right]	\\
	&\leq \sum_{m,n=1}^{\infty}\left[a_{m,n}^2 +\lambda_{m,n}\bl(\kappa^2+1\br)\bl(a_{m,n}^2+b_{m,n}^2\br)\right]	\\
	&\leq \sum_{m,n=1}^{\infty}2\lambda_{m,n}(1+\kappa^2)\bl(a_{m,n}^2+b_{m,n}^2\br)	\\
	&\leq 2C(1+\kappa^2)^2\sum_{m,n=1}^{\infty}\left[b_{m,n}^2 +\bbl(\kappa b_{m,n}-a_{m,n}\sqrt{\lambda_{m,n}}\bbr)^2\right]\mu_{m}^2(\lambda_{m,n}-\mu_{m})^2 \\
	&=\frac{8C(1+\kappa^2)^2}{L_{x}L_{y}}\norm{\bmat{\mbf{u}_{xxyy}(0)\\\mbf{u}_{txxyy}(0)}}_{L_{2}}^{2}.
\end{align*}
Defining, then, $M:=\sqrt{1+8C(1+\kappa^2)^2}$, it follows that
\begin{align*}
	\norm{\bmat{\mbf{u}(t)\\\mbf{u}_{t}(t)}}_{L_{2}}
	\leq M e^{-\kappa t} \norm{\bmat{\mbf{u}_{xxyy}(0)\\\mbf{u}_{txxyy}(0)}}_{L_{2}}.
\end{align*}

\end{document}